\documentclass[11pt,reqno]{amsart}

\usepackage{BS_Macros}
\usepackage{BS_Diagrams}

\title[Invariant Gibbs measures for the quintic NLS in infinite volume]{Invariant Gibbs measures for the one-dimensional quintic nonlinear Schrödinger equation in infinite volume}

\author{Bjoern Bringmann}
\address{\noindent Bjoern Bringmann, Department of Mathematics, Princeton University, Princeton, NJ 08544}
\email{bringmann@princeton.edu}

 \author{Gigliola Staffilani}
\address{\noindent\begin{minipage}[t]{\textwidth}
Gigliola Staffilani, 
Department of Mathematics, 
Massachusetts Institute of Technology, Cambridge, MA 02139
\end{minipage}}
\email{gigliola@math.mit.edu}

\date{\today}

\begin{document}

\begin{abstract}
We prove the invariance of the Gibbs measure for the defocusing quintic nonlinear Schr\"odinger equation on the real line. This builds on earlier work by Bourgain, who treated the  cubic nonlinearity. The key new ingredient is a growth estimate for the infinite-volume \mbox{$\Phi^{p+1}_1$-measures}, which is proven via the stochastic quantization method. 
\end{abstract}

\maketitle

\tableofcontents

\section{Introduction}

\noindent 
Over the last three decades, there has been tremendous interest in the construction and dynamics of Gibbs measures for defocusing nonlinear Schrödinger equations, which can be written as 
\begin{equation}\label{intro:eq-NLS-general}
i \partial_t u + \Delta u = |u|^{p-1}u. 
\end{equation}
In \eqref{intro:eq-NLS-general}, $p>1$ is a parameter, which is often chosen as an odd integer. The construction and dynamics of Gibbs measures for \eqref{intro:eq-NLS-general} differ in the periodic and infinite-volume setting, which correspond to the spatial domains $\T^d:=(\R / (2\pi \Z))^d$ and $\R^d$, respectively. The construction of Gibbs measures for \eqref{intro:eq-NLS-general}, which are called $\Phi^{p+1}_d$-measures, is nowadays understood in both the periodic and infinite-volume setting. The $\Phi^4_3$-measure, which is the most prominent member of this family, was first constructed in the periodic setting by Glimm and Jaffe \cite{GJ73} and later in the infinite-volume setting by Feldman and Osterwalder \cite{FO76}. For a more detailed discussion of $\Phi^{p+1}_d$-measures, we refer the reader to the introduction of \cite{GH21} and the references therein.

The dynamics of \eqref{intro:eq-NLS-general} with initial data drawn from the Gibbs measure were first studied in seminal works of Bourgain. In the periodic setting, Bourgain \cite{B94,B96} proved the invariance of the Gibbs measure under \eqref{intro:eq-NLS-general} for $d=1$ and $(d,p)=(2,3)$. The higher-order nonlinearities $p\geq 5$ in dimension $d=2$ were treated in more recent work of Deng, Nahmod, and Yue \cite{DNY24}. For a more detailed overview of invariant Gibbs measures for nonlinear dispersive equations in the periodic setting, we refer the reader to the introduction of \cite{BDNY24} and the references therein. In the infinite-volume setting, the dynamics of \eqref{intro:eq-NLS-general} with initial data drawn from the Gibbs measure are much less understood. The reason is that the initial data drawn from the Gibbs measure has no decay in space and, in fact, exhibits logarithmic growth (see Theorem~\ref{intro:thm-measure} below). Due to infinite speed of propagation, the growth of the initial data makes it very challenging to control the dynamics of \eqref{intro:eq-NLS-general}. So far, the only\footnote{To be more precise, this is the only result on the almost-sure convergence of $2\pi L$-periodic solutions of \eqref{intro:eq-NLS-general} with initial data drawn from Gibbs measures as $L\rightarrow \infty$. For results on the convergence in law of $2\pi L$-periodic solutions, which can be obtained using compactness arguments, see Subsection \ref{section:introduction-comments}.} available result has been obtained by Bourgain \cite{B00}, who proved the invariance of the Gibbs measure for $(d,p)=(1,3)$. The goal of this article is to extend Bourgain's result to the case $(d,p)=(1,5)$, i.e., to treat the quintic (rather than cubic) nonlinearity.

\subsection{Main results}\label{section:introduction-results}

For the rest of the article, we focus on the one-dimensional, defocusing nonlinear Schrödinger equations
\begin{equation}\label{intro:eq-NLS}
\begin{cases}
\begin{aligned}
i \partial_t u + \Delta u &= |u|^{p-1} u \qquad \qquad (t,x) \in \R \times \R, \\
u(0)&=\phi.
\end{aligned}
\end{cases}
\end{equation}
To treat the infinite-volume setting, we first consider $2\pi L$-periodic initial data, where $L\in\dyadic$, and then take the limit as $L\rightarrow\infty$. To make this more precise, we let $\T_L:= \R /(2\pi L \Z)$. We then introduce the $2\pi L$-periodic massive Gaussian free field $\cg_L$ and Gibbs measure $\mu_L$, which can be rigorously defined as 
\begin{equation}\label{intro:eq-GFF-rigorous}
\cg_L = \operatorname{Law}\Big( \sum_{n\in \Z_L} \frac{g_n}{\langle n \rangle} e^{inx} \Big) 
\end{equation}
where $\Z_L:= L^{-1} \Z$ and $(g_n)_{n\in \Z_L}$ are independent, standard complex-valued Gaussians, and 
\begin{equation}\label{intro:eq-Gibbs-rigorous}
\frac{\mathrm{d}\mu_L}{\mathrm{d}\cg_L}(\phi)
= \mathcal{Z}_L^{-1} \exp\Big( -  \tfrac{1}{p+1} \int_{\T_L} |\phi|^{p+1} \dx \Big). 
\end{equation}
The Gibbs measures $\mu_L$ have a unique weak limit $\mu$ as $L\rightarrow \infty$, which is simply called the infinite-volume limit. For details regarding this weak limit, see Lemma \ref{measure:lem-Wasserstein} below. 

\begin{theorem}[Dynamics]\label{intro:thm-dynamics}
Let $3 \leq p\leq 5$ and let $(\Omega,\mathcal{F},\bP)$ be a probability space. Furthermore, let $\phi_L,\phi \colon (\Omega \times \R, \mathcal{F}\times \mathcal{B}(\R))\rightarrow \C$, where $L\in \dyadic$ and $\mathcal{B}(\R)$ is the Borel $\sigma$-Algebra, be random continuous functions satisfying the following properties:
\begin{enumerate}[label=(\roman*)]
    \item\label{intro:item-distribution} (Distribution) For all $L\in \dyadic$, we have that $\operatorname{Law}_\bP(\phi_L)=\mu_L$ and $\Law_\bP(\phi)=\mu$.
    \item\label{intro:item-coupling} (Coupling) There exist constants $C\geq 1$ and  $\eta>0$ such that, for all $L\in \dyadic$,
    \begin{equation}\label{intro:eq-coupling}
    \bP \Big( \big\| \phi - \phi_L \big\|_{C^0_x([-L^\eta,L^\eta])} > L^{-\eta} \Big) \leq C L^{-\eta}.
    \end{equation}
\end{enumerate}
Finally, for all $L\in \dyadic$, let $u_L$ be the unique global solution of \eqref{intro:eq-NLS} with initial data $u_L(0)=\phi_L$. Then, the sequence $u_L$ has a $\bP$-a.s. limit $u$ in $C_t^0 C_x^\alpha([-T,T]\times I)$ for all $0\leq \alpha<\frac{1}{2}$, all $T\geq 1$, and all compact intervals $I\subseteq \R$. In fact, there exist constants $C^\prime \geq 1$ and $\eta^\prime>0$, depending only on $\alpha,C,\eta$, and $T$, such that the estimate 
\begin{equation}\label{intro:eq-as-convergence-estimate}
\bP \Big( \big\| u - u_L \big\|_{C_t^0 C_x^\alpha([-T,T]\times [-L^{\eta^\prime},L^{\eta^\prime}])}
> L^{-\eta^\prime} \Big) \leq C L^{-\eta^\prime}
\end{equation}
is satisfied for all $L\in \dyadic$. 
Furthermore, $u$ solves \eqref{intro:eq-NLS-general} in the sense of space-time distributions and leaves the Gibbs measure~$\mu$ invariant, i.e., it holds that $\Law_\bP(u(t))=\mu$ for all $t\in \R$.
\end{theorem}

 As already mentioned above, Theorem \ref{intro:thm-dynamics} is the extension of the main result of \cite{B00} from\footnote{In \cite{B00}, the condition is stated as $p\leq 4$, but this is due to a difference in notation. In \eqref{intro:eq-NLS}, the exponent of the nonlinearity is denoted by $p$, whereas in \cite[(0.1)]{B00}, the exponent is denoted by $p-1$.} $p=3$ to $3\leq p\leq 5$. The exponents $1<p<3$ in Theorem \ref{intro:thm-dynamics} were excluded to avoid technical difficulties related to the regularity of the function $z\mapsto |z|^{p-1}z$, but can be treated using minor modifications of the arguments below. 

 In addition to extending the result of \cite{B00}, we also simplify a technical aspect of the argument in \cite{B00}. To be more specific, we do not use any estimates of the kernel of $P_{\leq N} e^{it\Delta}$, where $P_{\leq N}$ is a Littlewood-Paley operator. For more details, see the proof of Proposition \ref{uniform:prop-hoelder} and Remark~\ref{uniform:rem-kernel-estimates}.

\begin{remark}[Couplings]
We note that, unlike in \cite{B00}, both the assumption \eqref{intro:eq-coupling} and the conclusion \eqref{intro:eq-as-convergence-estimate} in Theorem \ref{intro:thm-dynamics} are quantitative. The assumption \eqref{intro:eq-coupling} can be satisfied by choosing the random initial data as a suitable coupling of the Gibbs measures (see Proposition \ref{measure:prop-coupling}).
This coupling is constructed using the stochastic quantization method and a quantitative version of the Skorokhod representation theorem. We emphasize that the Skorokhod representation theorem is only used in our construction of a coupling satisfying the properties in \ref{intro:item-distribution} and \ref{intro:item-coupling}, but is not used in the proof of Theorem \ref{intro:thm-dynamics} itself. An even stronger coupling of the Gibbs measures than in \eqref{intro:eq-coupling} was previously constructed using different methods than in this article in \cite[Proposition 2.7]{FKV24}. 

We also note that if the quantitative assumption \eqref{intro:eq-coupling} is replaced by the qualitative assumption that the limit of $\phi_L$ exists $\bP$-a.s. in $C_x^0(I)$ for all compact $I\subseteq \R$, then a minor modification of our argument yields the $\bP$-a.s. convergence of a subsequence $u_{L_j}$ of $u_L$ in $C_t^0 C_x^\alpha([-T,T]\times I)$ for all $0\leq \alpha <\frac{1}{2}$, all $T\geq 1$, and all compact intervals $I\subseteq \R$. For more details, see Remark \ref{proof:rem-quantitative-assumption} below. 
\end{remark}

 We now briefly describe the main idea behind the proof of Theorem \ref{intro:thm-dynamics}. As in~\cite{B00}, we consider the difference $w:= u_L - u_{L/2}$ and control the local, frequency-truncated mass
\begin{equation}\label{intro:eq-MR}
M_R(w(t)) := \int_\R \big| P_{\leq R} w \big|^2 e^{-\langle \frac{x}{R} \rangle} \dx,
\end{equation}
where $R\geq 1$. With high probability, the derivative of $M_R(w(t))$ for $t\in [-1,1]$ can be bounded by
\begin{equation}\label{intro:eq-MR-derivative}
\frac{\mathrm{d}}{\mathrm{d}t} M_R(w(t)) 
\lesssim \Big( \max_{u=u_{L/2},u_L} \| u \|_{L_t^\infty L_x^\infty([-1,1]\times [-R^2,R^2])}^{p-1} \Big) \, M_R(w(t)) + R^{-1+\varepsilon},
\end{equation}
where $\varepsilon>0$ is a small parameter. For the details behind \eqref{intro:eq-MR-derivative}, we refer the reader to Proposition~\ref{difference:prop-main} and its proof. In \cite{B00}, Bourgain used an estimate of Brascamp-Lieb (see Lemma \ref{measure:lem-brascamp-lieb}) to control the $2\pi L$-periodic Gibbs measures using the $2\pi L$-periodic Gaussian free fields. Using standard estimates for the Gaussian free fields, one then obtains that 
\begin{equation}\label{intro:eq-Linfty-GFF}
\max_{u=u_{L/2},u_L} \| u \|_{L_t^\infty L_x^\infty([-1,1]\times [-R^2,R^2])}\lesssim \log(R)^{\frac{1}{2}}
\end{equation}
with high probability. By combining \eqref{intro:eq-MR-derivative}, \eqref{intro:eq-Linfty-GFF}, and Gronwall's inequality, it then follows that
\begin{equation}\label{intro:eq-MR-bound}
M_R(w(t)) \lesssim \exp\Big( C \log(R)^{\frac{p-1}{2}} |t| \Big) \big( M_R(w(0)) + R^{-1+\varepsilon}\big).
\end{equation}
Provided that $R$ is much smaller than a small power of $L$, it follows from our assumption in Theorem~\ref{intro:thm-dynamics}.\ref{intro:item-coupling} that $M_R(w(0))$ is small. In order to keep $M_R(w(t))$ small over a time-interval $[-\tau,\tau]$, where $\tau>0$ is a small constant, one then needs that $(p-1)/2\leq 1$, i.e., $p\leq 3$. To improve the condition on $p$, we rely on the fact that the tails of the $\Phi^{p+1}_1$-measures decay faster than the tail of the Gaussian free field. Instead of~\eqref{intro:eq-Linfty-GFF}, this allows us to prove that 
\begin{equation}\label{intro:eq-Linfty-Gibbs}
\max_{u=u_{L/2},u_L} \| u \|_{L_t^\infty L_x^\infty([-1,1]\times [-R^2,R^2])}\lesssim \log(R)^{\frac{2}{p+3}}
\end{equation}
with high probability. By combining \eqref{intro:eq-MR-derivative}, \eqref{intro:eq-Linfty-Gibbs}, and Gronwall's inequality, we then obtain the improved estimate 
\begin{equation}\label{intro:eq-MR-bound-improved}
M_R(w(t)) \lesssim \exp\Big( C \log(R)^{\frac{2(p-1)}{p+3}} |t| \Big) \big( M_R(w(0)) + R^{-1+\varepsilon}\big).
\end{equation}
To keep $M_R(w(t))$ small, one then only needs that $2(p-1)/(p+3)\leq 1$, i.e., $p\leq 5$.\\

As discussed above, growth estimates for the $\Phi^{p+1}_1$-measures play an essential role in the proof of Theorem \ref{intro:thm-dynamics}, and they are the subject of our next theorem.  

\begin{theorem}[Measures]\label{intro:thm-measure}
Let $p>1$ and let $C=C_p$ and $c=c_p$ be sufficiently large and small constants depending only on $p$, respectively. For all $R\geq 10$ and $\lambda \geq 1$, it then holds that 
\begin{equation}\label{measure:eq-Linfty}
\sup_{L\geq 10} \mu_L \Big( \Big\{ \big\| \phi \big\|_{L^\infty([-R,R])} \geq C \big( \log(R) + \lambda \big)^{\frac{2}{p+3}} \Big\} \Big) \leq C \exp\big( - c \lambda \big). 
\end{equation}
\end{theorem}

Since  \eqref{measure:eq-Linfty} concerns a Gibbs measure in only one spatial dimension, it can be obtained using classical methods based on SDEs/Feynman-Kac formulas (see e.g. \cite[Section 2]{B02}). For a recent proof of \eqref{measure:eq-Linfty} using such methods (for $p=3$), we refer the reader to \cite[Proposition 2.2]{FKV24}. In this article, we take a different approach towards \eqref{measure:eq-Linfty}, and instead prove it using stochastic quantization. In particular, we rely on an elegant argument of Hairer and Steele \cite{HS22}, which was originally developed for the $\Phi^4_3$-measure. Our motivation for this is that, in addition to proving Theorem~\ref{intro:thm-measure}, we would like to illustrate the Hairer-Steele method in a setting which is technically much simpler than the $\Phi^4_3$-model.

\begin{remark} In the proof of Theorem \ref{intro:thm-dynamics}, it is essential that the left-hand side of \eqref{measure:eq-Linfty} contains $\log(R)+\lambda$ rather than $\lambda \log(R)$. Due to this, the choice $\lambda \sim \log(R)$ gains us powers of $R$ on the right-hand side of \eqref{measure:eq-Linfty} without increasing the power of $\log(R)$ on the left-hand side of  \eqref{measure:eq-Linfty}. This gain of powers of $R$ on the right-hand side of \eqref{measure:eq-Linfty} will later allow us to sum probabilities over different dyadic scales.
\end{remark}

\subsection{Further comments}\label{section:introduction-comments}

We conclude this introduction with several additional comments. First, we mention that invariant measures have recently been constructed for many completely integrable nonlinear dispersive equations on the real line. In the breakthrough article \cite{KMV20}, Killip, Murphy, and Visan proved the invariance of white noise under the KdV equation on the real line. More recently, Forlano, Killip, and Visan \cite{FKV24} proved the invariance of the Gibbs measures under the mKdV equation on the real line. Using the Miura transformation, the authors also constructed new invariant measures for the KdV equation. We note that, in addition to several novel ingredients, the two articles \cite{FKV24,KMV20} rely on the method of commuting flows from \cite{KV19}. Due to this, the proofs in \cite{FKV24,KMV20} can bypass\footnote{To be more precise, Gronwall estimates are used in \cite[Section 5]{FKV24} and \cite[Section 6]{KMV20} to control the $\mathcal{H}_\kappa$-flows. However, since the $\mathcal{H}_\kappa$-flows and the KdV/mKdV flows in \cite{FKV24,KMV20} commute, the Gronwall estimates are not needed for the KdV/mKdV flows themselves.} Gronwall-estimates such as \eqref{intro:eq-MR-derivative}-\eqref{intro:eq-MR-bound-improved}.

Second, we note that invariant Gibbs measures of \eqref{intro:eq-NLS-general} in infinite volume have also been studied using weak methods in \cite{BL22,CdS20}. The weak methods can be applied to a larger class of nonlinear dispersive equations but, unlike Theorem \ref{intro:thm-dynamics}, only lead to the convergence in law (rather than almost-sure convergence) of a subsequence of the periodic solutions.

Third, we mention that \eqref{intro:eq-NLS} has also been studied for slowly-decaying and non-decaying deterministic initial data, see e.g. \cite{CHKP20,DSS20,DSS21,H23,S22} and the references therein. In particular, the local well-posedness of \eqref{intro:eq-NLS} has been shown in the modulation space $M^{s}_{\infty,1}$ for $s\geq 0$, which contains non-decaying initial data \cite{CHKP20}. However, to the best of our knowledge, there is no deterministic result for \eqref{intro:eq-NLS} with initial data exhibiting logarithmic growth (as in Theorem \ref{intro:thm-dynamics}).\\

\noindent 
\textbf{Acknowledgements:} The authors thank Van Duong Dinh, Tom Spencer, and Nikolay Tzvetkov for helpful comments and discussions. During parts of this work, B.B. was supported by the NSF under Grant No. DMS-1926686 and G.S. was supported by the NSF under Grant No. DMS-2306378 and by the Simons Foundation Collaboration
Grant on Wave  Turbulence. G.S. would also like to thank the Department of Mathematics at Princeton University for the generous hospitality during the completion of this work via a Minerva Fellowship.

\section{Preliminaries}\label{section:preliminaries}

In this section, we recall basic definitions and estimates from harmonic analysis (Subsection \ref{section:prelim-harmonic}) and probability theory (Subsection \ref{section:prelim-prob}). We encourage the expert reader to skip to Section~\ref{section:measure} and to return to this section periodically whenever its estimates are needed.

\subsection{Harmonic Analysis}\label{section:prelim-harmonic}

For any interval $I\subseteq \R$, parameter $1\leq p <\infty$, and any $\phi \colon I \rightarrow \C$, we define the $L^p$-norm and $L^\infty$-norm by 
\begin{equation}\label{prelim:eq-Lp}
\big\| \phi \big\|_{L^p(I)}:=  \Big( \int_I |\phi(x)|^p \dx \Big)^{\frac{1}{p}} \qquad \text{and} \qquad \big\| \phi \big\|_{L^\infty(I)}:= \operatorname*{ess\, sup}_{x\in I} |\phi(x)|,
\end{equation}
We also define a local variant of the $L^p$-norm, which is defined as 
\begin{equation}\label{prelim:eq-Lploc}
\| \phi \|_{L^p_{\loc}(I)} := \sup_{x_0 \in \R} \| \phi \|_{L^{p}(I \, \scalebox{0.7}{$\medcap$} \, [x_0-1,x_0+1])}.
\end{equation}
For continuous functions $\phi\colon I \rightarrow \C$, we also write $\| \phi\|_{C^0(I)}$ instead of $\| \phi\|_{L^\infty(I)}$. Furthermore, for any $\alpha\in (0,1]$, we define the Hölder-norm
\begin{equation*}
\| \phi \|_{C^\alpha(I)}:= \| \phi \|_{C^0(I)}+\max_{\substack{x,y\in I\colon \\ 0<|x-y|\leq 1}} \frac{|\phi(x)-\phi(y)|}{|x-y|^\alpha}.
\end{equation*}
For a Schwartz function $\phi \colon \R \rightarrow \C$, we define its Fourier transform by 
\begin{equation*}
\widehat{\phi}(\xi) = \frac{1}{\sqrt{2\pi}} \int_\R \phi(x) e^{-i\xi x} \dx. 
\end{equation*}
We let $\rho\colon \R \rightarrow [0,1]$ be a smooth function satisfying $\rho(\xi)=1$ for all $\xi\in[-1,1]$ and $\rho(\xi)=0$ for all $\xi \not \in [-\frac{9}{8},\frac{9}{8}]$. We define $(\rho_{\leq N})_{N\in \dyadic}$ and $(\rho_N)_{N\in \dyadic}$ by 
\begin{equation}\label{prelim:eq-rho}
\rho_{\leq N}(\xi) := \rho\big( \tfrac{\xi}{N} \big), \quad 
\rho_1(\xi):=\rho_{\leq 1}(\xi), \quad \text{and} \quad 
\rho_N(\xi):= \rho_{\leq N}(\xi)-\rho_{\leq N/2}(\xi) ~\text{ for all } N\geq 2.
\end{equation}
Finally, we define the Littlewood-Paley operators $(P_{\leq N})_{N\in\dyadic}$ and $(P_N)_{N\in \dyadic}$ by 
\begin{equation}\label{prelim:eq-rho-kernel}
P_{\leq N} \phi = \widecheck{\rho}_{\leq N} \ast \phi \qquad \text{and} \qquad P_N \phi = \widecheck{\rho}_N \ast \phi,
\end{equation}
where $\widecheck{\rho}_{\leq N}$ and $\widecheck{\rho}_{N}$ are the inverse Fourier-transforms of $\rho_{\leq N}$ and $\rho_{N}$ and $\ast$ denotes the convolution. 
Equipped with the local $L^p$-norms and Littlewood-Paley operators, we can now state and prove several basic estimates from harmonic analysis. 

\begin{lemma}[Local Bernstein-estimate]\label{prelim:lem-bernstein}
Let $1\leq p \leq  q \leq \infty$,  $R\geq 10$, and $N\geq 1$. Then, it holds for all $\phi\colon \R \rightarrow \C$ and all $D\geq 1$ that 
\begin{equation}\label{prelim:eq-bernstein}
\| P_{\leq N} \phi \|_{L^q_\loc([-R,R])} \lesssim_{D,p,q} N^{\frac{1}{p}-\frac{1}{q}} \| P_{\leq N} \phi \|_{L^p_\loc([-2R,2R])} + (RN)^{-D} \| \langle x \rangle^{-D} \phi\|_{L^1(\R)}.
\end{equation}
\end{lemma}

The estimate \eqref{prelim:eq-bernstein} is a simple consequence of the fact that the kernels of the Littlewood-Paley operators $P_{\leq N}$ are morally supported on intervals of size $\sim N^{-1}$. For the sake of completeness, we still sketch the proof of \eqref{prelim:eq-bernstein}. 

\begin{proof}[Proof of Lemma \ref{prelim:lem-bernstein}:] 
Due to the definition in \eqref{prelim:eq-Lploc}, it suffices to prove for all $x_0 \in [-R,R]$ that 
\begin{equation}\label{prelim:eq-bernstein-1}
\| P_{\leq N} \phi \|_{L^q([x_0-1,x_0+1])} \lesssim_{p,q} N^{\frac{1}{p}-\frac{1}{q}} \| P_{\leq N} \phi \|_{L^p_\loc([-2R,2R])} + (RN)^{-D} \| \langle x \rangle^{-D} \phi\|_{L^1(\R)}.
\end{equation}
To this end, we let $\widetilde{P}_{\leq N}:=P_{\leq 4N}$ be a fattened Littlewood-Paley projection and let $\chi \colon \R \rightarrow [-1,1]$ be a smooth cut-off function satisfying $\operatorname{supp}(\chi)\subseteq [-1,1]$ and $\sum_{y_0 \in \Z} \chi( \cdot - y_0 )=1$. We then estimate
\begin{equation*}
\| P_{\leq N} \phi \|_{L^q([x_0-1,x_0+1])}
\leq \sum_{y_0\in \Z} \big\| \widetilde{P}_{\leq N} \big( \chi(\cdot-y_0) P_{\leq N}\phi \big)\big\|_{L^q([x_0-1,x_0+1])}.
\end{equation*}
For $y_0\in \Z$ satisfying $|x_0-y_0|\leq 4$, it follows from the standard Bernstein-estimate that  
\begin{align*}
&\,\big\| \widetilde{P}_{\leq N} \big( \chi(\cdot-y_0) P_{\leq N}\phi \big)\big\|_{L^q([x_0-1,x_0+1])}
\leq \big\| \widetilde{P}_{\leq N} \big( \chi(\cdot-y_0) P_{\leq N}\phi \big)\big\|_{L^q(\R)} \\ 
\lesssim&\,  N^{\frac{1}{p}-\frac{1}{q}}
\big\| \chi(\cdot-y_0) P_{\leq N}\phi \big\|_{L^p(\R)} 
\leq N^{\frac{1}{p}-\frac{1}{q}}
\big\|  P_{\leq N}\phi \big\|_{L^p([y_0-1,y_0+1])},  
\end{align*}
which is bounded by the first term in \eqref{prelim:eq-bernstein-1}. For $y_0 \in \Z$ satisfying $|x_0-y_0|>4$, it follows from standard mismatch estimates~(see e.g. \cite[Lemma 5.10]{DLM19}) that
\begin{equation*}
\big\| \widetilde{P}_{\leq N} \big( \chi(\cdot-y_0) P_{\leq N}\phi \big)\big\|_{L^q([x_0-1,x_0+1])}
\lesssim (N \langle y_0 - x_0 \rangle)^{-2D} \big\|  \chi(\cdot-y_0) P_{\leq N}\phi \big\|_{L^1(\R)}.
\end{equation*}
The sum of the contributions for $|y_0-x_0|\ll R$ can be estimated by $N^{-2D} \| P_{\leq N} \phi \|_{L^1_\loc([-2R,2R])}$, which can also be bounded by the first term in \eqref{prelim:eq-bernstein-1}.  The sum of the contribution for $|y_0-x_0|\gtrsim R$ can be estimated by $N^{-2D} R^{-D} \| \langle x \rangle^{-D} \phi \|_{L^1(\R)}$, which can be bounded by the second term in \eqref{prelim:eq-bernstein-1}.
\end{proof}

In the next lemma, we state a variant of Lemma \ref{prelim:lem-bernstein} involving Hölder-norms. 

\begin{lemma}\label{prelim:lem-Hoelder-LWP}
Let $\alpha \in [0,1)$, let $D\geq 1$, let $R\geq 10$, and let $N\geq 1$. Furthermore, let $\phi\colon \R \rightarrow \C$. Then, it holds that
\begin{align}
\big\| P_{\leq N} \phi \big\|_{C_x^\alpha([-R,R])}
&\lesssim_{\alpha,D} N^\alpha \big\| \phi \big\|_{C_x^0([-2R,2R])} + (RN)^{-D} \big\| \langle x \rangle^{-D} \phi\big \|_{L_x^1(\R)},  \label{prelim:eq-Hoelder-LWP-1}\\
N^\alpha \big\| P_N \phi \big\|_{C_x^0([-R,R])}
&\lesssim_{\alpha,D}  \big\| \phi \big\|_{C_x^\alpha([-2R,2R])} + (RN)^{-D} \big\| \langle x \rangle^{-D} \phi\big \|_{L_x^1(\R)},  \label{prelim:eq-Hoelder-LWP-2} \\
N^{\alpha-1} \big\| \partial_x P_{\leq N} \phi \big\|_{C_x^0([-R,R])}
&\lesssim_{\alpha,D}  \big\| \phi \big\|_{C_x^\alpha([-2R,2R])} + (RN)^{-D} \big\| \langle x \rangle^{-D} \phi\big \|_{L_x^1(\R)}.  \label{prelim:eq-Hoelder-LWP-3}
\end{align}
\end{lemma}

We remark that \eqref{prelim:eq-Hoelder-LWP-1} also holds with $P_{\leq N}$ replaced by $P_N$, which can be obtained by using the identity $P_N=P_{\leq N}-P_{\leq N/2}$ and the triangle inequality.

\begin{proof}
The three estimates \eqref{prelim:eq-Hoelder-LWP-1}, \eqref{prelim:eq-Hoelder-LWP-2}, and \eqref{prelim:eq-Hoelder-LWP-3} follow easily from the three identities 
\begin{align*}
P_{\leq N} \phi(x) - P_{\leq N} \phi(y)  &= \int_{\R} \widecheck{\rho}_{\leq N}(z) \big( \phi(x-z) - \phi(y-z) \big) \dz, \\ 
N^\alpha P_N \phi(x) &= \int_{\R} \widecheck{\rho}_N(z)  N^\alpha \big( \phi(x-z) - \phi(x) \big) \dz, \qquad \text{where } N >1, \\
N^{\alpha-1} \partial_x  P_{\leq N}\phi(x)
&= \int_\R \big( N^{-1} \partial_x \widecheck{\rho}_{\leq N}(z)\big) N^\alpha \big( \phi(x-z)-\phi(x) \big) \dz,  
\end{align*}
and we therefore omit the details.
\end{proof}

We also record a weighted bound for $P_{\leq N}$, which will be derived from \eqref{prelim:cor-PN-sup} with $\alpha=0$. 

\begin{corollary}\label{prelim:cor-PN-sup}
Let $\gamma\geq 0$ and let $D\geq 1$. For all $N\geq 1$, $R\geq 10$, and  $\phi\colon \R \rightarrow \C$, we then obtain
\begin{equation}\label{prelim:eq-PN-sup}
\big\| \log(R+\langle x\rangle)^{-\gamma} P_{\leq N}\phi \big\|_{L_x^\infty(\R)}
\lesssim_{\gamma,D} \big\| \log(R+\langle x\rangle)^{-\gamma} \phi \big\|_{L_x^\infty(\R)} 
+ (RN)^{-D} \big\| \langle x \rangle^{-D} \phi\big \|_{L_x^1(\R)}.
\end{equation}
\end{corollary}
\begin{proof}
Using a dyadic decomposition, we obtain that 
\begin{equation*}
\big\| \log(R+\langle x\rangle)^{-\gamma} P_{\leq N}\phi \big\|_{L_x^\infty(\R)} 
\sim  \sup_{k\geq 0} \log\big(2^k R\big)^{-\gamma} \big\| P_{\leq N} \phi \big\|_{L_x^\infty([-2^k R,2^k R])}.
\end{equation*}
After using \eqref{prelim:eq-Hoelder-LWP-1} with $\alpha=0$, we readily obtain the desired estimate \eqref{prelim:eq-PN-sup}.
\end{proof}

We now show that the $L^\infty$-norm of a frequency-localized function can be bounded using the maximum over a grid. 

\begin{lemma}\label{prelim:lem-local-constancy}
Let $D\geq 10$ and let $C=C_D$ be a sufficiently large constant depending only on $D$. Let $R\geq 10$, let $N\geq 1$, and let $\Lambda_{R,N}\subseteq [-R,R]$ be a grid with step size $\leq (RN)^{-2D}$. Then, it holds for all $\phi\colon \R \rightarrow \C$ that
\begin{equation}
\big\| P_N \phi \big\|_{L_x^\infty([-R,R])} \leq \max_{x\in \Lambda_{R,N}} \big| P_N \phi(x)\big| + C (RN)^{-D} \big\| \langle x \rangle^{-D} \phi \big\|_{L^1(\R)}.
\end{equation}
\end{lemma}

\begin{proof}
For each $x\in [-R,R]$, there exists an $y\in \Lambda_{R,N}$ such that $|x-y|\leq (RN)^{-2D}$. Together with the fundamental theorem of calculus, it then follows that 
\begin{equation*}
\big\| P_N \phi \big\|_{L_x^\infty([-R,R])} \leq \max_{x\in \Lambda_{R,N}} \big| P_N \phi(x)\big| + (RN)^{-2D} \big\|\partial_x P_N \phi(x) \big\|_{L_x^\infty([-R,R])}.
\end{equation*}
It therefore suffices to prove that 
\begin{equation*}
\big\| \partial_x P_N \phi(x) \big\|_{L_x^\infty([-R,R])}
\lesssim_D (RN)^{D} \big\| \langle x \rangle^{-D} \phi\big\|_{L_x^1(\R)},
\end{equation*}
which follows directly from the properties of the Littlewood-Paley kernels. 
\end{proof}

As previously discussed in Subsection \ref{section:introduction-results}, we will later estimate the local mass of a difference of two solutions of \eqref{intro:eq-NLS}. To prepare for this, we introduce the weight 
\begin{equation}\label{prelim:eq-sigma}
\sigma_R(x):= e^{-\scalebox{1}{$\langle$} \tfrac{x}{R} \scalebox{1}{$\rangle$}}
\end{equation}
and state and prove the following two lemmas. 

\begin{lemma}\label{prelim:lem-PR-weight}
Let $D\geq 1$, let $R\geq 10$, and let $\sigma_R$ be as in \eqref{prelim:eq-sigma}. For all $\phi\colon \R \rightarrow \C$, it then holds that 
\begin{align}
\big\| \sqrt{\sigma_R} P_{\leq R} \phi\big\|_{L_x^2(\R)}
&\lesssim_D  \big\| \sqrt{\sigma_R} \phi \big\|_{L_x^2(\R)}
+ R^{-D} \big\| \langle x \rangle^{-D} \phi\|_{L_x^1(\R)}, \label{prelim:lem-PR-weight-e1}\\
\big\| \sqrt{\sigma_R} \partial_x P_{\leq R} \phi \big\|_{L_x^2(\R)}
&\lesssim_D R \big\| \sqrt{\sigma_R} P_{\leq R} \phi\big\|_{L_x^2(\R)}
+  R^{-D} \big\| \langle x \rangle^{-D} \phi\|_{L_x^1(\R)}.\label{prelim:lem-PR-weight-e2} 
\end{align}
Furthermore, for all $10 \leq R_1 \leq \tfrac{1}{4} R_2$, it holds that 
\begin{equation}\label{prelim:lem-PR-weight-e3}
\big\| \sqrt{\sigma_{R_1}} P_{\leq R_1} \phi\big\|_{L_x^2(\R)}
\lesssim_D  \big\| \sqrt{\sigma_{R_2}} P_{\leq R_2} \phi \big\|_{L_x^2(\R)}
+ R_1^{-D} \big\| \langle x \rangle^{-D} \phi\|_{L_x^1(\R)}. 
\end{equation}
\end{lemma}
The reason behind \eqref{prelim:lem-PR-weight-e1} and \eqref{prelim:lem-PR-weight-e2} is that the kernel of $P_{\leq R}$ is morally supported on the physical scale $\sim R^{-1}$, on which the weight $\sigma_R$ is morally constant.  Thus, the action of $P_{\leq R}$ and multiplication by $\sigma_R$ almost commute. Similar estimates were used and proven in \mbox{\cite[(4.9)-(4.11)]{B00}}. For the sake of completeness, we present the short proof. 

\begin{proof} 
We first prove the second estimate \eqref{prelim:lem-PR-weight-e2}, which is the most difficult estimate in this lemma. Let $\widetilde{P}_{\leq R}:=P_{\leq 4R}$ be a fattened Littlewood-Paley operator and note that, due to \eqref{prelim:eq-rho-kernel}, its kernel is given by $\widecheck{\rho}_{\leq 4R}$. Due to the decay and smoothness properties of $\widecheck{\rho}_{\leq 4R}$, it holds that
\begin{equation}\label{prelim:eq-PR-weight-q1}
 \big| (\partial_x \widecheck{\rho}_{\leq 4R})(x-y)\big| \lesssim R^2 \langle R (x-y) \rangle^{-A},
\end{equation}
where $A$ is a sufficiently large constant depending only on $D$. Furthermore, we have the identity 
\begin{equation}\label{prelim:eq-PR-weight-q2}
\partial_x P_{\leq R} \phi = \partial_x \widetilde{P}_{\leq R} P_{\leq R} \phi = \partial_x \big( \widecheck{\rho}_{\leq 4R} \ast P_{\leq R} \phi \big) =  (\partial_x \widecheck{\rho}_{\leq 4R}) \ast P_{\leq R} \phi.
\end{equation}
Using the identity \eqref{prelim:eq-PR-weight-q2}, the square of the left-hand side of \eqref{prelim:lem-PR-weight-e2} can be estimated by 
\begin{align}
\int_{\R} \big| \partial_x P_{\leq R} \phi(x) \big|^2  \sigma_R(x) \dx 
&= \int_{\R} \bigg| \int_\R (\partial_x \widecheck{\rho}_{\leq 4R})(x-y) P_{\leq R}\phi(y) \dy\bigg|^2 \sigma_R(x) \dx \notag \\
&\lesssim \int_{\R} \bigg| \int_\R \ind_{|x-y|\leq R} (\partial_x \widecheck{\rho}_{\leq 4R})(x-y) P_{\leq R}\phi(y) \dy\bigg|^2 \sigma_R(x) \dx \label{prelim:eq-PR-weight-q3} \\
 &+ \int_{\R} \bigg| \int_\R \ind_{|x-y|> R} (\partial_x \widecheck{\rho}_{\leq 4R})(x-y) P_{\leq R}\phi(y) \dy\bigg|^2 \sigma_R(x) \dx \label{prelim:eq-PR-weight-q4}. 
\end{align}
From the definition of $\sigma_R$, it follows that $\sigma_R(x)\sim \sigma_R(y)$ for all $x,y\in \R$ satisfying $|x-y|\leq R$. Using this, we can estimate 
\begin{align*}
\eqref{prelim:eq-PR-weight-q3} 
&\lesssim \int_\R \bigg( \int_\R |(\partial_x \widecheck{\rho}_{\leq 4R})(x-y)| |P_{\leq R}\phi(y)| \sqrt{\sigma_R(y)} \dy \bigg)^2 \dx  \\
&\lesssim \int_\R \bigg( \int_\R |(\partial_x \widecheck{\rho}_{\leq 4R})(x-y)| |P_{\leq R}\phi(y)|^2  \sigma_R(y) \dy \bigg) \bigg( \int_\R |(\partial_x \widecheck{\rho}_{\leq 4R})(x-y)| \dy \bigg)  \dx \\
&\lesssim \big\| \partial_x \widecheck{\rho}_{\leq 4R} \big\|_{L_x^1(\R)}^2 \big\| \sqrt{\sigma_R} P_{\leq R} \phi \big\|_{L_y^2(\R)}^2 
\lesssim R^2 \big\| \sqrt{\sigma_R} P_{\leq R} \phi \big\|_{L_y^2(\R)}^2.
\end{align*}
In the last inequality, we also used \eqref{prelim:eq-PR-weight-q1}. Thus, \eqref{prelim:eq-PR-weight-q3} yields an acceptable contribution to \eqref{prelim:lem-PR-weight-e2}. In order to estimate \eqref{prelim:eq-PR-weight-q4}, we first further estimate the right-hand side of \eqref{prelim:eq-PR-weight-q1}. For all $x,y\in \R$ satisfying $|x-y|>R \geq 1$, it holds that 
\begin{equation*}
R^2 \langle R (x-y) \rangle^{-A} \lesssim R^{-A+2} \langle x - y\rangle^{-A} \lesssim R^{-A+2}\langle x -y \rangle^{-D} \lesssim R^{-A+2} \langle x \rangle^D \langle y \rangle^{-D}, 
\end{equation*}
where we used that $A\geq D$. Using this, it follows that 
\begin{equation*}
\eqref{prelim:eq-PR-weight-q4} 
\lesssim R^{-2A+4} \int_\R \bigg( \int_\R \langle y \rangle^{-D} |P_{\leq R} \phi(y)| \dy \bigg)^2 \langle x \rangle^{2D}\sigma_R(x) \dx 
\lesssim R^{-2A+2D+5} \big\| \langle y \rangle^{-D} P_{\leq R} \phi \big\|_{L_y^1(\R)}^2.
\end{equation*}
After using that $A$ is sufficiently large depending on $D$ and that $P_{\leq R}$ is uniformly bounded on polynomially-weighted $L^1(\R)$-spaces, this completes the proof of \eqref{prelim:lem-PR-weight-e2}. The first estimate \eqref{prelim:lem-PR-weight-e1} can be obtained using similar arguments as for \eqref{prelim:lem-PR-weight-e2}, where the main difference is that \eqref{prelim:eq-PR-weight-q2} is replaced by the simpler identity $P_{\leq R} \phi=\rho_{\leq R}\ast \phi$. The third estimate \eqref{prelim:lem-PR-weight-e3} can be derived from the first estimate \eqref{prelim:lem-PR-weight-e1}. Indeed, using $P_{\leq R_1} \phi = P_{\leq R_1} P_{\leq R_2} \phi$ and \eqref{prelim:lem-PR-weight-e1}, we obtain that 
\begin{align*}
\big\| \sqrt{\sigma_{R_1}} P_{\leq R_1} \phi\big\|_{L_x^2(\R)} &= \big\| \sqrt{\sigma_{R_1}} P_{\leq R_1} P_{\leq R_2} \phi\big\|_{L_x^2(\R)} 
\lesssim \big\| \sqrt{\sigma_{R_1}} P_{\leq R_2} \phi\big\|_{L_x^2(\R)} + R_1^{-D}  \big\| \langle x \rangle^{-D} \phi\|_{L_x^1(\R)}. 
\end{align*}
Together with $\sigma_{R_1}\leq \sigma_{R_2}$, this implies \eqref{prelim:lem-PR-weight-e3}. 
\end{proof}

\begin{lemma}[Commutator estimate]\label{prelim:lem-commutator}
Let $\delta>0$, let $R\geq 10$, and let $\sigma_R$ be as in \eqref{prelim:eq-sigma}. For all $Q\colon \R \rightarrow \C$ and $\phi \colon \R \rightarrow \C$, it then holds that
\begin{equation}\label{prelim:eq-commutator}
\big\| \sqrt{\sigma_R} \,  \big[ P_{\leq R}, Q \big] \phi \big\|_{L_x^2(\R)}
\lesssim_\delta R^{-\frac{1}{2}+2\delta} \big\| \langle x \rangle^{-\delta} \partial_x Q \big\|_{L_x^\infty(\R)} \big\| \langle x \rangle^{-\delta} \phi \big\|_{L_x^\infty(\R)}.
\end{equation}
\end{lemma}
\begin{proof}
Using the definitions of the commutator and the $L^\infty$-norm, we have that \begin{align*}
&\, \int_\R \big| \big[ P_{\leq R}, Q \big] \phi(x)\big|^2 \sigma_R(x) \dx \\
=&\, \int_\R \bigg| \int_\R \rho_{\leq R}(x-y) \big( Q(x)-Q(y)\big) \phi(y) \dy \bigg|^2 \sigma_R(x) \dx \\
\lesssim&\, \| \langle x \rangle^{-\delta} \partial_x Q \|_{L_x^\infty(\R)}^2
\| \langle x \rangle^{-\delta} \phi \|_{L_x^\infty(\R)}^2 
\int_\R \bigg(\int_\R \big| \rho_{\leq R}(x-y) \big| \, \big( \langle x \rangle^\delta + \langle y \rangle^\delta\big) \big|x-y\big|  \, \langle y \rangle^\delta dy\bigg)^2 \sigma_R(x) \dx. 
\end{align*}
From a direct calculation, it follows that
\begin{align*}
\int_\R \bigg(\int_\R \big| \rho_{\leq R}(x-y) \big| \, \big( \langle x \rangle^\delta + \langle y \rangle^\delta\big) \big|x-y\big|  \, \langle y \rangle^\delta dy\bigg)^2 \sigma_R(x) \dx
\lesssim  R^{-2} \int_\R \langle x \rangle^{4\delta} \sigma_R(x)  \lesssim  R^{-1+4\delta },
\end{align*}
which then implies the desired estimate  \eqref{prelim:eq-commutator}. 
\end{proof}

\subsection{Probability theory}\label{section:prelim-prob}

We first state a lemma that controls the tails of maxima of random variables.

\begin{lemma}[Maximum tail estimate]\label{prelim:lem-maximum-tail} 
Let $J$ be a finite index set and let $(X_j)_{j\in J}$ be random variables. Furthermore, let $A\geq 0$, $B>0$, $\beta>0$, $\gamma>0$, and $q>0$ be parameters. Finally, assume that the tail estimate
\begin{equation}\label{prelim:eq-maximum-tail-assumption}
\max_{j\in J} \bP \Big( |X_j| \geq \beta^{-\frac{1}{q}} \big( A+\lambda)^{\frac{1}{q}} \Big) \leq B e^{-\gamma\lambda}
\end{equation}
is satisfied for all $\lambda \geq 0$. Then, the maximum tail estimate
\begin{equation}\label{prelim:eq-maximum-tail-estimate}
\bP \Big( \max_{j\in J} |X_j| \geq \beta^{-\frac{1}{q}} \big( A+ \gamma^{-1}\log(|J|)+\lambda)^{\frac{1}{q}} \Big) \leq B e^{-\gamma\lambda}
\end{equation}
holds for all $\lambda >0$. 
\end{lemma}

We remark that the tail estimate in \eqref{prelim:eq-maximum-tail-assumption} with $A=0$ and $\gamma=1$ can often be obtained from Markov's inequality and the moment estimate 
\begin{equation*}
\max_{j\in J} \E \Big[ e^{\beta |X_j|^q} \Big] \leq B.
\end{equation*}
For the sake of completeness, we include the simple proof of Lemma \ref{prelim:lem-maximum-tail}.  

\begin{proof}[Proof of Lemma \ref{prelim:lem-maximum-tail}:]
Using a union bound, we have that 
\begin{align*}
&\, \bP \Big( \max_{j\in J} |X_j| \geq \beta^{-\frac{1}{q}} \big( A+\gamma^{-1} \log(|J|)+\lambda)^{\frac{1}{q}} \Big)
\leq \sum_{j\in J} \bP \Big( |X_j| \geq \beta^{-\frac{1}{q}} \big( A+\gamma^{-1}\log(|J|)+\lambda)^{\frac{1}{q}} \Big) \\
\leq &\, 
\sum_{j\in J} B e^{-\log(|J|)-\gamma\lambda} \leq B e^{-\gamma\lambda}. \qedhere
\end{align*}
\end{proof}

We also need the following quantitative version of Kolmogorov's continuity theorem, which can be found in \cite[Theorem 4.3.2]{S93}. 

\begin{lemma}[Kolmogorov's continuity theorem]\label{prelim:lem-kolmogorov}
Let $(X(t))_{t\in [0,1]}$ be a continuous stochastic process taking values in a Banach space $B$. Furthermore, let $A>0$, let $0<\alpha<\beta<1$, and let $q\geq 1$. Finally, assume that 
\begin{equation*}
\sup_{0\leq s < t \leq 1} \E \bigg[ \bigg( \frac{\| X(t)-X(s)\|_B}{|t-s|^{\beta+\frac{1}{q}}} \bigg)^q \bigg]^{\frac{1}{q}} \leq A.
\end{equation*}
Then, it holds that 
\begin{equation*}
 \E \bigg[  \sup_{0\leq s < t \leq 1} \bigg( \frac{\| X(t)-X(s)\|_B}{|t-s|^{\alpha}} \bigg)^q \bigg]^{\frac{1}{q}} \lesssim_{\alpha,\beta} C_{\alpha,\beta} A,
\end{equation*}
where the implicit constant depends on $\alpha$ and $\beta$, but is uniform in $q$.
\end{lemma}

\section{\protect{The Gibbs measure in infinite volume}}\label{section:measure}

In this section, we study the infinite-volume $\Phi^{p+1}_1$-measure and its finite-volume approximations. In particular, we prove Theorem \ref{intro:thm-measure}, which controls the $L^\infty([-R,R])$-norm of the samples. In Subsection \ref{section:hairer-steele}, we use the Hairer-Steele method from \cite{HS22} to control exponential moments of the $(p+1)$-th power of $\| \phi \|_{L^{p+1}([-1,1])}$. Together with Bernstein's inequality, translation-invariance, and maximum tail inequalities, this implies that 
\begin{equation*}
\big\| P_{<N} \phi \big \|_{L^\infty([-R,R])} \lesssim N^{\frac{1}{p+1}} \log(R)^{\frac{1}{p+1}}
\end{equation*}
with high probability. In Subsection \ref{section:Gaussian}, we rely on a Brascamp-Lieb inequality \cite{BL76}, which allows us to control the Gibbs measures $\mu_L$ using the Gaussian measures $\cg_L$. Together with Gaussian estimates, this implies that
\begin{equation*}
\big\| P_{\geq N} \phi \big\|_{L^\infty([-R,R])} \lesssim N^{-\frac{1}{2}}  \log(RN)^{\frac{1}{2}}
\end{equation*}
with high probability. In Subsection \ref{section:proof-measure}, we then obtain Theorem \ref{intro:thm-measure} by combining both estimates and optimizing in $N$. Finally, in Subsection \ref{section:measure-coupling}, we construct a coupling of the Gibbs measures satisfying the assumptions stated in Theorem \ref{intro:thm-dynamics}. This coupling is constructed using a density estimate (Lemma \ref{measure:lem-density}), an estimate of the Wasserstein-distance between $\mu$ and $\mu_L$ (Lemma \ref{measure:lem-Wasserstein}), and a quantitative version of the Skorokhod representation theorem (Proposition \ref{sk:prop-main}). 

\subsection{\protect{The Hairer-Steele argument}}\label{section:hairer-steele} 

As explained above, the first step in our argument relies on the Hairer-Steele method from \cite{HS22}. In fact, since the $\Phi^4_3$-measure considered in \cite{HS22} is much more singular than the $\Phi^{p+1}_1$-measure considered in this article, the technical aspects of the proof will be much simpler than in \cite{HS22}. 

\begin{proposition}[Hairer-Steele estimate]\label{measure:prop-moments}
Let $p>1$ and let $0<\beta<\frac{1}{p+1}$. Then, it holds that 
\begin{equation}\label{measure:eq-moments}
\sup_{L\geq 10} \int \exp \Big( \beta \| \phi \|_{L^{p+1}([-1,1])}^{p+1} \Big) \mathrm{d}\mu_L(\phi) \lesssim_{\beta,p} 1.
\end{equation}
\end{proposition}

Using the definition of the Gibbs measure $\mu_L$, it is easy to see that the integral in \eqref{measure:eq-moments} is finite for any fixed $L\geq 10$. The important aspect of  \eqref{measure:eq-moments} is that the integral can be bounded uniformly in~$L$, which is non-trivial. We first define a probability measure $\nu_L$ by 
\begin{equation}\label{measure:eq-nu-L}
\mathrm{d}\nu_L (\phi) = \mathcal{Z}_L^{-1} 
\exp \Big( \beta \| \phi \|_{L^{p+1}([-1,1])}^{p+1} \Big) \mathrm{d}\mu_L(\phi),
\end{equation}
where $\mathcal{Z}_L$ is the normalization constant. By definition, it then follows that  
\begin{equation}\label{measure:eq-moment-ZL}
\mathcal{Z}_L = \int \exp \Big( \beta \| \phi \|_{L^{p+1}([-1,1])}^{p+1} \Big) \mathrm{d}\mu_L. 
\end{equation}
In order to obtain Proposition \ref{measure:prop-moments}, we therefore have to obtain an upper bound on $\mathcal{Z}_L$. To this end, we note that $\nu_L$ is the Gibbs measure corresponding to the energy 
\begin{equation}\label{measure:eq-modified-energy}
 \int_{\T_L} \bigg( \frac{|\phi|^2}{2} + \frac{|\nabla \phi|^2}{2} + \frac{|\phi|^{p+1}}{p+1} \bigg) \dx - \beta \int_{\T_L} \ind_{[-1,1]} |\phi|^{p+1} \dx. 
\end{equation}
The Langevin equation corresponding to the energy \eqref{measure:eq-modified-energy}, which leaves the measure $\nu_L$ invariant, is given by the nonlinear stochastic heat equation 
\begin{equation}\label{measure:eq-psiL}
\begin{cases}
\begin{aligned}
(\partial_t + 1 - \Delta ) \psi_L  &= - \big( 1-\beta  (p+1) \ind_{[-1,1]}\big) |\psi_L|^{p-1} \psi_L + \sqrt{2} \zeta_L, \\ 
\psi_L(0)&= \psi_L^{(0)}.
\end{aligned}
\end{cases}
\end{equation}
In \eqref{measure:eq-psiL}, $\zeta_L$ is an $2\pi L$-periodic, complex-valued, space-time white noise. 
To state our estimates for $\psi_L$, we also introduce the linear stochastic object $\linear[L]$, which is defined as the solution of 
\begin{equation}\label{measure:eq-linear-L}
\begin{cases}
\begin{aligned}
(\partial_t + 1 - \Delta ) \linear[L]  &=  \sqrt{2} \zeta_L, \\ 
\linear[L](0)&=0. 
\end{aligned}
\end{cases}
\end{equation}

\begin{lemma}[Pointwise estimates of $\psi_L$]\label{measure:lem-pointwise}
Let $p>1$, let $0<\beta<\frac{1}{p+1}$, and let  $C=C_{\beta,p}\geq 1$ be a sufficiently large constant depending only on $\beta$ and $p$. For all $L\geq 10$, it then holds that\footnote{As will be clear below, the linear dependence of the right-hand side in \eqref{measure:eq-pointwise} on $\linear[L]$ is unimportant.}
\begin{equation}\label{measure:eq-pointwise}
\big\| \psi_L \big\|_{L_{t,x}^\infty([\frac{1}{2},1]\times [-1,1])} \leq C \Big( 1 + \big\| \linear[L] \big\|_{L^\infty_{t,x}([0,1]\times [-2,2])}\Big).
\end{equation}
\end{lemma}

We emphasize that the estimate \eqref{measure:eq-pointwise} is uniform both in the size $L$ and the initial data $\psi_L^{(0)}$, which will be used heavily in the proof of Proposition \ref{measure:prop-moments}. The estimate \eqref{measure:eq-pointwise} is proven using a variant of the arguments from \cite{MW20-companion,MW20}.

\begin{proof}
In the following argument, we use $C=C_{\beta,p}\geq 1$ and  $c=c_{\beta,p}>0$ for sufficiently large or small constants, respectively. The precise values of $C$ and $c$ are left unspecified and can change from line to line.\\

In order to estimate $\psi_L$, we introduce the nonlinear remainder $\varphi_L:= \psi_L- \linear[L]$, which solves

\begin{equation}\label{measure:eq-pointwise-p1}
\begin{cases}
\begin{aligned}
(\partial_t + 1 - \Delta ) \varphi_L  &= - \big( 1-\beta  (p+1) \ind_{[-1,1]}\big) \big| \, \linear[L]+\,\varphi_L \big|^{p-1} \big( \linear[L] + \varphi_L \big), \\ 
\varphi_L(0)&= \psi_L^{(0)}.
\end{aligned}
\end{cases}
\end{equation}
Since $\linear[L]$ can clearly be bounded using the right-hand side of \eqref{measure:eq-pointwise}, it suffices to prove the estimate in \eqref{measure:eq-pointwise} with $\psi_L$ replaced by $\varphi_L$. Since $\varphi_L$ is complex-valued, however, it is difficult to directly work with $\varphi_L$, and we instead work with $\chi_L :=|\varphi_L|^2$. From \eqref{measure:eq-pointwise-p1}, it then follows that  
\begin{align*}
(\partial_t - \Delta) \chi_L 
&= - 2 |\nabla \varphi_L|^2  + 2 \Re \big( \widebar{\varphi_L} \big( \partial_t - \Delta \big) \varphi_L \big) \\
&=  - |\nabla \varphi_L|^2 -2 |\varphi_L|^2 - 2  \big( 1-\beta  (p+1) \ind_{[-1,1]}\big) \Re \Big(  \widebar{\varphi_L} \big| \linear[L]+\varphi_L \big|^{p-1} \big( \linear[L] + \varphi_L \big) \Big).
\end{align*}

Using $-|\nabla \varphi_L|^2 \leq 0$,  $-|\varphi_L|^2 \leq 0$, $\beta (p+1)<1$, and Young's inequality, we obtain that 
\begin{equation}\label{measure:eq-pointwise-p2}
(\partial_t - \Delta) \chi_L \leq - c |\varphi_L|^{p+1} +  C \big( 1 + \big| \,\linear[L] \big|\big)^{p+1} = - c \chi_L^{\frac{p+1}{2}} + C \big( 1 + \big| \,\linear[L] \big|\big)^{p+1}.
\end{equation}
Using the maximum principle from \cite[Theorem 4.2]{MW20-companion} or \cite[Lemma 2.7]{MW20}, it then follows that 
\begin{equation*}
\big\| \chi_L \big\|_{L^\infty_{t,x}([\frac{1}{2},1]\times [-1,1])}
\leq C \big( 1+ \big\| \linear[L] \big\|_{L^\infty_{t,x}([0,1]\times [-2,2])}^2\big).
\end{equation*}
Using the definition of $\chi_L$, we then obtain that 
\begin{equation*}
\big\| \varphi_L \big\|_{L^\infty_{t,x}([\frac{1}{2},1]\times [-1,1])}
\leq C \big( 1+ \big\| \linear[L] \big\|_{L^\infty_{t,x}([0,1]\times [-2,2])}\big),
\end{equation*}
which completes the proof.
\end{proof}

\begin{lemma}\label{measure:lem-tightness}
Let $p> 1$, let $0<\beta<\frac{1}{p+1}$, and let $\lambda=\lambda_{\beta,p}$ be a sufficiently large constant depending only on $\beta$ and $p$. Then, it holds that 
\begin{equation}\label{measure:eq-tightness}
\inf_{L\geq 10}\nu_L \Big( \Big\{  \,  \| \phi \|_{L^\infty([-1,1])} \leq \lambda \Big\} \Big) \geq \frac{1}{2}.
\end{equation}
\end{lemma}

\begin{proof} Throughout the argument, we work on an abstract probability space $(\Omega,\mathcal{F},\mathbb{P})$. For each $L\geq 10$, we let $\psi_L^{(0)}$ be a random function whose law equals $\nu_L$ and let $\zeta_L$ be an $2\pi L$-periodic, complex-valued space-time white noise which is independent of $\psi_L^{(0)}$. Furthermore, we let $\psi_L$ be the corresponding solution of \eqref{measure:eq-psiL}.
Using the invariance of $\nu_L$ under \eqref{measure:eq-psiL}, it then follows that 
\begin{align*}
    \nu_L \Big( \Big\{ \,  \| \phi \|_{L^\infty([-1,1])} > \lambda \Big\} \Big) 
    = \mathbb{P} \Big( \big\| \psi_L^{(0)} \big\|_{L^\infty([-1,1])} > \lambda \Big) 
    =\mathbb{P} \Big( \big\| \psi_L(1)\big\|_{L^\infty([-1,1])}> \lambda \Big). 
\end{align*}
Using Lemma \ref{measure:lem-pointwise} and Markov's inequality\footnote{One can obtain much better decay in $\lambda$ using higher moments of $\linear[L]$, but this is irrelevant for our argument.}, we further obtain that
\begin{align*}
    \mathbb{P} \Big( \big\| \psi_L(1)\big\|_{L^\infty([-1,1])}>\lambda \Big) &\leq \mathbb{P}\Big( C \big( 1+\big\| \, \linear[L] \big\|_{L^\infty([-1,1]\times [-2,2])}\big) > \lambda \Big) \\
    &\leq \frac{C}{\lambda} \Big( 1 + \mathbb{E}\big[\big\| \, \linear[L] \big\|_{L^\infty([-1,1]\times [-2,2])} \big] \Big).
\end{align*}
Using standard estimates for the linear stochastic object $\linear[L]$ (see e.g. \cite[Section 5]{MW17} or \cite[Proposition 2.4]{GHOZ24}), it holds that 
\begin{equation*}
\sup_{L\geq 10}\mathbb{E}\big[\big\| \, \linear[L] \big\|_{L^\infty([-1,1]\times [-2,2])} \big] \lesssim 1.
\end{equation*}
As a result, it follows that
\begin{equation*}
\sup_{L\geq 10}\nu_L \Big( \Big\{ \,   \| \phi \|_{L^\infty([-1,1])} > \lambda \Big\} \Big) 
\lesssim_{\beta,p} \frac{1}{\lambda}.
\end{equation*}
By choosing $\lambda$ sufficiently large, we then obtain the desired estimate.
\end{proof}

Equipped with Lemma \ref{measure:lem-tightness}, we can now prove Proposition \ref{measure:prop-moments}. 

\begin{proof}[Proof of Proposition \ref{measure:prop-moments}:] 
We let $\lambda = \lambda_{\beta,p}$ be as in Lemma \ref{measure:lem-tightness}.
To simplify the notation, set 
\begin{equation*}
\mathcal{K}_L := \Big\{  \, \big\| \phi \big\|_{L^\infty([-1,1])} \leq \lambda \Big\}. 
\end{equation*}
Using \eqref{measure:eq-nu-L} and Lemma \ref{measure:lem-tightness}, we then obtain that 
\begin{align*}
\frac{1}{2} \leq \nu_L \big( \mathcal{K}_L \big) 
&= \mathcal{Z}_L^{-1} \int_{\mathcal{K}_L} \exp \Big( \beta \| \psi_L \|_{L^{p+1}([-1,1])}^{p+1}\Big) \mathrm{d}\mu_L(\psi_L) \\
&\leq  \mathcal{Z}_L^{-1} \int_{\mathcal{K}_L} \exp \Big( \beta \lambda^{p+1}\Big) \mathrm{d}\mu_L(\psi_L) \\
&\leq \mathcal{Z}_L^{-1} \exp(\beta \lambda^{p+1}). 
\end{align*}
In the last inequality, we used that $\mu_L$ is a probability measure, which implies that $\mu_L(\mathcal{K}_L)\leq 1$. By rearranging the above estimate, we obtain that 
\begin{equation*}
\mathcal{Z}_L \leq 2 \exp(\beta \lambda^{p+1}). 
\end{equation*}
Together with \eqref{measure:eq-moment-ZL}, this implies the desired estimate \eqref{measure:eq-moments}.
\end{proof}

We now record a simple corollary of Proposition \ref{measure:prop-moments}, which will be used to control minor error terms below (see e.g. the proofs of Lemma \ref{measure:lem-moments-maximum} and  Lemma \ref{measure:lem-Gaussian}).

\begin{corollary}\label{measure:cor-crude}
Let $p>1$ and let $C=C_p$ and $c=c_p$ be sufficiently large and small constants, respectively. Then, it holds for all $\lambda>0$ that 
\begin{equation}
\sup_{L\geq 10} \mu_L \Big( \Big\{  \big\|  \langle  x \rangle^{-10} \phi \big\|_{L^1(\R)} > C \lambda^{\frac{1}{p+1}} \Big\} \Big) \leq  C e^{-c\lambda}. 
\end{equation}
\end{corollary}

\begin{proof}
We choose the constant $Z>0$ such that  $Z^{-1} \langle x \rangle^{-10} \mathrm{d}x$ is a probability measure on $\R$. Using Jensen's inequality, Tonelli's theorem, and Proposition \ref{measure:prop-moments}, it then follows that 
\begin{align*}
\int \exp\Big( c \big\| \langle x \rangle^{-10} \phi \|_{L^1(\R)}^{p+1} \Big) \mathrm{d}\mu_L(\phi) 
&\leq  \int \bigg( Z^{-1} \int \langle x \rangle^{-10}  \exp\big( c Z^{p+1} |\phi(x)|^{p+1} \big) \mathrm{d}x \bigg) \mathrm{d}\mu_L(\phi) \\
&= Z^{-1} \int  \langle x \rangle^{-10} \bigg( \int \exp\big( c Z^{p+1} |\phi(x)|^{p+1} \big) \mathrm{d}\mu_L(\phi)  \bigg) \mathrm{d}x\\ 
&\lesssim Z^{-1} \int \langle x \rangle^{-10} \mathrm{d}x \lesssim 1.
\end{align*}
Together with Markov's inequality, this implies the desired estimate.
\end{proof}

We now use Proposition \ref{measure:prop-moments}, together with Bernstein's inequality, translation-invariance, and maximum tail inequalities, to prove that 
\begin{equation*}
\big\| P_{<N} \phi \big\|_{L^\infty([-R,R])} \lesssim N^{\frac{1}{p+1}} \log(R)^{\frac{1}{2}}
\end{equation*}
holds with high probability.

\begin{lemma}\label{measure:lem-moments-maximum}
Let $C=C_p\geq 1$  and $c=c_p>0$ be sufficiently large and small constants, respectively. 
Furthermore, let $p>1$, let $N\geq 1$, let $R\geq 10$, and let $\lambda >0$. Then, it holds that 
\begin{equation}\label{measure:eq-moments-maximum}
\sup_{L\geq 10} \mu_L \Big( \Big\{ \big\| P_{< N} \phi \big\|_{L^\infty([-R,R])} > C N^{\frac{1}{p+1}} \big( \log(R)+\lambda\big)^{\frac{1}{p+1}} \Big\} \Big) 
\leq C e^{-c\lambda}.
\end{equation}
\end{lemma}

\begin{proof}
We first let $C^\prime=C^\prime_p$ be sufficiently large depending on $p$ and then let $C=C_{p,C^\prime}$ be sufficiently large depending on $p$ and $C^\prime$. Furthermore, we let $c^\prime:= 1/C^\prime$ and $c=1/C$.

From Bernstein's inequality (Lemma \ref{prelim:lem-bernstein}), it follows that 
\begin{equation*}
\big\| P_{<N} \phi \big\|_{L^\infty([-R,R])} 
\lesssim N^{\frac{1}{p+1}} \big\| \phi \big\|_{L^{p+1}_{\loc}([-2R,2R])} 
+ (RN)^{-10} \big\| \langle x \rangle^{-10} \phi \big\|_{L^1(\R)}.
\end{equation*}
In order to obtain \eqref{measure:eq-moments-maximum}, it therefore suffices to prove that 
\begin{align}
\sup_{L\geq 10} \mu_L \Big( \Big\{ \big\| \phi \big\|_{L^{p+1}_{\loc}([-2R,2R])} \geq C^\prime \big( \log(R)+\lambda\big)^{\frac{1}{p+1}} \Big\} \Big) 
&\leq C^\prime e^{-c^\prime \lambda}, \label{measure:eq-moments-maximum-p1} \\
\sup_{L\geq 10} \mu_L \Big( \Big\{  \big\| \langle x \rangle^{-10} \phi \big\|_{L^1(\R)}  \geq C^\prime \lambda^{\frac{1}{p+1}} \Big\}  \Big) 
&\leq C^\prime e^{-c^\prime \lambda}. \label{measure:eq-moments-maximum-p2} 
\end{align}
The second estimate \eqref{measure:eq-moments-maximum-p2} follows directly from Corollary \ref{measure:cor-crude}, and it therefore remains to prove the first estimate \eqref{measure:eq-moments-maximum-p1}. To this end, we choose any $\beta=\beta_p$ satisfying $0<\beta<1/(p+1)$. From Proposition \ref{measure:prop-moments} and translation-invariance, it then follows that 
\begin{equation}\label{measure:eq-moments-maximum-p3}
\sup_{L\geq 10} \sup_{x_0 \in \R} \int e^{\beta \| \phi \|_{L^{p+1}([x_0-1,x_0+1])}^{p+1}} \mathrm{d}\mu_L(\phi) \lesssim_p 1.
\end{equation}
Together with Markov's inequality, \eqref{measure:eq-moments-maximum-p3} then implies for all $\lambda >0$ that
\begin{equation}\label{measure:eq-moments-maximum-p3p}
\sup_{L\geq 10} \sup_{x_0 \in \R} \mu_L \Big( \Big\{ \big\| \phi \big\|_{L_x^{p+1}([x_0-1,x_0+1])} \geq \beta^{-\frac{1}{p+1}} \lambda^{\frac{1}{p+1}}\Big\}  \Big)  \lesssim_p e^{-\lambda}.
\end{equation}
We now let $\Lambda_R:= [-2R,2R] \medcap \Z$. From this, it follows that each interval $I\subseteq [-2R,2R]$ of length $|I|\leq 2$ can be covered using at most two intervals of the form $[x_0-1,x_0+1]$, where $x_0\in \Lambda_R$. From the triangle inequality and \eqref{prelim:eq-Lploc}, we then obtain that 
\begin{equation}\label{measure:eq-moments-maximum-p4}
\big\| \phi \big\|_{L^{p+1}_{\loc}([-2R,2R])} \leq 2 \max_{x_0 \in \Lambda_R} \big \| \phi \big\|_{L^{p+1}([x_0-1,x_0+1])}.
\end{equation}
By combining the maximum tail estimate (Lemma \ref{prelim:lem-maximum-tail}), \eqref{measure:eq-moments-maximum-p3p}, and \eqref{measure:eq-moments-maximum-p4}, we then obtain that
\begin{align*}
&\, \sup_{L\geq 10} \mu_L \Big( \Big\{ \big\| \phi \big\|_{L_{\loc}^{p+1}([-2R,2R])} \geq 2 \beta^{-\frac{1}{p+1}} \big( \log(\# \Lambda_R) +\lambda\big)^{\frac{1}{p+1}}  \Big\} \Big)\\
\leq &\, \sup_{L\geq 10} \mu_L \Big( \Big\{ \max_{x_0 \in \Lambda_R} \big \| \phi \big\|_{L^{p+1}([x_0-1,x_0+1])} \geq  \beta^{-\frac{1}{p+1}} \big( \log(\# \Lambda_R) +\lambda\big)^{\frac{1}{p+1}}  \Big\} \Big) \lesssim_p e^{-\lambda}.
\end{align*}
Since $\#\Lambda_R \sim R$, this implies \eqref{measure:eq-moments-maximum-p1}, and thereby completes the proof of \eqref{measure:eq-moments-maximum}.
\end{proof}

At the end of this subsection, we record a simple corollary of Lemma \ref{measure:lem-pointwise}, which will be useful in the proofs of Lemma \ref{measure:lem-density} and Lemma \ref{measure:lem-Wasserstein} below.

\begin{corollary}\label{measure:cor-Ltx-bound}
Let $p>1$. Let $C=C_p \geq 1$ and $c=c_p>0$ be sufficiently large and small constants, respectively. Let $L\geq 10$, let $\psi_L^{(0)}$ be drawn from the Gibbs measure $\mu_L$, let $\zeta_L$ be a $2\pi L$-periodic space-time white noise, and assume that $\psi_L^{(0)}$ and $\zeta_L$ are independent. Furthermore, let $\psi_L$ be the solution of \eqref{measure:eq-psiL} with $\beta=0$. Then, it holds for all $T\geq 1$, $R\geq 10$, and $\lambda>0$ that 
\begin{equation}\label{measure:eq-Ltx-bound}
\bP\Big( \big\| \psi_L \big\|_{L_t^\infty L_x^\infty([0,T]\times [-R,R])} \geq C \big(\log(T+R)+\lambda\big)^{\frac{1}{2}} \Big) \leq C e^{-c\lambda}.
\end{equation}
\end{corollary}

\begin{remark} We note that the exponent $1/2$ in \eqref{measure:eq-Ltx-bound} can later be improved by using Theorem \ref{intro:thm-measure} and a similar argument as in the proof of Proposition \ref{uniform:prop-linfty}. For our purposes, however, Corollary~\ref{measure:cor-Ltx-bound} is sufficient. 
\end{remark}

\begin{proof}[Proof of Corollary \ref{measure:cor-Ltx-bound}:]
Due to the invariance of the Gibbs measure $\mu_L$ under the Langevin dynamics and the invariance of the Gibbs measure and space-time white noise under spatial translations, the law of $\psi_L$ is invariant under space-time translations. Due to Lemma \ref{prelim:lem-maximum-tail}, it then suffices to prove that
\begin{equation}\label{measure:eq-Ltx-bound-1}
\bP\Big( \big\| \psi_L \big\|_{L_t^\infty L_x^\infty([\frac{1}{2},1]\times [-1,1])} \geq C \lambda^{\frac{1}{2}} \Big) \leq C e^{-c\lambda}.
\end{equation}
The estimate \eqref{measure:eq-Ltx-bound-1} follows directly from Lemma \ref{measure:lem-pointwise} and standard estimates for the linear stochastic object $\linear[L]$.
\end{proof}

\subsection{Gaussian estimates}\label{section:Gaussian}

In this subsection, we control the Gibbs measure $\mu_L$ using the Gaussian free field $\cg_L$, which leads to the following lemma.

\begin{lemma}[Gaussian estimate]\label{measure:lem-Gaussian} Let $p>1$. Let $C=C_p\geq 1$  and $c=c_p>0$ be sufficiently large and small constants, respectively. 
Furthermore, let $N\geq 1$, let $R\geq 10$, and let $\lambda >0$. Then, it holds that 
\begin{equation}\label{measure:eq-Gaussian-estimate}
\sup_{L\geq 10} \mu_L \Big( \Big\{ \big\| P_{\geq N} \phi \big\|_{L^\infty([-R,R])} > C N^{-\frac{1}{2}} \big( \log(RN)+\lambda\big)^{\frac{1}{2}} \Big\} \Big) 
\leq C e^{-c\lambda}.
\end{equation}
\end{lemma}

The main ingredient used in the proof is an estimate of Brascamp and Lieb \cite[Theorem 5.1]{BL76}, which has already been used 
to study nonlinear Schrödinger equations in \cite{B00}. For the reader's convenience, we state it as a separate lemma below.

\begin{lemma}[Brascamp-Lieb]\label{measure:lem-brascamp-lieb}
Let $n\geq 1$, let $V\colon \R^n\rightarrow \R$ be an even, convex function, and let $A\in \R^{n\times n}$ be a positive definite matrix. Define the Gibbs measure $\mu$ and Gaussian measure $\cg$ by 
\begin{equation*}
\mathrm{d}\mu(\phi) = \mathcal{Z}_1^{-1} \exp\Big( - \langle \phi , A \phi \rangle - V(\phi) \Big) \mathrm{d}\phi 
\qquad \text{and} \qquad
 \mathrm{d}\cg(\phi)=\mathcal{Z}_2^{-1} \exp\Big( - \langle \phi , A \phi \rangle\Big) \mathrm{d}\phi,
\end{equation*}
where $\mathcal{Z}_1$ and $\mathcal{Z}_2$ are normalization constants and $\mathrm{d}\phi$ is the Lebesgue measure on $\R^n$. For any linear function $f\colon \R^n \rightarrow \R$ and $q\geq 1$, it then holds that
\begin{equation}\label{measure:eq-Brascamp-Lieb-1}
\int \big| f(\phi) \big|^q \mathrm{d}\mu(\phi) \leq \int \big| f(\phi)\big|^q \mathrm{d}\cg(\phi).
\end{equation}
Furthermore, for any $\beta>0$, it also holds that 
\begin{equation}\label{measure:eq-Brascamp-Lieb-2}
\int \exp\big( \beta |f(\phi)|^2 \big) \mathrm{d}\mu(\phi) \leq \int \exp\big( \beta |f(\phi)|^2 \big)\mathrm{d}\cg(\phi).
\end{equation}
\end{lemma}
In \cite[Theorem 5.1]{BL76}, the left-hand side of \eqref{measure:eq-Brascamp-Lieb-1} also involves the mean of $f$ with respect to~$\mu$. However, since we made the additional assumption that $V$ is even, the mean equals zero. We also note that while \eqref{measure:eq-Brascamp-Lieb-2} is not stated as part of \cite[Theorem 5.1]{BL76}, it follows directly from~\eqref{measure:eq-Brascamp-Lieb-1} and an expansion of the exponential into a power series.

\begin{proof}[Proof of Lemma \ref{measure:lem-Gaussian}:] 
We choose constants $(C_j)_{j=0}^4$ satisfying 
\begin{align*}
C_0 \gg C_1 \gg C_2 \gg C_3 \gg C_4
\end{align*}
and then define $(c_j)_{j=0}^4$ as $c_j := C_j^{-1}$. Furthermore, we define $C$ and $c$ from the statement of the lemma as $C:= C_0$ and $c:= c_0$. In the rest of the proof, we assume that $e^{-c \lambda} \leq C^{-1}$, since otherwise the desired estimate is trivial.
For expository purposes, we separate the proof into five steps. \\

\emph{Step 1: Reduction to $P_N$ instead of $P_{\geq N}$.} To simplify the notation, we let 
\begin{align*}
E_{N,\lambda} &:= \Big\{ \big\| P_{\geq N} \phi \big\|_{L^\infty([-R,R])} \leq C N^{-\frac{1}{2}} \big( \log(RN)+\lambda \big)^{\frac{1}{2}} \Big\}, \\ 
\widetilde{E}_{N,\lambda} &:= \Big\{ \big\| P_N \phi \big\|_{L^\infty([-R,R])} \leq C_1 N^{-\frac{1}{2}} \big( \log(RN)+\lambda \big)^{\frac{1}{2}} \Big\}. 
\end{align*}
We now claim that, in order to obtain \eqref{measure:eq-Gaussian-estimate}, it suffices to prove that 
\begin{equation}\label{measure:eq-Gaussian-p1}
\sup_{L\geq 10} \mu_L \big(\widetilde{E}_{N,\lambda}^c \big) \leq C_1 e^{-c_1 \lambda}.
\end{equation}
To see this, we let $0<\delta<1$ and consider the event $\medcap_{M\geq N} \widetilde{E}_{M,(M/N)^\delta \lambda}$. On this event, it holds that 
\begin{align*}
\| P_{\geq N} \phi \|_{L^\infty([-R,R])} 
&\leq \sum_{M\geq N} \| P_M \phi \|_{L^\infty([-R,R])} \\  
&\leq C_1 \sum_{M\geq N} M^{-\frac{1}{2}} \Big( \log(RM) + (M/N)^\delta \lambda \Big)^{\frac{1}{2}} \\
&\leq C N^{-\frac{1}{2}} \Big( \log(RN)+ \lambda \Big)^{\frac{1}{2}}.
\end{align*}
As a result, it holds that $E_{N,\lambda} \subseteq \medcap_{M\geq N} \widetilde{E}_{M,(M/N)^\delta \lambda}$. Furthermore, from \eqref{measure:eq-Gaussian-p1} we have the probability estimate
\begin{equation*}
\sup_{L\geq 10} \mu_L \Big(  \Big( \bigcap_{M\geq N} \widetilde{E}_{M, (M/N)^\delta \lambda} \Big)^c \Big)
\leq C_1 \sum_{M\geq N} e^{-c_1 \lambda (M/N)^\delta } \leq C e^{-c \lambda}.
\end{equation*}
In estimating the sum over $M$, we used that $e^{-c\lambda} \leq C^{-1} \leq 2^{-1}$. \\

\emph{Step 2: From supremum to maximum.} Let $\Lambda_{R,N}\subseteq [-R,R]$ be a grid with step-size~$\sim (RN)^{-100}$. From Lemma \ref{prelim:lem-local-constancy}, it then follows that 
\begin{equation*}
\big\| P_N \phi \big\|_{L^\infty([-R,R])} \leq \max_{x\in \Lambda_{R,N}} \big| P_N \phi(x) \big| + (RN)^{-10} \big\| \langle x \rangle^{-10} \phi \big\|_{L^1(\R)}. 
\end{equation*}
Due to Corollary \ref{measure:cor-crude} and due to our earlier restriction to large values of $\lambda$, it holds that 
\begin{equation*}
\sup_{L\geq 10} \mu_L \Big( \Big\{ \big\| \langle x \rangle^{-10} \phi \big\|_{L^1(\R)} \geq C_2 \lambda^{\frac{1}{2}}\Big\}  \Big)
\leq \sup_{L\geq 10} \mu_L \Big( \Big\{ \big\| \langle x \rangle^{-10} \phi \big\|_{L^1(\R)} \geq C_2 \lambda^{\frac{1}{p+1}}\Big\}  \Big)
\leq C_2 e^{-c_2 \lambda}.
\end{equation*}
In order to obtain \eqref{measure:eq-Gaussian-p1}, it therefore suffices to prove that 
\begin{equation}\label{measure:eq-Gaussian-p2}
\sup_{L \geq 10} \mu_L \Big( \Big\{  \max_{x\in \Lambda_{R,N}} \big| P_N \phi(x) \big| \geq C_2 N^{-\frac{1}{2}} \big( \log(RN)+\lambda \big)^{\frac{1}{2}} \Big\} \Big)
\leq C_2 e^{-c_2 \lambda}.
\end{equation}

\emph{Step 3: Application of the maximum tail estimate.} Using Lemma \ref{prelim:lem-maximum-tail}, the maximum tail estimate~\eqref{measure:eq-Gaussian-p2} can be further reduced to the moment estimate 
\begin{equation}\label{measure:eq-Gaussian-p3}
\sup_{L\geq 10} \max_{x\in \Lambda_{R,N}} \int e^{c_3 N |P_N \phi(x)|^2} \mathrm{d}\mu_L(\phi) \leq C_3. 
\end{equation}
Due to the translation-invariance of $\mu_L$, the maximum over $x\in \Lambda_{R,N}$ is not needed, i.e., it suffices to estimate
\begin{equation}\label{measure:eq-Gaussian-p4}
\sup_{L\geq 10}  \int e^{c_3 N |P_N \phi(0)|^2} \mathrm{d}\mu_L(\phi) \leq C_3. 
\end{equation}

\emph{Step 4: Using the Brascamp-Lieb inequality.}  We now make use of the Brascamp-Lieb inequality. From \eqref{measure:eq-Gaussian-p4}, it follows that
\begin{equation*}
 \int e^{c_3 N |P_N \phi(0)|^2} \mathrm{d}\mu_L(\phi) \leq  \int e^{c_3 N |P_N \phi(0)|^2} \mathrm{d}\cg_L(\phi),
\end{equation*}
where $\cg_L$ is the Gaussian free field from \eqref{intro:eq-GFF-rigorous}. Since $\cg_L$ is a Gaussian measure, it then suffices to bound the variance of $P_N\phi(0)$ with respect to $\cg_L$, i.e., it suffices to prove that 
\begin{equation}\label{measure:eq-Gaussian-p5}
\int \big| P_N \phi(0) |^2 d\cg_L(\phi) \leq C_4 N^{-1}.
\end{equation}

\emph{Step 5: Estimate of variance under $\cg_L$.} 
We prove the remaining estimate \eqref{measure:eq-Gaussian-p5} using frequency-space methods. However, we mention that it can also be proven using physical-space methods, i.e., by working with the kernel of $(1-\Delta)^{-1}$. We let $(\Omega,\mathcal{F},\bP)$ be an abstract probability space and let $(g_n)_{n\in \Z_L}$ be a sequence of independent standard Gaussians, where $\Z_L:= L^{-1} \Z$. Since $( (2\pi L)^{-\frac{1}{2}} e^{inx})_{n\in \Z_L}$ is an orthonormal eigenbasis of $1-\Delta$ on $L^2(\T_L)$ with eigenvalues $\langle n \rangle^2$, it then holds that
\begin{equation*}
\operatorname{Law}_{\cg_L} \Big( P_N \phi(x) \Big) = \operatorname{Law}_{\bP} \bigg( \frac{1}{\sqrt{2\pi L}} \sum_{n\in \Z_L} \frac{\rho_N(n)}{\langle n \rangle} g_n e^{inx} \bigg),
\end{equation*}
where $\rho_N$ is the Littlewood-Paley symbol from \eqref{prelim:eq-rho}. From this, we obtain
\begin{equation}\label{measure:eq-Gaussian-p6}
\begin{aligned}
&\, \int \big| P_N \phi(0) |^2 d\cg_L(\phi)  
= \frac{1}{2\pi L} \E \bigg[ \Big| \sum_{n\in \Z_L} \frac{\rho_N(n)}{\langle n \rangle} g_n \Big|^2 \bigg] \\ 
=&\,  \frac{1}{2\pi L} \sum_{n \in \Z_L} \frac{\rho_N(n)^2}{\langle n \rangle^2} 
\lesssim \frac{1}{LN^2} \, \# \{ n \in \Z_L \colon |n| \sim N \}  \lesssim N^{-1}.
\end{aligned}
\end{equation}
This completes the proof of \eqref{measure:eq-Gaussian-p5}, and hence the proof of this lemma.
\end{proof}

We now record the following corollary of Lemma \ref{measure:lem-Gaussian} and its proof, which will be needed in Section~\ref{section:uniform} below.

\begin{corollary}\label{measure:cor-Gaussian}
For all $0\leq \beta<1$, $r\geq 1$, and $N\geq 1$, it holds that 
\begin{align}
\sup_{L\geq 10} \Big( \int \big\| P_N \phi \big\|_{C_x^\beta([-1,1])}^r \mathrm{d}\mu_L(\phi) \Bigg)^{\frac{1}{r}}
&\lesssim \sqrt{r} N^{-\frac{1}{2}+\beta} \big( \log(N)+1\big)^{\frac{1}{2}}, \label{measure:eq-Gaussian-cor-e1} \\
\sup_{L\geq 10} \Big( \int \big\| \Delta P_N \phi \big\|_{C_x^\beta([-1,1])}^r \mathrm{d}\mu_L(\phi) \Bigg)^{\frac{1}{r}}
&\lesssim \sqrt{r} N^{\frac{3}{2}+\beta} \big( \log(N)+1\big)^{\frac{1}{2}}. \label{measure:eq-Gaussian-cor-e2} 
\end{align}
\end{corollary}

\begin{proof}[Proof of Corollary \ref{measure:cor-Gaussian}:] 
After using \eqref{prelim:eq-Hoelder-LWP-1} from Lemma \ref{prelim:lem-Hoelder-LWP} and using Corollary \ref{measure:cor-crude} to control the weighted $L^1(\R)$-term in \eqref{prelim:eq-Hoelder-LWP-1}, it suffices to prove \eqref{measure:eq-Gaussian-cor-e1} and \eqref{measure:eq-Gaussian-cor-e2} for $\beta=0$. 
For all $\lambda \geq 1$ and $N\geq 1$, it holds that $\lambda^{\frac{1}{2}} (1+\log(N))^{\frac{1}{2}}\geq (\lambda +\log(N))^{\frac{1}{2}}$.
Using Lemma \ref{measure:lem-Gaussian}, we then obtain for all $\lambda \geq 1$ that 
\begin{equation}\label{measure:eq-Gaussian-cor-p1}
\mu_L \bigg( \frac{\| P_N \phi\|_{L^\infty_x([-1,1])}}{N^{-\frac{1}{2}}(1+\log(N))^{\frac{1}{2}}} \geq C \lambda^{\frac{1}{2}}\bigg) \leq C e^{-c\lambda}.
\end{equation}

We note that even though Lemma \ref{measure:lem-Gaussian} controls $P_{\geq N} \phi$, while \eqref{measure:eq-Gaussian-cor-p1} involves $P_N \phi$, Lemma \ref{measure:lem-Gaussian} can still be used to obtain \eqref{measure:eq-Gaussian-cor-p1}. The reason is that $P_N \phi$ can be written as  $P_N\phi=P_{\geq N} \phi-P_{\geq 2N} \phi$. Alternatively, one can use \eqref{measure:eq-Gaussian-p1} from the proof of Lemma \ref{measure:lem-Gaussian} rather than its statement.  The estimate \eqref{measure:eq-Gaussian-cor-e1} now follows from the standard relation between tail and moment estimates, see e.g.~\cite[Proposition 2.5.2]{V18}. Since we restricted to the finite interval $[-1,1]$, the bound \eqref{measure:eq-Gaussian-cor-e2} cannot be deduced from \eqref{measure:eq-Gaussian-cor-e1} and $\| P_N\Delta\|_{L^\infty_x(\R)\rightarrow L^\infty_x(\R)}\lesssim N^2$. However, it follows from a minor modification of the proof of \eqref{measure:eq-Gaussian-cor-e1}, where we have to include an additional $|n|^2$-factor in \eqref{measure:eq-Gaussian-p6}. 
\end{proof}

\subsection{Proof of Theorem \ref{intro:thm-measure}}\label{section:proof-measure}

Equipped with Lemma \ref{measure:lem-moments-maximum} and Lemma \ref{measure:lem-Gaussian}, we are now ready to prove the main result of this section. 

\begin{proof}[Proof of Theorem \ref{intro:thm-measure}:] 
Let $N = N_R \geq 1$ be a deterministic parameter that remains to be chosen, let $C^\prime=C^\prime_p$ be sufficiently large, and let $c^\prime=c^\prime_p$ be sufficiently small. 
We now consider the event 
\begin{equation}\label{measure:eq-main-p1}
\begin{aligned}
&\Big\{ \big\| P_{< N} \phi \big\|_{L^\infty([-R,R])} \leq C^\prime N^{\frac{1}{p+1}} \big( \log(R) + \lambda\big)^{\frac{1}{p+1}} \Big\} \\ 
 &\bigcap \Big\{ \big\| P_{\geq N} \phi \big\|_{L^\infty([-R,R])} \leq C^\prime N^{-\frac{1}{2}} \big( \log(RN) + \lambda\big)^{\frac{1}{2}} \Big\}. 
 \end{aligned}
\end{equation}
From Lemma \ref{measure:lem-moments-maximum} and Lemma \ref{measure:lem-Gaussian}, it follows that the complement of the event in \eqref{measure:eq-main-p1} has probability $\leq 2 C^\prime e^{-c^\prime \lambda}$, which is sufficient. Furthermore, on the event \eqref{measure:eq-main-p1}, it holds that
\begin{align*}
 \big\|  \phi \big\|_{L^\infty([-R,R])}  &\leq  \big\| P_{< N} \phi \big\|_{L^\infty([-R,R])}  +  \big\| P_{\geq N} \phi \big\|_{L^\infty([-R,R])} \\
 &\leq C^\prime N^{\frac{1}{p+1}} \big( \log(R) + \lambda\big)^{\frac{1}{p+1}} + C^\prime N^{-\frac{1}{2}} \big( \log(RN) + \lambda\big)^{\frac{1}{2}} .
\end{align*}
By choosing 
\begin{equation*}
N \sim \big( \log(R) + \lambda \big)^{\frac{p-1}{p+3}},
\end{equation*}
we obtain the desired estimate. 
\end{proof}

\subsection{Coupling}\label{section:measure-coupling}

In this subsection, we construct a coupling of the Gibbs measures satisfying the assumptions in Theorem \ref{intro:thm-dynamics}. In fact, we prove a stronger estimate than in \eqref{intro:eq-coupling}, in which the probabilities have exponential rather than polynomial decay in $L$. 

\begin{proposition}[Coupling]\label{measure:prop-coupling}
Let $p>1$ and let $C\geq 1$ and $c,\eta>0$ be sufficiently large and small constants depending only on $p$, respectively. 
Then, there exist a probability space $(\Omega,\mathcal{F},\bP)$ and random continuous functions $\phi_L,\phi\colon (\Omega \times \R,\mathcal{F} \times \mathcal{B}(\R))\rightarrow \C$, where $L\in \dyadic$,  such that the following properties are satisfied:
\begin{enumerate}[label=(\roman*)]
    \item\label{measure:item-distribution} For all $L\in \dyadic$, we have that $\operatorname{Law}_\bP(\phi_L)=\mu_L$. Furthermore, we have that $\operatorname{Law}_\bP(\phi)=\mu$. 
    \item\label{measure:item-coupling} For all $L\in \dyadic$, it holds that
    \begin{equation*}
    \bP \Big( \big\| \phi - \phi_L \big\|_{C^0([-L^\eta,L^\eta])} > L^{-\eta} \Big) \leq C e^{-c L^\eta}.
    \end{equation*}
\end{enumerate}
\end{proposition}

In Proposition \ref{measure:prop-coupling}, the infinite-volume Gibbs measure $\mu$ is the weak limit of the finite-volume Gibbs measures $\mu_L$. The existence and uniqueness of the infinite-volume limit will be obtained as part of Lemma \ref{measure:lem-Wasserstein} below. We prove Proposition \ref{measure:prop-coupling} using Proposition \ref{sk:prop-main}, which is a quantitative version of the Skorokhod representation theorem. In order to use Proposition \ref{sk:prop-main}, we need to verify that the Gibbs measures satisfy the assumptions from Proposition \ref{sk:prop-main}.\ref{sk:item-h}-\ref{sk:item-w}. The Hölder-estimate from \ref{sk:item-h} directly follows from Lemma \ref{prelim:lem-Hoelder-LWP} and Lemma \ref{measure:lem-Gaussian}. Before we turn to the assumptions in Proposition \ref{sk:prop-main}.\ref{sk:item-d}-\ref{sk:item-w}, we need to make a few preparations. For all $\theta>0$ and $\phi \colon \R \rightarrow \C$, we define the exponentially-weighted norm
\begin{equation}\label{measure:eq-def-CE}
\| \phi \|_{\CE^\theta(\R)}:= \big\| e^{-\theta|x|}\phi(x)\big\|_{C^0(\R)}.
\end{equation}
To make use of \eqref{measure:eq-def-CE}, we first show that the massive heat-operator is bounded on $\CE^\theta(\R)$. 

\begin{lemma}\label{measure:lem-heat-CE}
Let $\theta>0$ and $t\geq 0$. For all $\phi\colon \R \rightarrow \C$, it then holds that 
\begin{equation}\label{measure:eq-heat-CE}
\big\| e^{t\Delta} \phi \big\|_{\CE^\theta(\R)} \leq 2 e^{\theta^2 t}  \| \phi \|_{\CE^\theta(\R)}.
\end{equation}
\end{lemma}

\begin{proof}
From the definition of the $\CE^\theta$-norm and the explicit formula for the kernel of $e^{t\Delta}$, it directly follows that
\begin{equation*}
\big\| e^{t\Delta} \phi \big\|_{\CE^\theta(\R)} \leq \Big( \sup_{x\in \R} \frac{1}{\sqrt{4\pi t}} \int_{\R} e^{-\theta |x|} e^{-\frac{|x-y|^2}{4t}} e^{\theta|y|} \dy \Big) \| \phi \big\|_{\CE^\theta(\R)}.
\end{equation*}
Together with the elementary estimate
\begin{equation*}
\frac{1}{\sqrt{4\pi t}} \int_{\R} e^{-\theta |x|} e^{-\frac{|x-y|^2}{4t}} e^{\theta|y|} \dy 
\leq \frac{1}{\sqrt{4\pi t}} \int_{\R} e^{\theta |x-y|} e^{-\frac{|x-y|^2}{4t}}  \dy \leq 2 e^{\theta^2 t},
\end{equation*}
this implies \eqref{measure:eq-heat-CE}.
\end{proof}

Equipped with Lemma \ref{measure:lem-heat-CE}, we now state and prove a density estimate for the one-point marginals of the Gibbs measures.

\begin{lemma}[A simple density estimate]\label{measure:lem-density}
Let $0<\kappa<\frac{3}{4}$ and let $C=C_{\kappa,p} \geq 1$ be sufficiently large. For all $L\geq 10$, all $x\in \R$, and all $a,b\in \R$ satisfying $a<b$, it then holds that 
\begin{equation}\label{measure:eq-density}
\mu_L \Big( \Big\{ \Re \phi (x) \in [a,b] \Big\} \Big),\mu_L \Big( \Big\{ \Im \phi (x) \in [a,b] \Big\} \Big) \leq C |b-a|^{\kappa}.
\end{equation}
\end{lemma}

\begin{proof} To simplify the notation below, we let $\delta:=b-a$. Since \eqref{measure:eq-density} is trivial for large $\delta$, we may assume that $0<\delta\leq 1$. We let $(\Omega,\mathcal{F},\bP)$ be a probability space that can support the following two independent random variables: A random function $\psi_L^{(0)}$ distributed according to the Gibbs measure $\mu_L$ and a $2\pi L$-periodic, complex-valued space-time white noise~$\zeta_L$. We then let $\psi_L$ be the corresponding solution of \eqref{measure:eq-modified-energy} with $\beta=0$. Due to the invariance of the Gibbs measure $\mu_L$ under its Langevin dynamics, it then holds that $\Law_\bP(\psi_L(t))=\mu_L$ for all $t\geq 0$. We now decompose
\begin{equation*}
\psi_L(t,x) = e^{t (\Delta-1)}\psi_L^{(0)}(x) + \linear[L] (t,x) + \varphi_L(t,x), 
\end{equation*}
where $\linear$ is as in \eqref{measure:eq-linear-L} and $\varphi_L$ is the solution of 
\begin{equation*}
(\partial_t + 1 - \Delta) \varphi_L = - |\psi_L|^{p-1} \psi_L
\end{equation*}
with initial data $\varphi(0)=0$. For any $t\geq 0$, it then holds that
\begin{align}
&\, \mu_L \Big( \big\{ \Re \phi(x) \in [a,a+\delta] \big\} \Big) \notag \\ 
=&\, \bP \Big( \Re \big( e^{t(\Delta-1)} \psi_L^{(0)} (x) + \linear[L](t,x) + \varphi_L (t,x) \big) \in [a,a+\delta] \Big) \notag \\
\leq&\, \bP \Big( \Re \big( e^{t(\Delta-1)} \psi_L^{(0)} (x) + \linear[L](t,x) \big) \in [a-\delta,a+2\delta] \Big) +  \bP \Big( \Big| \Re \big( \varphi_L (t,x) \big) \Big| \geq \delta \Big). \label{measure:eq-density-1} 
\end{align}
In the following, we restrict ourselves to $0< t \leq 1$ and estimate the two terms in \eqref{measure:eq-density-1} separately. To estimate the first term in \eqref{measure:eq-density-1}, we note that $\psi_L^{(0)}$ and $\linear[L]$ are independent. Furthermore, we note that $\Re\linear[L](t,x)$ is a Gaussian random variable whose variance is comparable to
\begin{equation*}
\int_{[0,t]} \int_\R \bigg( \frac{1}{\sqrt{4\pi (t-s)}} e^{-(t-s)} e^{-\frac{|x-y|^2}{4(t-s)}} \bigg)^2 \dy \ds \sim t^{\frac{1}{2}}.    
\end{equation*}
From this, we obtain that 
\begin{align*}
   &\, \bP \Big( \Re \big( e^{t(\Delta-1)} \psi_L^{(0)} (x) + \linear[L](t,x) \big) \in [a-\delta,a+2\delta] \Big) \\
   \leq& \, \sup_{a^\prime \in \R} \bP \Big( \Re \big(  \linear[L](t,x) \big) \in [a^\prime,a^\prime+3\delta] \Big) \lesssim t^{-\frac{1}{4}} \delta.
\end{align*}
We now turn to the second term in \eqref{measure:eq-density-1}. By translation invariance, it suffices to treat the case $x=0$. 
Using $t\in [0,1]$ and Lemma \ref{measure:lem-heat-CE}, we obtain that 
\begin{equation*}
|\varphi_L(t,0)| \leq \big\| \varphi_L(t) \big\|_{\CE^p(\R)}
\leq 2e^{2p^2} \big\| |\psi_L(s)|^{p-1} \psi_L(s) \big\|_{L_s^1 \CE^p_x ([0,t]\times \R)} \leq 2 e^{2p^2}  t \,  \big\| \psi_L(s)\big\|_{L_s^\infty \CE^1_x([0,1]\times \R)}^p.
\end{equation*}
The $L_s^\infty \CE_x^1$-norm above can be bounded by 
\begin{equation*}
\big\| \psi_L(s)\big\|_{L_s^\infty \CE^1_x([0,1]\times \R)} \leq \sum_{R\in \dyadic} e^{-\frac{1}{2} R} 
\big\| \psi_L(s) \big\|_{L_s^\infty L_x^\infty([0,1]\times [-R,R])}.
\end{equation*}
We now let $C^\prime\geq1 $ and $c^\prime>0$ be sufficiently large and small constants depending only on $p$, respectively, whose precise value may change from line to line. Using a union bound and Corollary~\ref{measure:cor-Ltx-bound}, we then obtain for all $\lambda \geq 1$ that 
\begin{align*}
\bP \Big( \big\| \psi_L(s)\big\|_{L_s^\infty \CE^1_x([0,1]\times \R)} \geq C^\prime \lambda^{\frac{1}{2}} \Big)
&\leq \sum_{R\in \dyadic} \bP \Big( e^{-\frac{1}{2} R} 
\big\| \psi_L(s) \big\|_{L_s^\infty L_x^\infty([0,1]\times [-R,R]} \geq 4C^\prime e^{-\frac{1}{4} R} \lambda^{\frac{1}{2}} \Big) \\
&\leq C^\prime \sum_{R\in \dyadic} \exp\Big(-c^\prime e^{\frac{R}{2}} \lambda \Big) \leq C^\prime e^{-c^\prime \lambda}.
\end{align*}
As a result, we obtain that 
\begin{align*}
\bP \Big( \Big| \Re \big( \varphi_L (t,0) \big) \Big| \geq \delta \Big)
\leq \bP \Big( \big\| \psi_L(s)\big\|_{L_s^\infty \CE^1_x([0,1]\times \R)}^p \geq  \big( 2 e^{2p^2} t\big)^{-1} \delta \Big)
\leq C^\prime \exp\Big( - c^\prime (t^{-1} \delta)^{\frac{2}{p}} \Big). 
\end{align*}
After choosing $t \sim \delta^{4(1-\kappa)}$, we obtain the desired estimate \eqref{measure:eq-density}. 
\end{proof}

In the next lemma, we prove an estimate that quantifies the weak convergence of $\mu_L$ to $\mu$. To make this quantitative estimate, we introduce the Wasserstein distance. For any two probability measures $\mu$ and $\nu$ on $\CE^\theta(\R)$, it is defined as 
\begin{equation}\label{measure:eq-def-Wasserstein}
\Wasserstein(\mu,\nu) = \inf_{\gamma \in \Gamma(\mu,\nu)} \int \| \phi -  \psi \|_{\CE^\theta(\R)} \dgamma(\phi,\psi). 
\end{equation}
In \eqref{measure:eq-def-Wasserstein}, $\Gamma(\mu,\nu)$ is the set of couplings of $\mu$ and $\nu$. That is, $\Gamma(\mu,\nu)$ is the set of all measures on the product space $\CE^\theta(\R)\times \CE^\theta(\R)$ whose first and second marginal are given by $\mu$ and $\nu$, respectively.

\begin{lemma}[Wasserstein-distance]\label{measure:lem-Wasserstein}
Let $p>1$ and let $0<\theta \leq \frac{1}{4}$. Furthermore, let $C\geq 1$ and $c>0$ be sufficiently large and small depending on $p$ and $\theta$, respectively. For all $K,L\in \dyadic$ satisfying $K\geq L$, it then holds that
\begin{equation}\label{measure:eq-Wasserstein}
\Wasserstein(\mu_K,\mu_L) \leq C e^{-cL}.
\end{equation}
In particular, the infinite-volume Gibbs measure $\mu$ can be defined as the unique weak limit of the finite-volume Gibbs measures $\mu_L$ and satisfies
\begin{equation}\label{measure:eq-Wasserstein-limit}
\Wasserstein(\mu,\mu_L) \leq C e^{-cL}.
\end{equation}
\end{lemma}

\begin{proof} 
It suffices to prove the estimate \eqref{measure:eq-Wasserstein}, since it directly implies the existence and uniqueness of the weak limit $\mu$ and the limiting estimate \eqref{measure:eq-Wasserstein-limit}.
In the following proof, we let $C=C_p \geq 1$ be a sufficiently large constant whose precise value may change from line to line. We let $\gamma_{K,L}$ be any coupling of $\mu_K$ and $\mu_L$ such that
\begin{equation}\label{measure:eq-Wasser-1}
\int \big\| \phi_K - \phi_L \big\|_{\CE^\theta(\R)} \dgamma_{K,L}(\phi_K,\phi_L) \leq 2 \Wasserstein(\mu_K,\mu_L). 
\end{equation}
The idea behind our argument is to use Langevin dynamics to construct a new coupling of $\mu_K$ and $\mu_L$ out of $\gamma_{K,L}$ which, unless $\gamma_{K,L}$ already witnesses \eqref{measure:eq-Wasserstein}, significantly improves on $\gamma_{K,L}$. In our estimates of the new coupling, we heavily rely on the convexity of the potential $ z \mapsto \frac{1}{p+1} |z|^{p+1}$. \\

To construct the new coupling, we let $(\Omega,\mathcal{F},\bP)$ be a sufficiently rich probability space. We let $\phi_K$ and $\phi_L$ be random functions satisfying $\Law_\bP((\phi_K,\phi_L))=\gamma_{K,L}$. Furthermore, we let $\zeta$ be a space-time white noise which is independent of $(\phi_K,\phi_L)$. We also let $\zeta_K$ and $\zeta_L$ be the $2\pi K$ and $2\pi L$-periodic space-time white noises that agree with $\zeta$ on the space-time cylinders $[0,\infty)\times [-\pi K,\pi K]$ and $[0,\infty)\times [-\pi L,\pi L]$, respectively. We define $\psi_K$ and $\psi_L$ as the solutions of 
\begin{alignat}{3}
(\partial_t +1 - \Delta) \psi_K &= - |\psi_K|^{p-1} \psi_K + \sqrt{2} \zeta_K,  \qquad & \qquad \psi_K(0)=\phi_K, \label{measure:eq-Wasser-2} \\
(\partial_t +1 - \Delta) \psi_L &= - |\psi_L|^{p-1} \psi_L + \sqrt{2} \zeta_L,  \qquad & \qquad \psi_L(0)=\phi_L. \label{measure:eq-Wasser-3}
\end{alignat}
Due to the invariance of the Gibbs measures $\mu_K$ and $\mu_L$ under \eqref{measure:eq-Wasser-2} and \eqref{measure:eq-Wasser-3}, respectively, it follows that $\Law_\bP(\psi_K(t))=\mu_K$ and $\Law_\bP(\psi_L(t))=\mu_L$ for all $t\geq 0$. In particular, it holds that 
\begin{equation*}
\Law_\bP\big( (\psi_K(t),\psi_L(t)) \big) \in \Gamma(\mu_K,\mu_L)
\end{equation*}
for all $t\geq 0$. In order to estimate the difference between $\psi_K$ and $\psi_L$, we introduce the following variables:
\begin{enumerate}[label=(\roman*)]
\item We define $\varphi_{K,L} := \psi_K - \psi_L$, i.e., we define $\varphi_{K,L}$ as the difference between $\psi_K$ and $\psi_L$. 
\item We define  $\lineardif[K,L]$ as the solution of $(\partial_t +1 - \Delta) \lineardif[K,L]\, = \sqrt{2} (\zeta_K-\zeta_L)$ with initial data $\lineardif[K,L] (0)=0$. 
\item We define the nonlinear remainder $\varrho_{K,L} := \varphi_{K,L} - \lineardif[K,L]$.
\item For similar reasons as in the proof of Lemma \ref{measure:lem-pointwise}, we define $\chi_{K,L} := |\varrho_{K,L}|^2$.
\end{enumerate}
From the definition of the linear stochastic object $\lineardif[K,L]$, together with the fact that $\zeta_K$ and $\zeta_L$ agree on $[-\pi L,\pi L]$, it follows that 
\begin{equation}\label{measure:eq-Wasser-5}
\lineardif[K,L](t,x) = \sqrt{2} \int_{[0,t]} \int_{\R \backslash [-\pi L,\pi L]} \frac{1}{\sqrt{4\pi(t-s)}} e^{-(t-s)} e^{-\frac{|x-y|^2}{4(t-s)}} \big( \zeta_K(s,y)- \zeta_L(s,y) \big) \, \dy \ds.
\end{equation}
Furthermore, using the definitions of $\varphi_{K,L}$, $\varrho_{K,L}$, and $\chi_{K,L}$, one obtains that 
\begin{equation*}
(\partial_t + 1 - \Delta ) \varrho_{K,L} = - \big( |\psi_K|^{p-1} \psi_K - |\psi_L|^{p-1} \psi_L \big) 
\end{equation*}
and 
\begin{equation}\label{measure:eq-Wasser-6}
(\partial_t + 2 - \Delta ) \chi_{K,L} = - |\partial_x \varrho_{K,L}|^2 -  2 \Re \big( \overline{\varrho_{K,L}} \, \big( |\psi_K|^{p-1} \psi_K - |\psi_L|^{p-1} \psi_L \big) \big).
\end{equation}
For any convex function $V\colon \R^n \rightarrow \R$, where $n\in \mathbb{N}$, we have that $\langle \nabla V(x)- \nabla V(y),x-y \rangle \geq 0$ for all $x,y\in \R^n$. Since $z\in \C \mapsto \frac{1}{p+1} |z|^{p+1}$ is convex, it therefore follows that 
\begin{equation*}
\Re \big( \big( |\psi_K|^{p-1} \psi_K - |\psi_L|^{p-1} \psi_L \big)  \overline{(\psi_K - \psi_L)}  \big) \geq 0. 
\end{equation*}
Together with \eqref{measure:eq-Wasser-6} and the definition of $\varrho_{K,L}$, we then obtain that 
\begin{equation}\label{measure:eq-Wasser-7}
(\partial_t +2 - \Delta) \chi_{K,L} \leq C \big( |\psi_K|+|\psi_L| \big)^p \big| \, \lineardif[K,L] \big|.  
\end{equation}
By using both \eqref{measure:eq-Wasser-7} and the non-negativity of the heat-kernel, this yields 
\begin{equation}
\chi_{K,L}(t,x) \leq e^{t(\Delta-2)} \chi_{K,L}(0) + C \int_0^t  e^{(t-s) (\Delta-2)} \Big( \big( |\psi_K|+|\psi_L| \big)^p \big| \, \lineardif[K,L] \big|  \Big) \ds . 
\end{equation}
Using Lemma \ref{measure:lem-heat-CE}, we then obtain that 
\begin{align*}
\big\| \chi_{K,L}(t) \big\|_{\CE^{2\theta}(\R)}
&\leq \big\| e^{t(\Delta-2)} \chi_{K,L}(0) \big\|_{\CE^{2\theta}(\R)} + C \int_0^t \Big\| e^{(t-s) (\Delta-2)} \Big( \big( |\psi_K|+|\psi_L| \big)^p \big| \, \lineardif[K,L] \big|  \Big) \Big\|_{\CE^{2\theta}(\R)} \ds \\
&\leq 2 e^{-(2-4\theta^2)t} \|  \chi_{K,L}(0) \|_{\CE^{2\theta}(\R)} + C \int_0^t  e^{-(2-4\theta^2)(t-s)} \big\|   \big( |\psi_K|+|\psi_L| \big)^p  \, \lineardif[K,L]   \big\|_{\CE^{2\theta}(\R)} \ds \\
&\leq 2 e^{-(2-4\theta^2)t} \big\| \chi_{K,L}(0) \big\|_{\CE^{2\theta}(\R)} + C \, \big\| \big( |\psi_K|+|\psi_L| \big)^p\,  \lineardif[K,L] \big\|_{L^\infty_s \CE^{2\theta}([0,t]\times \R)}. 
\end{align*}
In the last inequality we used our assumption $0<\theta\leq 1/4$ to bound the integral of $e^{-(2-4\theta^2)(t-s)}$ in $s$. 
Using the definition of $\chi_{K,L}$, $\varrho_{K,L}$, and $\varphi_{K,L}$, the trivial identity $\| |\varphi|^2\|_{\CE^{2\theta}(\R)}=\| \varphi \|_{\CE^\theta(\R)}^2$, and the sub-additivity of the square-root function, it then follows that
\begin{equation}\label{measure:eq-Wasser-8}
\begin{aligned}
&\,\big\| \psi_K(t) - \psi_L(t) \big\|_{\CE^\theta(\R)} \\ 
\leq&\, \sqrt{2} e^{-(1-2\theta^2)t} \big\| \phi_K - \phi_L \big\|_{\CE^\theta(\R)} + \big\| \, \lineardif[K,L](t)\, \big\|_{\CE^\theta(\R)}
+ C \, \big\| \big( |\psi_K|+|\psi_L| \big)^p\,  \lineardif[K,L] \big\|_{L^\infty_s \CE^{2\theta}([0,t]\times \R)}^{\frac{1}{2}}.
\end{aligned}
\end{equation}
Since the laws of both $(\psi_K(t),\psi_L(t))$ and $(\phi_K,\phi_L)$ are couplings of the Gibbs measures $\mu_K$ and $\mu_L$, and the latter coupling satisfies \eqref{measure:eq-Wasser-1}, we then obtain from \eqref{measure:eq-Wasser-8} that 
\begin{align*}
&\, \Wasserstein(\mu_K,\mu_L) \\
\leq&\, \E \big[\big\| \psi_K(t) - \psi_L(t) \big\|_{\CE^\theta(\R)} \big] \\
\leq&\, \sqrt{2} e^{-(1-2\theta^2) t} \E \big[ \big\| \phi_K - \phi_L \big\|_{\CE^\theta(\R)}  \big] 
+ \E \Big[ \big\| \, \lineardif[K,L](t)\, \big\|_{\CE^\theta(\R)}  
+ C \, \big\| \big( |\psi_K|+|\psi_L| \big)^p\,  \lineardif[K,L] \big\|_{L^\infty_s \CE^{2\theta}([0,t]\times \R)}^{\frac{1}{2}} \Big] \\
\leq&\, 2\sqrt{2}  e^{-(1-2\theta^2)t} \Wasserstein(\mu_K,\mu_L) + \E \Big[ \big\| \, \lineardif[K,L](t)\, \big\|_{\CE^\theta(\R)}  
+ C \, \big\| \big( |\psi_K|+|\psi_L| \big)^p\,  \lineardif[K,L] \big\|_{L^\infty_s \CE^{2\theta}([0,t]\times \R)}^{\frac{1}{2}} \Big].
\end{align*}
By choosing $t=4$, which is such that $2\sqrt{2}e^{-(1-2\theta^2)t}\leq 1/2$, and using a kick-back argument, we then obtain
\begin{equation*}
\Wasserstein(\mu_K,\mu_L) \leq C \E \Big[ \big\| \, \lineardif[K,L]\, \big\|_{L_t^\infty \CE^\theta([0,4]\times \R)}  
+  \, \big\| \big( |\psi_K|+|\psi_L| \big)^p\,  \lineardif[K,L] \big\|_{L^\infty_t \CE^{2\theta}([0,4]\times \R)}^{\frac{1}{2}} \Big].
\end{equation*}
As a result, it only remains to prove that
\begin{equation}\label{measure:eq-Wasserstein-9}
\E \Big[ \big\| \, \lineardif[K,L] \big\|_{L_t^\infty \CE^\theta([0,4]\times \R)}  
+  \, \big\| \big( |\psi_K|+|\psi_L| \big)^p\,  \lineardif[K,L] \big\|_{L^\infty_t \CE^{2\theta}([0,4]\times \R)}^{\frac{1}{2}} \Big] \leq C e^{-cL}.
\end{equation}
This follows easily from Corollary \ref{measure:cor-Ltx-bound} and \eqref{measure:eq-Wasser-5}, and we omit the remaining details.
\end{proof}

Equipped with the previous lemmas, we are now ready to prove the main result of this subsection.

\begin{proof}[Proof of Proposition \ref{measure:prop-coupling}:]
The statement in Proposition \ref{measure:prop-coupling} follows directly from our quantitative version of the Skorokhod representation theorem, i.e., Proposition \ref{sk:prop-main}. The three assumptions in Proposition \ref{sk:prop-main} are satisfied due to Lemma \ref{measure:lem-Gaussian} (and Lemma \ref{prelim:lem-Hoelder-LWP}), Lemma \ref{measure:lem-density}, and Lemma~\ref{measure:lem-Wasserstein}, respectively.
\end{proof}

\section{Uniform estimates for the nonlinear Schrödinger equations}\label{section:uniform}

In the previous section, we obtained uniform bounds for the samples $\phi_L$ drawn from the Gibbs measures~$\mu_L$. Using the invariance of $\mu_L$, we now upgrade them to uniform bounds of the corresponding solution $u_L$ of \eqref{intro:eq-NLS}. The idea to use the invariance of Gibbs measures to bound solutions of \eqref{intro:eq-NLS} was first used in the periodic setting in \cite{B94,B96}, and later in the infinite-volume setting in \cite{B00}. 

\begin{proposition}[\protect{$C_{t,x}^0$-norms}]\label{uniform:prop-linfty}
Let $p>1$. Let $\mbox{$C=C_p\geq 1$}$ and $c=c_p>0$ be sufficiently large and small constants, respectively.
For all $L\geq 1$, let $\mbox{$\phi_L \colon (\Omega \times \R,\mathcal{F}\times \mathcal{B}(\R))\rightarrow \C$}$ be a random function satisfying $\operatorname{Law}(\phi_L)=\mu_L$ and let $u_L$ be the unique global solution of~\eqref{intro:eq-NLS} with initial data $u_L(0)=\phi_L$.  Then, it holds for all $T\geq 1$, $R\geq 10$, and $\lambda >0$ that
\begin{equation}\label{uniform:eq-linfty}
\sup_{L\geq 10} \bP \Big(  \| u_L \|_{C_t^0 C_x^0([-T,T]\times [-R,R])}  > C  ( \log(T+R) +\lambda )^{\frac{2}{p+3}} + C (TR)^{-20} \lambda^2  \Big) \leq C e^{-c\lambda}.
\end{equation}
\end{proposition}

The estimate \eqref{uniform:eq-linfty} can be slightly improved, since the $(TR)^{-20} \lambda^2 $-term is certainly not optimal. Due to the $R^{-20}$-factor, however, this term is completely irrelevant in our application of \eqref{uniform:eq-linfty}, and we therefore do not pursue this further.

In order to deal with the supremum over $t\in [-T,T]$ in \eqref{uniform:eq-linfty}, we want to make use of Kolmogorov's continuity theorem (see Lemma \ref{prelim:lem-kolmogorov}). In order to use Kolmogorov's continuity theorem, however, we need control of the Hölder-norms of $u_L$. Since estimates of the Hölder-norms will also be useful in Section \ref{section:difference}, we record them in a separate proposition. 

\begin{proposition}[\protect{$C_t^\alpha C_x^\beta$-norms}]\label{uniform:prop-hoelder}
Let $p>1$. Let $C=C_p\geq 1$ and $c=c_p>0$ be sufficiently large and small constants, respectively. Furthermore, let $\alpha,\beta\in [0,1)$ satisfy $2\alpha+\beta<\frac{1}{2}$. 
For all $L\geq 1$, let $\phi_L \colon (\Omega \times \R,\mathcal{F}\times \mathcal{B}(\R))\rightarrow \C$ be a random function satisfying $\operatorname{Law}(\phi_L)=\mu_L$ and let $u_L$ be the unique global solution of~\eqref{intro:eq-NLS} with initial data $u_L(0)=\phi_L$.  Then, it holds for all $T\geq 1$, $R\geq 10$, and $\lambda >0$ that
\begin{equation}\label{uniform:eq-hoelder}
\sup_{L\geq 10} \bP \Big(  \| u_L \|_{C_t^\alpha C_x^\beta([-T,T]\times [-R,R])}  > C  ( \log(T+R) +\lambda )^2\Big) \leq C e^{-c\lambda}.
\end{equation}
\end{proposition}

As can be seen from the proof below, the exponent $2$ in $( \log(T+R) +\lambda )^2 $ is certainly not optimal, but it is more than sufficient for our purposes. In contrast, the condition $2\alpha+\beta<\frac{1}{2}$ is likely optimal. The reason is that $\phi_L$ has regularity $\frac{1}{2}-$ and that, due to the scaling-symmetry of~\eqref{intro:eq-NLS}, time-derivatives of $u_L$ cost twice as much as spatial-derivatives of $u_L$.

\begin{remark}\label{uniform:rem-kernel-estimates}
While the proof of \eqref{uniform:eq-hoelder} follows the same strategy as in \cite[Section 2]{B00}, it improves on one of the technical aspects of \cite{B00}. Instead of the Duhamel integral formulation of \eqref{intro:eq-NLS}, we directly work with a frequency-truncated version of \eqref{intro:eq-NLS}. Due to this, we do not rely on the estimates of the kernel of $P_{\leq N} e^{it\Delta}$ that were derived in \cite[(2.10)-(2.22)]{B00}.
\end{remark}

\begin{proof}[Proof of Proposition \ref{uniform:prop-hoelder}:]
For future use, we choose $\gamma \in [0,1)$ satisfying $\alpha<\gamma$ and $\beta+2\gamma<\frac{1}{2}$. Due to Lemma \ref{prelim:lem-maximum-tail}, it suffices to prove that 
\begin{equation}\label{uniform:eq-hoelder-p1} 
\sup_{t_0 \in [-T,T]} \sup_{x_0\in [-R,R]} 
\bP  \Big(  \| u_L \|_{C_t^\alpha C_x^\beta([t_0,t_0+1]\times [x_0-1,x_0+1])}  > C^\prime \lambda^2 \Big) \leq C^\prime e^{-c^\prime\lambda},
\end{equation}
where $C^\prime=C^\prime_p$ and $c^\prime=c^\prime_p$ are constants. We now note that the law of $u_L$ is invariant under space-time translations, which follows from the invariance of the Gibbs measure $\mu_L$ under \eqref{intro:eq-NLS} and the invariance of $\mu_L$ under spatial translations. As a result, it suffices to estimate the probability in  \eqref{uniform:eq-hoelder-p1} for $t_0=0$ and $x_0=0$, i.e., it suffices to prove that 
\begin{equation}\label{uniform:eq-hoelder-p2}  
\bP  \Big(  \| u_L \|_{C_t^\alpha C_x^\beta([0,1]\times [-1,1])}  > C^\prime \lambda^2 \Big) \leq C^\prime e^{-c^\prime\lambda}.
\end{equation}
Using the equivalence of tail and moment estimates (see e.g. \cite[Propositions 2.5.2 and 2.7.1]{V18}), it then suffices to prove for all $r\geq 1$ that 
\begin{equation}\label{uniform:eq-hoelder-p3}
\E \Big[ \| u_L \|_{C_t^\alpha C_x^\beta([0,1]\times [-1,1])}^r \Big]^{\frac{1}{r}} \lesssim_{\alpha,\beta,\gamma}  r^2. 
\end{equation}
Due to Hölder's inequality, it further suffices to prove \eqref{uniform:eq-hoelder-p3} for $r\geq r_{\alpha,\beta,\gamma}$, where $r_{\alpha,\beta,\gamma}$ is sufficiently large depending on $\alpha$, $\beta$, and $\gamma$. In particular, we may therefore assume that $\alpha+\frac{1}{r}<\gamma$. Using Kolmogorov's continuity theorem (Lemma \ref{prelim:lem-kolmogorov}), it then suffices to show that 
\begin{equation}\label{uniform:eq-hoelder-3}
\sup_{0\leq s < t \leq 1}  \E \Bigg[ \Bigg( \frac{\| u_L(t)-u_L(s)\|_{C_x^\beta([-1,1])}}{|t-s|^\gamma}\Bigg)^r  \Bigg]^{\frac{1}{r}} \lesssim_{\alpha,\beta,\gamma}  r^2. 
\end{equation}
To this end, we use a dyadic decomposition and estimate the left-hand side of \eqref{uniform:eq-hoelder-3} by 
\begin{equation}\label{uniform:eq-hoelder-4}
\begin{aligned}
&\, \sup_{0\leq s < t \leq 1}  \E \Bigg[ \Bigg( \frac{\| u_L(t)-u_L(s)\|_{C_x^\beta([-1,1])}}{|t-s|^\gamma}\Bigg)^r  \Bigg]^{\frac{1}{r}} \\
\lesssim&\, \sup_{0\leq s < t \leq 1} \sum_{N\in \dyadic}  |t-s|^{-\gamma} 
\E \Big[ \| P_N u_L(t)-P_N u_L(s)\|_{C_x^\beta([-1,1])}^r \Big]^{\frac{1}{r}}.
\end{aligned}
\end{equation}
We now estimate the moments on the right-hand side of \eqref{uniform:eq-hoelder-4} in two different ways. First, using the invariance of the Gibbs measure under \eqref{intro:eq-NLS} and Corollary \ref{measure:cor-Gaussian}, we obtain that 
\begin{equation}\label{uniform:eq-hoelder-5} 
\begin{aligned}
 \E \Big[ \| P_N u_L(t)-P_N u_L(s)\|_{C_x^\beta([-1,1])}^r \Big]^{\frac{1}{r}}
&\lesssim \E \Big[ \| P_N u_L(t)\|_{C_x^\beta([-1,1])}^r \Big]^{\frac{1}{r}}
+ \E \Big[ \| P_N u_L(s)\|_{C_x^\beta([-1,1])}^r \Big]^{\frac{1}{r}} \\
&\lesssim \E \Big[ \| P_N \phi_L \|_{C_x^\beta([-1,1])}^r \Big]^{\frac{1}{r}}
\lesssim N^{-\frac{1}{2}+\beta} (1+\log(N))^{\frac{1}{2}} r^{\frac{1}{2}}.
\end{aligned}
\end{equation}
Second, using \eqref{intro:eq-NLS}, we have that 
\begin{align*}
&\, \| P_N u_L(t)-P_N u_L(s)\|_{C_x^\beta([-1,1])} \\
\leq&\, \int_s^t \| P_N \Delta u_L (t^\prime) \|_{C_x^\beta([-1,1])} \dt^\prime
+ \int_s^t \| P_N ( |u_L|^{p-1} u_L )(t^\prime) \|_{C_x^\beta([-1,1])} \dt^\prime.
\end{align*}
Together with Minkowki's inequality and the invariance of the Gibbs measure $\mu_L$ under \eqref{intro:eq-NLS}, it then follows that
\begin{equation}\label{uniform:eq-hoelder-6} 
\begin{aligned}
 &\, \E \Big[ \| P_N u_L(t)-P_N u_L(s)\|_{C_x^\beta([-1,1])}^r \Big]^{\frac{1}{r}} \\
 \leq&\, \int_s^t \E \Big[ \| P_N \Delta u_L (t^\prime) \|_{C_x^\beta([-1,1])}^r \Big]^{\frac{1}{r}} \dt^\prime
+ \int_s^t \E \Big[ \| P_N ( |u_L|^{p-1} u_L )(t^\prime) \|_{C_x^\beta([-1,1])}^r \Big]^{\frac{1}{r}} \dt^\prime \\
=&\, |t-s| \E \Big[ \| P_N \Delta \phi_L \|_{C_x^\beta([-1,1])}^r \Big]^{\frac{1}{r}}
+ |t-s| \E \Big[ \| P_N ( |\phi_L|^{p-1} \phi_L ) \|_{C_x^\beta([-1,1])}^r \Big]^{\frac{1}{r}}.
\end{aligned}
\end{equation}
The first term in \eqref{uniform:eq-hoelder-6} is bounded using Corollary \ref{measure:cor-Gaussian}. The second term in \eqref{uniform:eq-hoelder-6} is bounded by first using Lemma \ref{prelim:lem-Hoelder-LWP}, which also eliminates the Littlewood-Paley operator $P_N$, and then using Theorem \ref{intro:thm-measure}. In total, it then follows that 
\begin{equation}\label{uniform:eq-hoelder-6p}
\begin{aligned}
   \eqref{uniform:eq-hoelder-6} \lesssim&\, N^\beta |t-s| \Big( N^{\frac{3}{2}} (1+\log(N))^{\frac{1}{2}} r^{\frac{1}{2}} + r^{\frac{2p}{p+3}} \Big) \lesssim |t-s| N^{\frac{3}{2}+\beta} (1+\log(N))^{\frac{1}{2}} r^{2}. 
\end{aligned}
\end{equation}

In the last inequality, we also used that $(2p)/(p+3)\leq 2$. By combining \eqref{uniform:eq-hoelder-4}-\eqref{uniform:eq-hoelder-6p}, we then obtain that 
\begin{equation}\label{uniform:eq-hoelder-7}
\begin{aligned}
  &\, \sup_{0\leq s < t \leq 1}  \E \Bigg[ \Bigg( \frac{\| u_L(t)-u_L(s)\|_{C_x^\beta([-1,1])}}{|t-s|^\gamma}\Bigg)^r  \Bigg]^{\frac{1}{r}}  \\
  \lesssim&\, \sup_{0\leq s < t \leq 1} \sum_{N\geq 1} N^\beta |t-s|^{-\gamma} \min \Big( N^{-\frac{1}{2}}, |t-s| N^{\frac{3}{2}} \Big) (1+\log(N))^{\frac{1}{2}} r^2. 
\end{aligned}
\end{equation}
Using the parameter condition $\beta+2\gamma<\frac{1}{2}$, this implies \eqref{uniform:eq-hoelder-p3}, which completes the proof.
\end{proof}

\begin{proof}[Proof of Proposition \ref{uniform:prop-linfty}:]
We let $\Lambda_{T,R}\subseteq [-T,T]$ be a grid with spacing $(TR)^{-100}$. Using the definition of the Hölder norm, it then follows that 
\begin{equation*}
\| u_L \|_{C_t^0 C_x^0([-T,T]\times [-R,R])}
\leq \max_{t\in \Lambda_{T,R}} \|u_L(t)\|_{C_x^0([-R,R])}
+ (TR)^{-25} \| u_L \|_{C_t^{1/4} C_x^0([-T,T]\times [-R,R])}.
\end{equation*}
It therefore suffices to show that
\begin{align}
\bP \Big( \max_{t\in \Lambda_{T,R}} \| u_L(t) \|_{C_x^0([-R,R])} > C (\log(T+R)+\lambda)^{\frac{2}{p+3}} \Big) &\leq \tfrac{1}{2} C e^{-c\lambda} \label{uniform:eq-linfty-1} \\
\text{and} \qquad 
\bP \Big( \| u_L(t) \|_{C_t^{1/4} C_x^0([-R,R])} > C (TR)^5 \lambda^2 \Big) 
&\leq \tfrac{1}{2} C e^{-c\lambda}. \label{uniform:eq-linfty-2}
\end{align}
The first estimate \eqref{uniform:eq-linfty-1} follows directly from Theorem \ref{intro:thm-measure}, Lemma \ref{prelim:lem-maximum-tail}, and the invariance of the Gibbs measure $\mu_L$ under \eqref{intro:eq-NLS}. The second estimate \eqref{uniform:eq-linfty-2} follows directly from Proposition~\ref{uniform:prop-hoelder}. 
\end{proof}

\section{Difference estimates for the nonlinear Schrödinger equations}\label{section:difference}

The goal of this section is to bound the differences of solutions to \eqref{intro:eq-NLS}. 
As in Theorem \ref{intro:thm-dynamics}, we let $u_L$ and $u_{L/2}$ be the solutions of \eqref{intro:eq-NLS} with the initial data $\phi_L$ and $\phi_{L/2}$. We then define $w_L$ as the difference, i.e., 
\begin{equation}\label{difference:eq-w}
w_L := u_L - u_{L/2}.
\end{equation}
From the definition of $w_L$, it  follows that $w_L$ is a solution of the linear Schrödinger equation
\begin{equation}\label{difference:eq-w-equation}
i \partial_t w_L + \Delta w_L = Q_{L}^+w_L + Q_{L}^- w_L.  
\end{equation}
Here, $Q_L^+$ and $Q_L^-$ can be expressed using $F(z,\overline{z})=|z\overline{z}|^{\frac{p-1}{2}}z$ as 
\begin{align}
Q_L^+ &= \int_0^1 (\partial_z F)\Big( u_{L/2} + \theta (u_L - u_{L/2}), \overline{u_{L/2}+\theta (u_L - u_{L/2})} \Big) \mathrm{d}\theta, \label{difference:eq-Q-plus} \\
Q_L^- &= \int_0^1 (\partial_{\overline{z}} F)\Big( u_{L/2} + \theta (u_L - u_{L/2}), \overline{u_{L/2}+\theta (u_L - u_{L/2})} \Big) \mathrm{d}\theta.\label{difference:eq-Q-minus}
\end{align}
If the parameter $p$ is an odd integer, both $Q_L^+$ and $Q_L^-$ are polynomials in $u_L,u_{L/2},\overline{u_L}$, and $\overline{u_{L/2}}$ of degree $p-1$. However, both $Q_L^+$ and $Q_L^-$ are well-defined for general parameters $p>1$. \\

In order to control $w_L$, we will need the a-priori estimates of $u_L$ and $u_{L/2}$ from Section \ref{section:uniform}. For expository purposes, we now introduce the good event $\goodind$, which captures these a-priori estimates.  

\begin{definition}[Good event]\label{difference:def-good-event}
Let $A_0\geq 1$ be a sufficiently large constant and let $\delta>0$ be sufficiently small depending on $\alpha$ from \eqref{intro:eq-as-convergence-estimate}.  For all $L\geq 10$, $R\geq 10$, and $T\geq 1$, we then define the good event 
\begin{equation}\label{difference:eq-good-event}
\begin{aligned}
 \goodind &:= \Big\{ \max_{u=u_L,u_{L/2}}  \Big\| \log\big(T+R+\langle x \rangle\big)^{-\frac{2}{p+3}} u \Big\|_{L_t^\infty L_x^\infty([-T,T]\times \R)} \leq A_0 \Big\} \\
 &\,\,\bigcap \Big\{  \max_{u=u_L,u_{L/2}} \sup_{N\geq 1} N^{\frac{1}{2}-\delta} \Big\| \log\big(T+R+\langle x \rangle\big)^{-\frac{1}{2}} P_{>N} u \Big\|_{L_t^\infty L_x^\infty([-T,T]\times \R)} \leq A_0 \Big\}.
\end{aligned}
\end{equation}
\end{definition}
From the definition of the good event, it directly follows for all $R_1\leq R_2$ that 
\begin{equation}\label{difference:eq-good-inclusion}
\good_{L,T,R_1} \subseteq \good_{L,T,R_2},
\end{equation}
which will be useful in the proof of Lemma \ref{difference:lem-iterated} below. Using our earlier estimates in Section \ref{section:uniform}, we obtain that the good event has high probability.

\begin{lemma}[Probability of the good event]\label{difference:lem-probability-good}
Let $A_1 \geq 1$ be a sufficiently large constant depending on~$A_0$. For all $L\geq 10$, $R\geq 10$, and $T\geq 1$, it then holds that 
\begin{equation*}
\bP \big( \goodind \big) \geq 1 - A_1 (TR)^{-100}.
\end{equation*}
\end{lemma}

\begin{proof}
It suffices to prove for $u=u_{L/2},u_{L}$ that
\begin{align}
\bP \Big( \big\| \log\big(T+R+\langle x \rangle\big)^{-\frac{2}{p+3}} u \big\|_{L_t^\infty L_x^\infty([-T,T]\times \R)}>A_0 \Big) &\leq \tfrac{1}{4} A_1 (TR)^{-100}, \label{difference:eq-probability-good-p1} \\
\bP \Big( \sup_{N\geq 1} N^{\frac{1}{2}-\delta} \big\| \log\big(T+R+\langle x \rangle\big)^{-\frac{1}{2}} P_{>N} u\big\|_{L_t^\infty L_x^\infty([-T,T]\times \R)}>A_0 \Big) &\leq \tfrac{1}{4} A_1 (TR)^{-100}. \label{difference:eq-probability-good-p2} 
\end{align}
We first prove \eqref{difference:eq-probability-good-p1}. By using that $\log(y/2)^{-\frac{2}{p+3}}\leq 2 \log(y)^{-\frac{2}{p+3}}$ for all $y\geq 4$ and decomposing the real line into the regions $|x|\leq R$ and $2^{k-1}R<|x|\leq 2^k R$, where $k\geq 1$, we obtain that 
\begin{equation*}
\big\| \log\big(T+R+\langle x \rangle\big)^{-\frac{2}{p+3}} u \big\|_{L_t^\infty L_x^\infty([-T,T]\times \R)}
\leq 2 \sup_{k\geq 0} \log\big(T+2^k R\big)^{-\frac{2}{p+3}} \big\|  u \big\|_{L_t^\infty L_x^\infty([-T,T]\times [-2^k R,2^k R])}.
\end{equation*}
Let $C=C_p\geq 1$ and $c=c_p>0$ be as in Proposition \ref{uniform:prop-linfty} and let $\lambda=A_0^{\frac{1}{4}} \log(T+2^k R)$. For this choice of $\lambda$, it holds that 
\begin{align*}
C \big(\log(T+2^k R)+\lambda\big)^{\frac{2}{p+3}} + C (2^k TR)^{-20} \lambda^2 
\lesssim_C A_0^{\frac{1}{2}} \log(T+2^k R)^{\frac{2}{p+3}} 
\leq \tfrac{1}{2} A_0 \log(T+2^k R)^{\frac{2}{p+3}}.
\end{align*}
Using Proposition \ref{uniform:prop-linfty}, it therefore follows that 
\begin{equation*}
\bP \Big(  \log(T+2^k R)^{-\frac{2}{p+3}} 
\big\| u \big\|_{L_t^\infty L_x^\infty([-T,T]\times [-2^k R,2^k R])} > \tfrac{1}{2} A_0 \Big) 
\leq C e^{-c\lambda} = C (T+2^k R)^{-cA_0^{\frac{1}{4}}}.
\end{equation*}
Using that $A_0$ is sufficiently large, using a union bound, and summing over $k\geq 0$, this readily implies \eqref{difference:eq-probability-good-p1}. The proof of \eqref{difference:eq-probability-good-p2} is similar to the proof of \eqref{difference:eq-probability-good-p1}, except that we use Proposition \ref{uniform:prop-hoelder} instead of Proposition \ref{uniform:prop-linfty}, and we therefore omit the details.
\end{proof}

From the bounds in \eqref{difference:eq-good-event}, one can obtain several other estimates on the solutions $u_L$ and $u_{L/2}$, the difference $w_L$, as well as the functions $Q_L^+$ and $Q_L^-$. Since some of these estimates will be used repeatedly below, we record them in the following lemma.

\begin{lemma}[Consequences of the good event]\label{difference:lem-good-consequences} 
Let $p\geq 3$, let $L\geq 10$, let $R\geq 10$, and let $T\geq 1$. On the good event $\goodind$ from Definition \ref{difference:def-good-event}, we then have the following estimates:
\begin{align}
\sup_{N\geq 1}  \big\| \log\big( T + R + \langle x \rangle\big)^{-\frac{2}{p+3}} P_{\leq N} w_L \big\|_{L_t^\infty L_x^\infty} &\lesssim 1,  \label{difference:eq-good-w-LWP-sup}\\
\sup_{N\geq 1} N^{\frac{1}{2}-\delta} \big\| \log\big( T + R + \langle x \rangle\big)^{-\frac{1}{2}}   P_{>N} w_L \big\|_{L_t^\infty L_x^\infty} &\lesssim 1, \label{difference:eq-good-w-high} \\ 
\sup_{N\geq 1} \big\| \log\big (T+R+  \langle x \rangle \big)^{-\frac{2(p-1)}{p+3}}  P_{\leq N} Q^{\pm}_L \big\|_{L_t^\infty L_x^\infty}&\lesssim 1,    \label{difference:eq-good-Q-pointwise} \\
\sup_{N\geq 1} \big\| \log\big( T + R + \langle x \rangle\big)^{-\frac{p-1}{2}}  N^{\frac{1}{2}-\delta} P_{>N} Q_L^\pm \big\|_{L_t^\infty L_x^\infty} &\lesssim 1, \label{difference:eq-good-Q-high} \\
\sup_{N\geq 1} N^{-\frac{1}{2}-\delta} \big\| \log\big( T + R + \langle x \rangle\big)^{-\frac{p-1}{2}}   \partial_x P_{\leq N} Q_L^\pm\big\|_{L_t^\infty L_x^\infty} &\lesssim 1,  \label{difference:eq-good-Q-derivative} 
\end{align}
where the $L_t^\infty L_x^\infty$-norms are taken over $[-T,T]\times \R$ and the implicit constants depend on $A_0,\delta$, and~$p$.
\end{lemma}

Since the definition of the $L_t^\infty L_x^\infty$-norm is based on suprema, the suprema over $N$ and the $L_t^\infty L_x^\infty$-norms in \eqref{difference:eq-good-w-LWP-sup}-\eqref{difference:eq-good-Q-derivative} commute. 

\begin{proof}[Proof of Lemma \ref{difference:lem-good-consequences}:]
The first and second estimate \eqref{difference:eq-good-w-LWP-sup} and \eqref{difference:eq-good-w-high} follow from~Corollary \ref{prelim:cor-PN-sup}, Definition~\ref{difference:def-good-event}, and \eqref{difference:eq-w}. The third estimate \eqref{difference:eq-good-Q-pointwise} follows directly from $|Q_L^\pm|\lesssim |u_L|^{p-1}+|u_{L/2}|^{p-1}$, Corollary \ref{prelim:cor-PN-sup}, and Definition~\ref{difference:def-good-event}. To obtain the fourth estimate \eqref{difference:eq-good-Q-high}, we first observe that
\begin{align*}
&\, \, \,\, \big| Q_L^\pm(x) - Q_L^\pm(y) \big| \\
\lesssim&\, \big( |u_L(x)|+|u_L(y)|+|u_{L/2}(x)|+|u_{L/2}(y)|\big)^{p-2} \big( |u_L(x)-u_L(y)|+|u_{L/2}(x)-u_{L/2}(y)|\big),
\end{align*}
which follows directly from \eqref{difference:eq-Q-plus}, \eqref{difference:eq-Q-minus}, and $p\geq 3$. From this, we obtain for all $\alpha \in [0,1)$ and all compact intervals $I\subseteq \R$ that
\begin{equation*}
\big\| Q_L^\pm \big\|_{C_x^\alpha(I)} \lesssim \big( \| u_L \|_{C_x^0(I)} + \|u_{L/2} \|_{C_x^0(I)} \big)^{p-2} \big( \| u_L \|_{C_x^\alpha(I)} + \|u_{L/2} \|_{C_x^\alpha(I)} \big).
\end{equation*}
Together with Lemma \ref{prelim:lem-Hoelder-LWP} and Definition \ref{difference:def-good-event}, we then readily obtain \eqref{difference:eq-good-Q-high}. Finally, the fifth estimate \eqref{difference:eq-good-Q-derivative} follows from Lemma \ref{prelim:lem-Hoelder-LWP} and \eqref{difference:eq-good-Q-high}. 
\end{proof}

Equipped with Definition \ref{difference:def-good-event} and Lemma \ref{difference:lem-good-consequences}, we can now state and prove the main result of this section.

\begin{proposition}[Difference estimate]\label{difference:prop-main}
Let $p\geq 3$, let $\delta>0$ be as in Definition \ref{difference:def-good-event}, and let $A_2=A_2(A_0,A_1,p,\delta)\geq 1$ be sufficiently large. Let $L\geq 10$, let $R\geq 10$, let  $T\geq 1$, and let $\lambda\geq 1$. Furthermore, let $u_L$ and $u_{L/2}$ be as in Theorem \ref{intro:thm-dynamics} and let $w_L:= u_L - u_{L/2}$. Finally, as in \eqref{intro:eq-MR} and \eqref{prelim:eq-sigma}, let 
\begin{equation}\label{difference:eq-mass-sigma}
M_R(w_L(t)):= \int_{\R} \big| P_{\leq R} w_L(t,x)\big|^2 \sigma_R(x) \dx \qquad \text{and} \qquad \sigma_R(x) := e^{-\langle \frac{x}{R} \rangle}. 
\end{equation}
On the good event $\goodind$, it then holds for all $t,t_0\in [-T,T]$ that
\begin{equation}\label{difference:eq-main}
\begin{aligned}
  M_R\big(w_L(t) \big)
\leq  \exp\Big( A_2 \log(T+R)^{\frac{2(p-1)}{p+3}} |t-t_0| \Big) 
\Big( M_R\big(w_L(t_0)\big) + A_2 R^{-1+8\delta} \log(T+R)^p \Big). 
\end{aligned}
\end{equation}
\end{proposition}

\noindent 
We note that, in order for \eqref{difference:eq-main} to be useful in the proof of Theorem \ref{intro:thm-dynamics}, we at least need that
\begin{equation*}
\sup_{t\in [t_0-\tau,t_0+\tau]} \exp\Big( A_2 \log(R)^{\frac{2(p-1)}{p+3}} |t-t_0| \Big) R^{-1+8\delta} \lesssim R^{-\varepsilon},
\end{equation*}
where $\tau>0$ and $\varepsilon>0$ are small constants. For this, it is necessary that $2(p-1)/(p+3)\leq 1$, i.e., that $p\leq 5$, which is the condition from Theorem \ref{intro:thm-dynamics}. We also recall that, as described in Subsection~\ref{section:introduction-results}, the idea behind the proof of Proposition~\ref{difference:prop-main} is to control the growth of $M_R(w_L(t))$ using the $L_t^\infty L_x^\infty$-bounds on $u_L$ and $u_{L/2}$ from Proposition \ref{uniform:prop-linfty} and Gronwall's inequality. 

\begin{proof} Throughout this proof, all implicit constants may depend on $A_0$, $A_1$, $\delta$, and $p$. To simplify the notation, we now omit the subscript $L$ in $w_L$, $Q_L^+$, and $Q_L^-$.  
Using Gronwall's inequality,  \eqref{difference:eq-main} can be reduced to the estimate
\begin{equation}\label{difference:eq-main-differential}
\Big| \tfrac{\mathrm{d}}{\dt} M_R(w(t)) \Big| \lesssim \log(T+R)^{\frac{2(p-1)}{p+3}} M_R(w(t)) + R^{-1+8\delta} \log(T+R)^p.
\end{equation}
In order to prove \eqref{difference:eq-main-differential}, we first obtain from the definition of $M_R$ that  
\begin{align}
\tfrac{\mathrm{d}}{\dt} M_R(w(t)) 
&= 2 \int_\R \Re \big( \overline{P_{\leq R} w} \, \partial_t P_{\leq R} w \big)\sigma_R \, \dx\notag \\
&= - 2 \int_\R \Im \big( \overline{P_{\leq R}w} \, \Delta P_{\leq R} w \big) \sigma_R \, \dx \notag \\
&+ 2 \int_\R \Im \Big( \overline{P_{\leq R}w} \, P_{\leq R} \big( Q_+ w + Q_- \overline{w} \big) \Big) \sigma_R \, \dx \notag \\
&=  2 \int_\R \Im \big( \overline{P_{\leq R}w} \, \partial_x P_{\leq R} w \big) \partial_x \sigma_R \, \dx \label{difference:eq-main-5} \\
&+ 2 \int_\R \Im \Big( \overline{P_{\leq R}w} \, P_{\leq R} \big( Q_+ w + Q_- \overline{w} \big) \Big) \sigma_R \, \dx.  \label{difference:eq-main-6}
\end{align} 
We now estimate \eqref{difference:eq-main-5} and \eqref{difference:eq-main-6} separately.\\

\emph{Case 1: Estimate of \eqref{difference:eq-main-5}.}
Due to the definition of $\sigma_R$, it holds that $|\partial_x \sigma_R| \lesssim R^{-1} \sigma_R$. Together with Cauchy-Schwarz, Young's inequality, and Lemma \ref{prelim:lem-PR-weight}, we then obtain that 
\begin{equation}\label{difference:eq-main-7}
\begin{aligned}
\big| \eqref{difference:eq-main-5} \big| 
&\lesssim R^{-1} \big\| \sqrt{\sigma_R} P_{\leq R}w\big\|_{L_x^2}
\big\| \sqrt{\sigma_R} \partial_x P_{\leq R} w \big\|_{L_x^2}\\ 
&\lesssim R^{-1} \big\| \sqrt{\sigma_R} P_{\leq R}w\big\|_{L_x^2} \Big( 
R \big\| \sqrt{\sigma_R} P_{\leq R}w\big\|_{L_x^2} + R^{-10} \big\| \langle x \rangle^{-10} w \big\|_{L_x^1} \Big) \\
&\lesssim \big\| \sqrt{\sigma_R} P_{\leq R}w\big\|_{L_x^2}^2 + R^{-20} \big\| \langle x \rangle^{-10} w \big\|_{L_x^1}^2. 
\end{aligned}
\end{equation}
The first term in \eqref{difference:eq-main-7} equals $M_R(w(t))$. Using \eqref{difference:eq-good-w-LWP-sup},  the second term in \eqref{difference:eq-main-7} can be bounded by \mbox{$R^{-20} \log(T+R)^{\frac{4}{p+3}}$}. Thus, both terms in \eqref{difference:eq-main-7} lead to acceptable contributions to \eqref{difference:eq-main-differential}. \\

\emph{Case 2: Estimate of \eqref{difference:eq-main-6}.} Using Cauchy-Schwarz, we estimate 
\begin{equation}\label{difference:eq-main-cauchy-schwarz}
\Big| \eqref{difference:eq-main-6} \Big| 
\lesssim \big\| \sqrt{\sigma_R} P_{\leq R} w \big\|_{L_x^2} 
\big\| \sqrt{\sigma_R} P_{\leq R} \big( Q^+ w + Q^- \overline{w} \big) \big\|_{L_x^2}.
\end{equation}
We now use the triangle inequality and that $P_{\leq R}$ commutes with complex conjugation, which allows us to estimate
\begin{equation}\label{difference:eq-main-complex-conjugate}
\big\| \sqrt{\sigma_R} P_{\leq R} \big( Q^+ w + Q^- \overline{w} \big) \big\|_{L_x^2}
\lesssim \max_{Q=Q^+,\overline{Q^-}} \big\| \sqrt{\sigma_R} P_{\leq R} \big( Q w \big) \big\|_{L_x^2}.
\end{equation}
By splitting $w$ into low and high-frequency terms, we also obtain that 
\begin{align}\label{difference:eq-main-splitting}
 \big\| \sqrt{\sigma_R} P_{\leq R} \big( Q w \big) \big\|_{L_x^2}
 \lesssim \big\| \sqrt{\sigma_R} P_{\leq R} \big( Q P_{\leq R} w \big) \big\|_{L_x^2}
 + \big\| \sqrt{\sigma_R} P_{\leq R} \big( Q P_{>R} w \big) \big\|_{L_x^2}.
\end{align}
We now estimate the contributions of $P_{\leq R} w$ and $P_{>R} w$ separately. \\

\emph{Case 2.a: Contribution of $P_{\leq R} w$.} In this case, we show that 
\begin{equation*}
\big\| \sqrt{\sigma_R} P_{\leq R} \big( Q P_{\leq R} w \big) \big\|_{L_x^2} \lesssim \log(T+R)^{\frac{2(p-1)}{p+3}} M_R(w(t))^{\frac{1}{2}} +\log(T+R)^{\frac{2p}{p+3}}  R^{-10}. 
\end{equation*}
Together with \eqref{difference:eq-main-cauchy-schwarz}, \eqref{difference:eq-main-complex-conjugate}, and \eqref{difference:eq-main-splitting}, the contribution of $P_{\leq R} w$ then yields an acceptable contribution to \eqref{difference:eq-main-differential}. 
Using Lemma \ref{prelim:lem-PR-weight}, we obtain that
\begin{equation}
\begin{aligned}\label{difference:eq-main-10} 
  &\, \big\| \sqrt{\sigma_R} P_{\leq R} \big( Q P_{\leq R} w \big) \big\|_{L_x^2} \\
 \lesssim&\, \big\| \sqrt{\sigma_R}  Q P_{\leq R} w \big\|_{L_x^2}+ R^{-10} \big\| \langle x \rangle^{-10} Q P_{\leq R} w \big\|_{L_x^1} \\ 
 \lesssim&\, \big\| \mathbf{1}_{|x|\leq R^2} \sqrt{\sigma_R}  Q P_{\leq R} w \big\|_{L_x^2}
 + \big\| \mathbf{1}_{|x|> R^2} \sqrt{\sigma_R}  Q P_{\leq R} w \big\|_{L_x^2}
 + R^{-10} \big\| \langle x \rangle^{-10} Q P_{\leq R} w \big\|_{L_x^1}.
\end{aligned}
\end{equation}
We first control the first summand in \eqref{difference:eq-main-10}, which is the main term. Using \eqref{difference:eq-good-Q-pointwise}, we obtain that 
\begin{align*}
\big\|  \mathbf{1}_{|x|\leq R^2} \sqrt{\sigma_R}  Q P_{\leq R} w \big\|_{L_x^2}
\lesssim \log(T+R)^{\frac{2(p-1)}{p+3}} \big\| \sqrt{\sigma_R}  P_{\leq R} w \big\|_{L_x^2}
= \log(T+R)^{\frac{2(p-1)}{p+3}} M_R(w(t))^{\frac{1}{2}},
\end{align*}
which is acceptable. We next estimate the second and third summand in \eqref{difference:eq-main-10}, which are minor terms. Using \eqref{difference:eq-good-w-LWP-sup} and \eqref{difference:eq-good-Q-pointwise},  we obtain that 
\begin{align*}
&\, \big\| \mathbf{1}_{|x|> R^2} \sqrt{\sigma_R}  Q P_{\leq R} w \big\|_{L_x^2}
 + R^{-10} \big\| \langle x \rangle^{-10} Q P_{\leq R} w \big\|_{L_x^1} \\
 \lesssim&\, \Big\| \mathbf{1}_{|x|> R^2} \sqrt{\sigma_R} \log(T+R+\langle x \rangle)^{\frac{2p}{p+3}} \Big\|_{L_x^2}
 + R^{-10} \Big\| \langle x \rangle^{-10}  \log(T+R+\langle x \rangle)^{\frac{2p}{p+3}}  \Big\|_{L_x^1}\\
 \lesssim&\, \log(T+R)^{\frac{2p}{p+3}} \Big( \exp\big( - \tfrac{1}{4} R \big) + R^{-10} \Big) \lesssim \log(T+R)^{\frac{2p}{p+3}}  R^{-10},
\end{align*}
which is also acceptable. \\

\emph{Case 2.b: Contribution of $P_{>R}w $.} We show that 
\begin{equation*}
\big\| \sqrt{\sigma_R} P_{\leq R} \big( Q P_{>R} w \big) \big\|_{L_x^2} \lesssim  \log(T+R)^{\frac{2(p-1)}{p+3}} M_R(w(t))^{\frac{1}{2}}+ \log(T+R)^{\frac{p}{2}} R^{-\frac{1}{2}+4\delta}.
\end{equation*}
To simplify the notation below, we define $Q_{\leq R}:= P_{\leq R} Q$ and
$Q_{>R}:= P_{>R} Q$. Furthermore, we let $[P_{\leq R},Q_{\leq R}]$ be the commutator of $P_{\leq R}$ and $Q_{\leq R}$. 
Equipped with this notation, we now split 
\begin{align}
\big\| \sqrt{\sigma_R} P_{\leq R} \big( Q P_{>R} w \big) \big\|_{L_x^2}
&\lesssim 
\big\| \sqrt{\sigma_R} P_{\leq R} \big( Q_{>R} P_{>R} w \big) \big\|_{L_x^2} \label{difference:eq-main-11}\\
&+ \big\| \sqrt{\sigma_R} \, \big[ P_{\leq R}, Q_{\leq R} \big] P_{>R} w  \big\|_{L_x^2} \label{difference:eq-main-12} \\
&+ \big\| \sqrt{\sigma_R} Q_{\leq R} P_{\leq R} P_{>R} w \big\|_{L_x^2}. \label{difference:eq-main-13}
\end{align}
We now estimate \eqref{difference:eq-main-11}, \eqref{difference:eq-main-12}, and \eqref{difference:eq-main-13} separately. Using Lemma \ref{prelim:lem-PR-weight}, \eqref{difference:eq-good-w-high}, and \eqref{difference:eq-good-Q-high}, we obtain 
\begin{align*}
\eqref{difference:eq-main-11}
&\lesssim \big\| \sqrt{\sigma_R}  Q_{>R} P_{>R} w\big\|_{L_x^2(\R)}
+ R^{-10} \big\| \langle x \rangle^{-10} Q_{>R} P_{>R} w\big\|_{L_x^1(\R)} \\
&\lesssim R^{-1+2\delta}  \big\| \sqrt{\sigma_R}  \log(T+R + \langle x \rangle)^{\frac{p}{2}} \big\|_{L_x^2(\R)}
+ R^{-11+2\delta} \big\| \langle x \rangle^{-10}  \log(T+R + \langle x \rangle)^{\frac{p}{2}} \big\|_{L_x^1(\R)} \\
&\lesssim R^{-\frac{1}{2}+2\delta} \log(T+R)^{\frac{p}{2}},
\end{align*}
which is acceptable. Using Lemma \ref{prelim:lem-commutator}, \eqref{difference:eq-good-w-high}, and \eqref{difference:eq-good-Q-derivative}, we estimate 
\begin{align*}
\eqref{difference:eq-main-12} 
&\lesssim R^{-\frac{1}{2}+2\delta} \big\| \langle x \rangle^{-\delta} \partial_x Q_{\leq R} \big\|_{L_x^\infty(\R)} \big\| \langle x \rangle^{-\delta} P_{>R} w \big\|_{L_x^\infty(\R)} \\ 
&\lesssim R^{-\frac{1}{2}+2\delta} R^{\frac{1}{2}+\delta} R^{-\frac{1}{2}+\delta}
\Big\| \langle x \rangle^{-\delta} \log\big( T+R+\langle x\rangle \big)^{\frac{p-1}{2}} \Big\|_{L_x^\infty}
\Big\| \langle x \rangle^{-\delta} \log\big( T+R+\langle x\rangle \big)^{\frac{1}{2}} \Big\|_{L_x^\infty} \\
&\lesssim R^{-\frac{1}{2}+4\delta} \log\big( T+R\big)^{\frac{p}{2}},
\end{align*}
which is acceptable. The remaining term \eqref{difference:eq-main-13} can be treated using a similar argument as in the proof of \eqref{difference:eq-main-10}. Indeed, using \eqref{difference:eq-good-w-LWP-sup} and \eqref{difference:eq-good-Q-pointwise}, we first obtain that 
\begin{align}
\eqref{difference:eq-main-13}
&\lesssim \big\| \mathbf{1}_{|x|\leq R^2} \sqrt{\sigma_R} Q_{\leq R} P_{>R} P_{\leq R} w \big\|_{L_x^2}
+ \big\| \mathbf{1}_{|x|> R^2} \sqrt{\sigma_R} Q_{\leq R} P_{>R} P_{\leq R} w \big\|_{L_x^2} \notag   \\
&\lesssim \log(T+R)^{\frac{2(p-1)}{p+3}}\big\|  \sqrt{\sigma_R}P_{>R} P_{\leq R} w \big\|_{L_x^2}
+ \big\| \mathbf{1}_{|x|> R^2} \sqrt{\sigma_R} \log\big(T+R+\langle x \rangle\big)^{\frac{p}{2}} \big\|_{L_x^2}\label{difference:eq-main-14}
\end{align}
Using $P_{>R}=1-P_{\leq R}$, the triangle inequality, and Lemma \ref{prelim:lem-PR-weight}, the first summand in \eqref{difference:eq-main-14} can be  bounded by 
\begin{equation*}
 \log(T+R)^{\frac{2(p-1)}{p+3}} M_R(w(t))^{\frac{1}{2}} +  \log(T+R)^{\frac{2p}{p+3}}  R^{-10},
\end{equation*}
which is acceptable. Using a direct computation, the second summand in \eqref{difference:eq-main-14} can be bounded by $\log(T+R)^{\frac{p}{2}} e^{-\frac{1}{4}R}$, which is also acceptable.
\end{proof}

Even  under the assumption $p\leq 5$, a direct application of Proposition \ref{difference:prop-main} only allows us to show that $M_R(w_L(t))$ is small on a small time-interval, i.e., a time-interval of size $0<\tau \ll 1$. However, it is possible to show that $M_R(w_L(t))$ is small on a time-interval of size $T\geq 1$ by iterating Proposition~\ref{difference:prop-main}, provided that the $R$-parameter changes in each step of the iteration. In the statement below, $A_0,A_1$, and $A_2$ are as in Definition \ref{difference:def-good-event}, Lemma \ref{difference:lem-good-consequences}, and Proposition \ref{difference:prop-main}, respectively.  

\begin{lemma}[Iterated difference estimate]\label{difference:lem-iterated}
Let $3\leq p\leq 5$, $L\geq 10$, $T\geq 1$, and $R\geq 10$. 
Let $A_3=A_3(A_0,A_1,A_2,\delta,p)$ be sufficiently large, 
let $\tau_0=\tau_0(A_0,A_1,A_2,\delta,p)$ be sufficiently small, and let $J\in \mathbb{N}$ be such that $\tau:=T/J$ satisfies $\tau\leq \tau_0$.
Let $R_j$, where $0\leq j \leq J$, be defined iteratively as 
\begin{equation*}
R_J := R \qquad \text{and} \qquad R_{j-1} := R_j^2~~\text{ for all } 1\leq j \leq J.
\end{equation*}
Finally, assume that 
\begin{equation}\label{difference:eq-iterated-initial}
M_{R_0}\big( w_L(0) \big) \leq A_3 R_0^{-\frac{1}{2}}.
\end{equation}
On the good event $\goodind$, it then holds that 
\begin{equation}\label{difference:eq-iterated-final}
\sup_{t\in [0,T]} M_R(w_L(t))\leq A_3^{J+2} R^{-\frac{1}{2}}.
\end{equation}
\end{lemma}

\begin{proof}[Proof of Lemma \ref{difference:lem-iterated}:]
In the following proof, all implicit constants are allowed to depend on $\delta$, $p$, $A_0$, $A_1$, and $A_2$, but not on $J$ or $A_3$. As in the proof of Proposition \ref{difference:prop-main}, we write $w$ instead of $w_L$. We may assume that $T\leq R$, since otherwise the desired estimate \eqref{difference:eq-iterated-final} easily follows from \eqref{difference:eq-good-w-LWP-sup} in Lemma \ref{difference:lem-good-consequences}. In particular, it then holds that $R_j \geq T$ for all $0\leq j \leq J$. We define the sequences of times $t_j:= j\tau$, where $0\leq j \leq J$.  To simplify the notation, we also set $t_{-1}:=t_0=0$. We now prove by induction that, for all $0\leq j\leq J$,
\begin{equation}\label{difference:eq-iterated-1}
\sup_{t\in [t_{j-1},t_j]} M_{R_j}(w(t)) \leq A_3^{j+1} R_j^{-\frac{1}{2}}.
\end{equation}
\emph{Base case: $j=0$.} By definition, it holds that 
\begin{equation*}
\sup_{t\in [t_{-1},t_0]} M_{R_0}(w(t))=M_{R_0}(w(0)).
\end{equation*}
Thus, the desired estimate follows directly from our condition \eqref{difference:eq-iterated-initial}.\\ 

\noindent\emph{Induction step: $j-1\rightarrow j$.}
We first recall from \eqref{difference:eq-good-inclusion} that $\goodind\subseteq \good_{L,T,R_j}$. 
Using \eqref{prelim:lem-PR-weight-e3} from Lemma \ref{prelim:lem-PR-weight}, \eqref{difference:eq-good-w-LWP-sup}, and the induction hypothesis, it then holds that
\begin{align*}
M_{R_j}(w(t_{j-1})) &\lesssim M_{R_{j-1}}(w(t_{j-1})) + R_{j}^{-10} \big\| \langle x\rangle^{-10} w(t_{j-1}) \big\|_{L^1(\R)} \\
&\lesssim A_3^j R_{j-1}^{-\frac{1}{2}} + R_{j}^{-10} \log(R_j)^{\frac{2}{p+3}} \lesssim A_3^j R_j^{-1}.
\end{align*}
Using Proposition \ref{difference:prop-main} and $p\leq 5$, we then obtain that
\begin{align}
\sup_{t\in [t_{j-1},t_j]} M_{R_j}(w(t)) 
&\lesssim \sup_{t\in [t_{j-1},t_j]} \exp\big( A_2 \log(2R_j) |t-t_{j-1}|\big) \Big( M_{R_j}(w(t_{j-1})) +  R_{j}^{-1+8\delta} \Big)\notag \\
&\lesssim R_j^{2A_2 \tau} \big( A_3^j R_j^{-1} + R_j^{-1+8\delta} \big) 
\lesssim A_3^j R_j^{-1+8\delta+2A_2\tau}. \label{difference:eq-iterated-2}
\end{align}
Since $\delta$ is  small (see Definition \ref{difference:def-good-event}) and $\tau\leq \tau_0(A_0,A_1,A_2)$ is sufficiently small depending on~$A_2$, we have $-1+8\delta+2A_2\tau \leq -\frac{1}{2}$. Since $A_3$ has been chosen as sufficiently large depending on the implicit constant in \eqref{difference:eq-iterated-2}, we therefore obtain \eqref{difference:eq-iterated-1}. \\

Finally, using \eqref{prelim:lem-PR-weight-e3} from Lemma \ref{prelim:lem-PR-weight}, \eqref{difference:eq-good-w-LWP-sup}, and \eqref{difference:eq-iterated-1}, we obtain that
\begin{align*}
\sup_{t\in [0,T]} M_R(w(t)) &= \max_{0\leq j \leq J} \sup_{t\in [t_{j-1},t_j]} M_R(w(t)) \\
&\lesssim  \max_{0\leq j \leq J} \sup_{t\in [t_{j-1},t_j]} \Big( M_{R_{j}}(w(t)) + R^{-10} \big\| \langle x  \rangle^{-10} w(t) \big\|_{L_x^1(\R)} \Big)  \\
&\lesssim \max_{0\leq j \leq J} \big( A_3^{j+1} R_j^{-\frac{1}{2}} + R^{-10} \log(R)^{\frac{2}{p+3}} \big) 
\lesssim A_3^{J+1} R^{-\frac{1}{2}} \leq A_3^{J+2} R^{-\frac{1}{2}},
\end{align*}
which yields \eqref{difference:eq-iterated-final}.
\end{proof}

\section{Proof of the main theorem}

Equipped with Lemma \ref{difference:lem-probability-good} and Lemma \ref{difference:lem-iterated}, as well as the estimates from Subsection \ref{section:prelim-harmonic}, we can now prove the main theorem of this article.

\begin{proof}[Proof of Theorem \ref{intro:thm-dynamics}:] 
We first prove the quantitative estimate \eqref{intro:eq-as-convergence-estimate}, which directly implies the \mbox{$\bP$-a.s.} convergence of $u_L$ in $C_t^0 C_x^\alpha([-T,T]\times I)$ for all $0\leq \alpha<\frac{1}{2}$, all $T\geq 1$, and all compact intervals $I\subseteq \R$. Since $C^\prime\geq 1$ and $\eta^\prime>0$ in \eqref{intro:eq-as-convergence-estimate} are allowed to depend on $\alpha,C,\eta$, and $T$ from Theorem~\ref{intro:thm-dynamics} and $A_0,A_1,A_2$, and $A_3$ from the previous lemmas, we can also allow all implicit constants below to depend on $\alpha,C,\eta,T,A_0,A_1,A_2$, and $A_3$. In particular, we can replace all $\log(T+R)$-terms from our previous estimates with $\log(R)$-terms.  By increasing the value of $C^\prime$, if necessary, we may also assume that $L\geq L_0$, where $L_0=L_0(\alpha,C,\eta,T,A_0,A_1,A_2,A_3)$ is sufficiently large. Finally, we let $D=D(\alpha,C,\eta,T,A_0,A_1,A_2,A_3)$ be a sufficiently large parameter. \\

We now define parameters $R_0$, $R$, and $N$, all depending on $L$, such that 
\begin{equation}\label{proof:eq-N-R}
R_0 := L^{\frac{\eta}{4}}, \qquad R:= R_0^{2^{-J}}, \qquad \text{and} \qquad  N:= R^{\frac{1}{8}},
\end{equation}
where $J$ is chosen depending on $A_0,A_1,A_2$, and $T$ as in Lemma \ref{difference:lem-iterated}. We note that the relationship between $R$ and $R_0$ is exactly as in Lemma \ref{difference:lem-iterated}, which will be needed later. Due to assumption \eqref{intro:eq-coupling} and Lemma \ref{difference:lem-probability-good}, it holds that 
\begin{equation}\label{proof:eq-probability-estimate}
\bP \Big(  \big\| \phi_L - \phi_{L/2} \big\|_{C^0_x([-(L/2)^\eta,(L/2)^\eta])}
> 2 (L/2)^{-\eta} \Big) + \bP \Big( \Omega \backslash \goodind \Big)
\lesssim L^{-\eta} + R^{-100} \leq L^{-\eta^\prime}. 
\end{equation}
In the last inequality, we also used \eqref{proof:eq-N-R} and that $\eta^\prime$ is sufficiently small depending on $\alpha$, $C$, $\eta$, and~$T$. 
As a result, it suffices to show that 
\begin{equation}\label{proof:eq-inclusion}
\begin{aligned}
&\, \Big\{ \big\| \phi_L - \phi_{L/2} \big\|_{C^0_x([-(L/2)^\eta,(L/2)^\eta])}
\leq 2 (L/2)^{-\eta} \Big\} \medcap \goodind \\ 
\subseteq&\, \Big\{ \big\| u_L - u_{L/2} \big\|_{C_t^0 C_x^\alpha([-T,T]\times [-L^{\eta^\prime},L^{\eta^\prime}])} \leq L^{-\eta^\prime} \Big\}.
\end{aligned}
\end{equation}
Due to the time-reflection symmetry of \eqref{intro:eq-NLS}, we may replace $[-T,T]$ on the right-hand side of~\eqref{proof:eq-inclusion} by $[0,T]$. In all of the following, we implicitly restrict ourselves to the event on the left-hand side of~\eqref{proof:eq-inclusion}. We recall that, due to Definition \ref{difference:def-good-event}, 
\begin{equation}\label{proof:eq-crude}
\big\| \log(R+\langle x \rangle)^{-\frac{2}{p+3}} u_{L/2} \big\|_{L_x^\infty(\R)}
+ \big\| \log(R+\langle x \rangle)^{-\frac{2}{p+3}} u_{L} \big\|_{L_x^\infty(\R)}
\lesssim 1. 
\end{equation}
Using \eqref{proof:eq-crude}, we will be able to easily control the minor error terms from Lemmas \ref{prelim:lem-bernstein}, \ref{prelim:lem-Hoelder-LWP}, and \ref{prelim:lem-PR-weight}, which will be used repeatedly below. 
In order to later use Lemma \ref{difference:lem-iterated}, we now verify that 
\begin{equation}\label{proof:eq-initial}
M_{R_0}\big( u_L(0)-u_{L/2}(0) \big) =  M_{R_0}\big( \phi_L-\phi_{L/2} \big) \leq R_0^{-\frac{1}{2}},
\end{equation}
where the frequency-truncated, localized mass is as in \eqref{intro:eq-MR}. Using Lemma \ref{prelim:lem-PR-weight}, \eqref{proof:eq-crude}, and $R\leq R_0$,  we obtain
\begin{equation}\label{proof:eq-initial-1}
\begin{aligned}
M_{R_0}\big(  \phi_L-\phi_{L/2} \big)^{\frac{1}{2}} 
= \big\| \sqrt{\sigma_{R_0}} P_{\leq R_0} \big(  \phi_L-\phi_{L/2} \big) \big\|_{L_x^2(\R)} 
\lesssim \big\| \sqrt{\sigma_{R_0}} \big(  \phi_L-\phi_{L/2} \big) \big\|_{L_x^2(\R)} +R_0^{-D+1}. 
\end{aligned}
\end{equation}
Using \eqref{proof:eq-crude}, the first term in \eqref{proof:eq-initial-1} can be bounded by 
\begin{equation}\label{proof:eq-initial-2}
\begin{aligned}
&\, \big\| \sqrt{\sigma_{R_0}} \big(  \phi_L-\phi_{L/2} \big) \big\|_{L_x^2(\R)} \\
\lesssim&\, R_0 \big\| \phi_L - \phi_{L/2} \big\|_{C_x^0([-R_0^2,R_0^2])}  
+ \big\| \sqrt{\sigma_{R_0}} \phi_L \big\|_{L_x^2(\R \backslash [-R_0^2,R_0^2])} 
+ \big\| \sqrt{\sigma_{R_0}} \phi_{L/2}\big\|_{L_x^2(\R \backslash [-R_0^2,R_0^2])}  \\
\lesssim&\,  R_0 \big\| \phi_L - \phi_{L/2} \big\|_{C_x^0([-R_0^2,R_0^2])}   + e^{-\frac{1}{4} R_0}.
\end{aligned}
\end{equation}
Using \eqref{proof:eq-N-R}, the first term in \eqref{proof:eq-initial-2} can be bounded by 
\begin{equation}\label{proof:eq-initial-3}
R_0 \big\| \phi_L - \phi_{L/2} \big\|_{C_x^0([-R_0^2,R_0^2])}
\leq R_0 \big\| \phi_L - \phi_{L/2} \big\|_{C_x^0([-(L/2)^\eta,(L/2)^\eta])} \lesssim R_0 L^{-\eta} \lesssim R_0^{-1}.
\end{equation}
By combining \eqref{proof:eq-initial-1}, \eqref{proof:eq-initial-2}, and \eqref{proof:eq-initial-3}, we obtain that
\begin{equation}\label{proof:eq-initial-4}
M_{R_0}\big(  \phi_L-\phi_{L/2} \big)
\lesssim R_0^{-2} + e^{-\frac{1}{2} R_0} + R_0^{-2D+2},
\end{equation}
which clearly implies \eqref{proof:eq-initial}. We now turn to space-time estimates of the difference between $u_{L/2}$ and $u_{L/2}$. We first use a decomposition into high and low frequencies, which yields that 
\begin{align}
\big\| u_L - u_{L/2} \big\|_{C_t^0 C_x^\alpha([0,T]\times [-N,N])}
&\leq\big\| P_{> N} \big( u_L - u_{L/2} \big) \big\|_{C_t^0 C_x^\alpha([0,T]\times [-N,N])} \label{proof:eq-high} \\
&+ \big\| P_{\leq N} \big( u_L - u_{L/2} \big) \big\|_{C_t^0 C_x^\alpha([0,T]\times [-N,N])}. \label{proof:eq-low}
\end{align}
We first estimate the high-frequency term \eqref{proof:eq-high}. By using the triangle inequality, a dyadic decomposition of $P_{>N}$, and the identity $P_K =P_K P_{>K/4}$ for all $K> N$, we obtain that 
\begin{equation}\label{proof:eq-high-estimated-1}
\eqref{proof:eq-high} \leq \sum_{\substack{K\in \dyadic\colon\\[1pt] K>N}} 
\Big( \big\| P_K P_{> K/4} u_L \big\|_{C_t^0 C_x^\alpha([0,T]\times [-N,N])} 
+ \big\| P_K P_{> K/4}  u_{L/2} \big\|_{C_t^0 C_x^\alpha([0,T]\times [-N,N])} \Big).
\end{equation}
Since an identical argument can be used for the $u_{L/2}$-term, we only estimate the $u_L$-term in \eqref{proof:eq-high-estimated-1}. By using Lemma \ref{prelim:lem-Hoelder-LWP}, Definition \ref{difference:def-good-event}, and \eqref{proof:eq-crude}, we obtain that 
\begin{equation}\label{proof:eq-high-estimated-2}
\begin{aligned}
&\,\big\| P_K P_{> K/4} u_L \big\|_{C_t^0 C_x^\alpha([0,T]\times [-N,N])} \\
\lesssim&\,  K^\alpha \big\| P_{> K/4} u_L \big\|_{C_t^0 C_x^0([0,T]\times [-N,N])}  + (KN)^{-D} \big\| \langle x \rangle^{-D} P_{>K/4} u_L \big\|_{L^1(\R)} \\
\lesssim&\, K^{\alpha} K^{-\frac{1}{2}+\delta} \log(R)^{\frac{1}{2}} + (KN)^{-D} \log(R)^{\frac{2}{p+3}}.
\end{aligned}
\end{equation}
In the last estimate, we also used the boundedness of $P_{>K/4}$ on our weighted $L^1(\R)$-space. By using that $\delta$ is sufficiently small depending on $\alpha$ (see Definition \ref{difference:def-good-event}) and by combining \eqref{proof:eq-high-estimated-1} and~\eqref{proof:eq-high-estimated-2}, we then obtain that
\begin{equation}\label{proof:eq-high-estimated-3}
\eqref{proof:eq-high} \lesssim N^{\alpha-\frac{1}{2}+\delta}\log(R)^{\frac{1}{2}} + N^{-2D} \log(R)^{\frac{2}{p+3}}.
\end{equation}
We now turn to the low-frequency term \eqref{proof:eq-low}, which is more difficult to estimate. Using Lemma~\ref{prelim:lem-bernstein}, Lemma~\ref{prelim:lem-Hoelder-LWP}, and \eqref{proof:eq-crude}, we estimate
\begin{equation}\label{proof:eq-low-1}
\begin{aligned}
\eqref{proof:eq-low}
&\lesssim N^\alpha \big\| P_{\leq N} \big( u_L - u_{L/2} \big) \big\|_{C_t^0 C_x^0([0,T]\times [-2N,2N])}  + N^{-2D} \log(R)^{\frac{2}{p+3}}  \\
&\lesssim N^{\frac{1}{2}+\alpha} \big\| P_{\leq N} \big( u_L - u_{L/2} \big) \big\|_{C_t^0 L^2_{x,\loc}([0,T]\times [-4N,4N])}  + N^{-2D+\alpha} \log(R)^{\frac{2}{p+3}},
\end{aligned}
\end{equation}
where $L^2_{x,\loc}$ is as in \eqref{prelim:eq-Lploc}. By using the trivial estimate $L^2_x([-4N,4N])\hookrightarrow L^2_{x,\loc}([-4N,4N])$, using that $\sigma_N(x)\gtrsim 1$ for all $x\in[-4N,4N]$, and using Lemma \ref{prelim:lem-bernstein}, the first term in \eqref{proof:eq-low-1} can be estimated by 
\begin{equation}\label{proof:eq-low-2}
\begin{aligned}
&\, N^{\frac{1}{2}+\alpha} \big\| P_{\leq N} \big( u_L - u_{L/2} \big) \big\|_{C_t^0 L^2_{x,\loc}([0,T]\times [-4N,4N])} \\
\lesssim&\,  N^{\frac{1}{2}+\alpha}
\big\| \sqrt{\sigma_N}  P_{\leq N} \big( u_L - u_{L/2} \big) \big\|_{C_t^0 L_x^2([0,T]\times \R)} \\
\lesssim&\,  N^{\frac{1}{2}+\alpha}
\big\| \sqrt{\sigma_R}  P_{\leq R} \big( u_L - u_{L/2} \big) \big\|_{C_t^0 L_x^2([0,T]\times \R)}+ N^{-D+\frac{1}{2}+\alpha} \log(R)^{\frac{2}{p+3}}. 
\end{aligned}
\end{equation}
Using Lemma \ref{difference:lem-iterated} and \eqref{proof:eq-initial}, the first term in \eqref{proof:eq-low-2} can be estimated by 
\begin{equation}\label{proof:eq-low-3}
N^{\frac{1}{2}+\alpha}
\big\| \sqrt{\sigma_R}  P_{\leq R} \big( u_L - u_{L/2} \big) \big\|_{C_t^0 L_x^2([0,T]\times \R)}  
\lesssim N^{\frac{1}{2}+\alpha} R^{-\frac{1}{2}}.
\end{equation}
By collecting our estimates from \eqref{proof:eq-high}-\eqref{proof:eq-low-3}, we obtain that 
\begin{equation}
\begin{aligned}\label{proof:eq-difference-final}
&\, \big\| u_L - u_{L/2} \big\|_{C_t^0 C_x^\alpha([0,T]\times [-N,N])} \\
\lesssim&\,  N^{\alpha+\delta-\frac{1}{2}} \log(R)^{\frac{1}{2}} + N^{\frac{1}{2}+\alpha} R^{-\frac{1}{2}} + N^{-D+\frac{1}{2}+\alpha} \log(R)^{\frac{2}{p+3}} \leq L^{-\eta^\prime}.
\end{aligned}
\end{equation}
In the last estimate, we used that $\delta>0$ is sufficiently small depending on $\alpha$ (as in Definition~\ref{difference:def-good-event}) and that $\eta^\prime>0$ is sufficiently small depending on $\alpha,C,\eta$, and $T$. This completes our proof of~\eqref{proof:eq-inclusion} and, as a result, our proof of \eqref{intro:eq-as-convergence-estimate}.\\ 

It remains to show that the almost-sure limit $u$ solves \eqref{intro:eq-NLS} in the sense of space-time distributions and preserves the Gibbs measure. For the first claim, we note that the almost-sure convergence of $u_L$ to $u$ in the space $C_t^0 C_{x}^\alpha([-T,T]\times I)$ implies the almost-sure convergence of the power-type nonlinearities $|u_L|^{p-1}u_L$ to $|u|^{p-1}u$ in the same space. Since $u_L$ solves \eqref{intro:eq-NLS} in the sense of space-time distributions, this readily implies that $u$ also solves \eqref{intro:eq-NLS} in the sense of space-time distributions. In order to obtain the second claim, we  let $K\subseteq \R$ be a compact interval and let $F\colon C_x^\alpha(K)\rightarrow \C$ be bounded and continuous. Using the invariance of $\mu_L$ under \eqref{intro:eq-NLS}, the weak convergence of $\mu_L$ to $\mu$, and the almost-sure convergence of $u_L$ to $u$, we obtain for all $t\in \R$ that 
\begin{equation*}
\E \big[ F(u(t)) \big]
= \lim_{L\rightarrow \infty} \E \big[ F(u_L(t)) \big]
= \lim_{L\rightarrow \infty} \E \big[ F(u_L(0)) \big]
= \lim_{L\rightarrow \infty} \int F(\phi) \mathrm{d}\mu_L(\phi) 
= \int F(\phi) \mathrm{d}\mu(\phi). 
\end{equation*}
From this it follows that $\Law_\bP(u(t))=\mu$ for all $t\in \R$. 
\end{proof}

\begin{remark}\label{proof:rem-quantitative-assumption}
Let us briefly assume that assumption \eqref{intro:eq-coupling} is not necessarily satisfied, and we only know that $\phi$ is the $\bP$-as limit of $\phi_L$ in $C_x^0(I)$ for all compact intervals $I\subseteq \R$. In that case, we claim that it is possible to find an increasing sequence $(L_k)_{k=0}^{\infty}\subseteq \dyadic$, which may depend on $\alpha$ and $T$, such that
\begin{equation}\label{proof:eq-qualitative}
\bP\Big( \big\| u_{L_k} - u_{L_{k-1}}\big\|_{C_t^0 C_x^\alpha([-T,T]\times [-2^k,2^k])} >2^{-k} \Big) \leq 2^{-k}
\end{equation}
for all $k\geq 0$. By passing to a further subsequence, which is chosen using a diagonal argument and accounts for different choices of $\alpha$ and $T$, one then sees that the limit of $(u_{L_k})_{k=0}^\infty$ exists $\bP$-a.s. in $C_t^0 C_x^\alpha([-T,T]\times I)$ for all $\alpha\in [0,1/2)$, all $T\geq 1$, and all compact intervals $I\subseteq \R$.

The proof of \eqref{proof:eq-qualitative} under this weaker assumption is close to the proof of Theorem \ref{intro:thm-dynamics}, and the main difference lies in the choice of the parameters. One first chooses $N$ as a sufficiently large power of $2^k$ and, similarly as in \eqref{proof:eq-N-R}, then chooses $R:=N^8$ and $R_0:=R^{2^J}$. The sequence element $L_{k-1}$ is then chosen such that
\begin{equation*}
\bP \Big( R_0 \sup_{L\geq L_{k-1}} \big\| \phi_L - \phi_{L_{k-1}} \big\|_{C_x^0([-R_0^2,R_0^2])} > R_0^{-1} \Big) \leq 2^{-k-1}.
\end{equation*}
With this choice, the final estimate in \eqref{proof:eq-initial-3} then holds with high probability. 
\end{remark}

\begin{appendix}
\section{Quantitative Skorokhod representation theorem}\label{section:Skorokhod}

In this appendix, we prove a quantitative version of the Skorokhod representation theorem, which is needed in the proof of Proposition \ref{measure:prop-coupling}.  To this end, we recall that the function space $\CE^\theta$ and the Wasserstein-distance $\Wasserstein$ were defined in \eqref{measure:eq-def-CE} and \eqref{measure:eq-def-Wasserstein}, respectively. 

\begin{proposition}[A quantitative version of the Skorokhod representation theorem]\label{sk:prop-main}
Let $C\geq 1$ and $0<c\leq 1$ be constants. Furthermore, let $\alpha \in (0,1]$, $\beta>0$, $\kappa \in (0,1]$, and $\theta>0$ be parameters. Let $\mu_L$, where $L\in \dyadic \medcup \{ \infty\}$,  be probability measures on $\CE^\theta(\R)$ which satisfy the following conditions:
\begin{enumerate}[label=(\roman*)]
\item\label{sk:item-h} (Hölder-regularity) For all $R\geq 10$ and $\lambda>0$, it holds that 
\begin{equation}\label{sk:eq-hoelder}
\sup_{L} \mu_L \Big( \Big\{ \| \varphi \|_{C^\alpha([-R,R])} \geq C (\log(R)+\lambda)^{\frac{1}{2}} \Big\} \Big) \leq C e^{-c\lambda}.
\end{equation}
\item\label{sk:item-d} (Density estimates) For all $x\in \R$ and all $a,b\in \R$ satisfying $a<b$, it holds that 
\begin{equation}\label{sk:eq-density}
\sup_{L} \mu_L \Big( \Big\{ \Re \varphi(x) \in [a,b] \Big\} \Big), 
\sup_{L} \mu_L \Big( \Big\{ \Im \varphi(x) \in [a,b] \Big\} \Big)
\leq C |b-a|^\kappa.
\end{equation}
\item\label{sk:item-w} (Wasserstein-estimate) For all $L\in \dyadic$, it holds that 
\begin{equation}\label{sk:eq-Wasserstein}
\Wasserstein(\mu_\infty,\mu_L) \leq C e^{-c L^\beta}.
\end{equation}
\end{enumerate}
Then, there exist constants $C^\prime\geq 1$, $c^\prime>0$, and $\eta>0$, depending only on $C,c,\alpha,\beta,\kappa$, and $\theta$, a common probability space $(\Omega,\mathcal{F},\bP)$,  and random functions $\phi_L \colon (\Omega,\mathcal{F})\rightarrow \CE^\theta(\R)$, where $L\in~\dyadic~\medcup~\{\infty\}$, such that the following properties are satisfied:
\begin{enumerate}[label=(\alph*)]
\item (Coupling) For all $L\in \dyadic \medcup \{\infty\}$, it holds that $\Law_{\bP}(\phi_L)=\mu_L$.
\item (Quantitative almost-sure convergence) For all $L\in \dyadic$, it holds that
\begin{equation}\label{sk:eq-final-estimate}
\bP \Big( \Big\{ \big\| \phi_\infty - \phi_L \big\|_{C^0([-L^\eta,L^\eta])} > L^{-\eta} \Big\} \Big) \leq C^\prime e^{-c^\prime L^\eta}.
\end{equation}
\end{enumerate}
\end{proposition}

\begin{remark}
We note that the Wasserstein-estimate \eqref{sk:eq-Wasserstein} implies the weak convergence of $\mu_L$ to $\mu_\infty$ with respect to the $\CE^\theta$-norm. From the Skorokhod representation theorem, it therefore follows\footnote{The Skorokhod representation theorem also requires that $\mu_\infty$ has separable support, but this follows directly from the Hölder-estimate in \eqref{sk:eq-hoelder}.} that there exists a common probability space $(\Omega,\mathcal{F},\bP)$ and random variables $\phi_L\colon (\Omega,\mathcal{F})\rightarrow \CE^\theta(\R)$, where $L\in \dyadic \medcup \{\infty\}$, such that $\phi_L$ converges to $\phi_\infty$ $\bP$-a.s. For our purposes, however, this is insufficient, since we require a more quantitative estimate of the difference between $\phi_L$ and $\phi_\infty$. This more quantitative estimate is provided by \eqref{sk:eq-final-estimate}.
\end{remark}

\begin{proof}
We follow the structure of the proof of the standard Skorokhod representation theorem given in \cite[Theorem 6.7]{B99}, but make each of the steps more quantitative in $L$. Throughout the proof, we can assume that $L$ is sufficiently large depending on the parameters appearing in \ref{sk:item-h}-\ref{sk:item-w}, i.e., 
\begin{equation*}
L\geq L_0(C,c,\alpha,\beta,\kappa,\theta). 
\end{equation*}
The reason is that, for $L<L_0$, \eqref{sk:eq-final-estimate} is trivially satisfied, and we can therefore choose $\phi_L$ as any random variable satisfying $\Law_\bP(\phi_L)=\mu_L$. The new parameter $\eta$ is chosen as small depending on the parameters appearing  \ref{sk:item-h}-\ref{sk:item-w}. Furthermore, the new parameters $C^\prime\geq 1$ and $c^\prime>0$ are chosen as sufficiently large and small depending on both the parameters appearing  \ref{sk:item-h}-\ref{sk:item-w} and $\eta$, respectively.
To simplify the notation, we write $\mu:=\mu_\infty$.  \\

\noindent \emph{Step 1: Construction of a suitable partition of $\CE^\theta(\R)$.} 
We now introduce several parameters depending on $L\geq L_0$. First, we define $R_L \geq 1$, $K_L \in \mathbb{N}$, and $\varepsilon_L,\widetilde{\varepsilon}_L>0$ as 
\begin{equation}\label{sk:eq-parameters-collection}
R_L:=L^\eta, \qquad K_L:=\big\lceil R_L L^{\frac{2\eta}{\alpha}} \big\rceil, \qquad \varepsilon_L := e^{-\frac{1}{16} c\kappa L^\beta}, \qquad  \text{and} \qquad \widetilde{\varepsilon}_L := e^{-\frac{1}{16} c\kappa L^\beta},
\end{equation}
where $\lceil \, \cdot\, \rceil$ is the ceiling function. While the two parameters  $\varepsilon_L$ and $\widetilde{\varepsilon}_L$ can be chosen to have the same value, they play different roles in our argument below, and we therefore use different notation for them. Second, we define the step-size $\delta_L>0$ and the grid points $x^L_k$, where $-K_L\leq k \leq K_L$, as 
\begin{equation}\label{sk:eq-parameters-delta}
\delta_L := \frac{R_L}{K_L} \leq L^{-\frac{2\eta}{\alpha}} \qquad \text{and} \qquad x^L_k:= k \delta_L.
\end{equation}
We note that $x_{-K_L}=-R_L$ and $x_{K_L}=R_L$, and hence the grid points are all contained in the interval $[-R_L,R_L]$. Third, we choose a parameter $M_L>0$ such that\footnote{In order to prove \eqref{sk:eq-prob-BL} below, it is important that the probability on the left-hand side of \eqref{sk:eq-parameter-ML} is not too small.}
\begin{equation}\label{sk:eq-parameter-ML}
\mu \Big( \bigcup_{k=-K_L}^{K_L} \big\{ \varphi(x^L_k) \not \in [-M_L,M_L] \big\} \Big) = \widetilde{\varepsilon}_L.
\end{equation}
This is possible since, due to \eqref{sk:eq-hoelder} and \eqref{sk:eq-density}, the function 
\begin{equation*}
M\in [0,\infty) \mapsto \mu \Big( \bigcup_{k=-K_L}^{K_L} \big\{ \varphi(x^L_k) \not \in [-M,M] \big\} \Big) 
\end{equation*}
is continuous\footnote{For more details on this, see the estimates in \eqref{sk:eq-prob-difference-6} and \eqref{sk:eq-prob-difference-8} below.}, takes the value one at $M=0$, and tends to zero as $M$ tends to infinity. From \eqref{sk:eq-hoelder} and \eqref{sk:eq-parameter-ML}, it follows that 
\begin{equation*}
\widetilde{\varepsilon}_L 
\leq \mu \Big( \big\{ \| \varphi \|_{C^0([-R_L,R_L])} \geq M_L \big\} \Big) 
\leq C \exp\Big( - c \big( C^{-2} M_L^2 - \log(R_L) \big) \Big),
\end{equation*}
which implies the upper bound
\begin{equation}\label{sk:eq-parameter-ML-upper}
M_L \lesssim_{C,c,\kappa} L^{\frac{\beta}{2}}\leq L^\beta.
\end{equation}
Fourth, we define $J_L \in \mathbb{N}$ and $\tau_L>0$ as
\begin{equation}\label{sk:eq-parameters-J-tau}
J_L := \big\lceil 8 L^\eta M_L \big \rceil \qquad \text{and} \qquad  \tau_L := \frac{M_L}{J_L}\leq \tfrac{1}{8} L^{-\eta}.
\end{equation}
Equipped with the grid points $x^L_k$ and the parameter $\tau_L$, we now define  $\Zc^L$ as the set of all functions 
\begin{equation}\label{sk:eq-zL}
z^L \colon \big\{ -K_L,-K_L+1,\hdots, K_L-1,K_L \big\} \rightarrow \big\{ j_1 \tau_L + i j_2 \tau_L\colon -J_L \leq j_1,j_2 \leq J_L\big\}. 
\end{equation}
The set $\Zc^L$ will be used to discretize functions, see \eqref{sk:eq-AL} below. From the definition of $\Zc^L$, together with \eqref{sk:eq-parameters-collection}, \eqref{sk:eq-parameter-ML-upper}, and \eqref{sk:eq-parameters-J-tau}, it follows that
\begin{equation}\label{sk:eq-zL-cardinality}
\# \Zc^L  = (2J_L+1)^{2(2K_L+1)} \leq e^{cL^{\frac{4\eta}{\alpha}}}.
\end{equation}
Finally, for each $z^L \in \Zc^L$, we define
\begin{equation}\label{sk:eq-AL}
\AL := \bigcap_{k=-K_L}^{K_L} \Big\{ \varphi \in \CE^\theta(\R) \colon  \Re\big(\varphi(x^L_k)-z^L_k\big),\Im\big(\varphi(x^L_k)-z^L_k\big) \in [0,\tau_L) \Big\}.
\end{equation}
From \eqref{sk:eq-zL} and \eqref{sk:eq-AL}, it directly follows that the sets $(\AL)_{z^L\in \Zc^L}$ are disjoint. In order to obtain \eqref{sk:eq-prob-AL} below, we need to restrict ourselves to $z^L \in \Zc^L$ for which the probabilities of $\AL$ are not too small. For this reason, we define 
\begin{equation}\label{sk:eq-ZLgood}
\Zc^L_g := \Big\{ z^L \in \Zc^L \colon \mu \big( \AL \big) \geq \widetilde{\varepsilon}_L  \Big\}.
\end{equation}
Equipped with $\Zc^L_g$, we can now define the good and bad events $G^L$ and $B^L$ as 
\begin{equation}\label{sk:eq-good-and-bad}
G^L = \bigcup_{z^L \in \Zc^L_g} \AL \qquad \text{and} \qquad B^L:= \CE^\theta(\R) \backslash G^L.
\end{equation}

\noindent \emph{Step 2: The probabilities of $\AL$, where $z^L \in \Zc^L_g$, and $B^L$.} In this step, we prove that 
\begin{align}
\mu_L(\AL) &\geq (1-\varepsilon_L) \mu(\AL) \qquad \text{for all } z^L\in \Zc^L_g,\label{sk:eq-prob-AL} \\
\mu_L(B^L) &\geq (1-\varepsilon_L) \mu(B^L), \label{sk:eq-prob-BL} \\
\mu(B^L) &\leq e^{-\frac{1}{32} c\kappa L^\beta}.\label{sk:eq-probability-BL-upper}
\end{align}
In order to prove \eqref{sk:eq-prob-AL}, \eqref{sk:eq-prob-BL}, and \eqref{sk:eq-probability-BL-upper},  we first show that 
\begin{equation}\label{sk:eq-prob-AL-difference}
\big| \mu(\AL) - \mu_L(\AL) \big| \leq e^{-\frac{1}{4}c\kappa L^\beta} \qquad \text{for all } z^L \in \Zc^L.
\end{equation}
We remark that, in contrast to \eqref{sk:eq-prob-AL}, this estimate even holds for $z^L\in \Zc^L \backslash\Zc^L_g$. In order to obtain~\eqref{sk:eq-prob-AL-difference}, we first use \eqref{sk:eq-Wasserstein}. Due to this, there exists a coupling $\gamma_L \in \Gamma(\mu,\mu_L)$ satisfying
\begin{equation}\label{sk:eq-wasserstein-used}
\int \| \varphi - \varphi_L \|_{\CE^\theta(\R)} \dgamma_L(\varphi,\varphi_L) \leq 2Ce^{-cL^\beta}.
\end{equation}
Using the coupling $\gamma_L$, we can then rewrite the probability as 
\begin{align}
 \big| \mu(\AL) - \mu_L(\AL) \big| 
 =&\, \Big| \int \big( \ind\{ \varphi \in \AL \big\} - \ind\{ \varphi_L \in \AL \big\} \big) \dgamma_L(\varphi,\varphi_L) \Big| \notag \\
 \leq&\, \int \big|  \ind\{ \varphi \in \AL \big\} - \ind\{ \varphi_L \in \AL \big\} \big| \dgamma_L(\varphi,\varphi_L). \label{sk:eq-prob-difference-2} 
 \end{align}
From the definition of $\AL$, it follows that if $\varphi \in \AL$ and $\varphi_L\not \in \AL$, then there must exist an index $k\in\{-K_L,\hdots,K_L\}$ and an $F\in \{\Re,\Im\}$ such that
\begin{equation*}
    F(\varphi(x^L_k)-z^L_k)\in [0,\tau_L) \qquad \text{and} \qquad F(\varphi_L(x^L_k)-z^L_k)\not \in [0,\tau_L).
\end{equation*}
By also using a similar argument in the case $\varphi \not \in \AL$ and $\varphi_L \in \AL$, we then obtain that  
 \begin{align}
 \hspace{-2ex}\eqref{sk:eq-prob-difference-2} \leq&\, \sum_{F=\Re,\Im} \sum_{k=-K_L}^{K_L}
 \hspace{-0.5ex}\int \ind \big\{ F(\varphi(x^L_k)-z_k^L) \in [0,\tau_L), F(\varphi_L(x^L_k)-z_k^L) \not \in [0,\tau_L) \big\} \dgamma_L(\varphi,\varphi_L) \label{sk:eq-prob-difference-3} \\
 +&\, \sum_{F=\Re,\Im} \sum_{k=-K_L}^{K_L}\hspace{-0.5ex}
 \int \ind \big\{ F(\varphi(x^L_k)-z_k^L) \not \in [0,\tau_L), F(\varphi_L(x^L_k)-z_k^L) \in [0,\tau_L) \big\} \dgamma_L(\varphi,\varphi_L). \label{sk:eq-prob-difference-4}
\end{align}
Since the estimates of \eqref{sk:eq-prob-difference-3} and \eqref{sk:eq-prob-difference-4} are similar, we only treat \eqref{sk:eq-prob-difference-3}. 
To control \eqref{sk:eq-prob-difference-3}, we introduce the additional parameter 
\begin{equation}\label{sk:eq-prob-rho}
\rho_L := e^{-\frac{1}{2} c L^{\beta}}.
\end{equation}
By splitting the interval $[0,\tau_L)$ into $[0,\rho_L)$, $[\rho_L,\tau_L-\rho_L)$, and $[\tau_L-\rho_L,\tau)$, we then obtain that 
\begin{align}
&\,\hspace{1ex} \eqref{sk:eq-prob-difference-3}  \notag \\
\leq&\,  \sum_{F=\Re,\Im} \sum_{k=-K_L}^{K_L}
 \hspace{-0.5ex}\int \ind \big\{ F(\varphi(x^L_k)-z_k^L) \in [\rho_L,\tau_L-\rho_L), F(\varphi_L(x^L_k)-z_k^L) \not \in [0,\tau_L) \big\} \dgamma_L(\varphi,\varphi_L)\label{sk:eq-prob-difference-5} \\ 
 +&\, \sum_{F=\Re,\Im} \sum_{k=-K_L}^{K_L}
 \hspace{-0.5ex}\int \ind \big\{ F(\varphi(x^L_k)-z_k^L) \in [0,\rho_L) \cup [\tau_L-\rho_L,\tau_L) \big\} \dgamma_L(\varphi,\varphi_L)\label{sk:eq-prob-difference-6}. 
\end{align}
Using \eqref{sk:eq-wasserstein-used}, the first term \eqref{sk:eq-prob-difference-5} can be estimated by 
\begin{equation}\label{sk:eq-prob-difference-7}
\begin{aligned}
\eqref{sk:eq-prob-difference-5}
&\leq \sum_{F=\Re,\Im} \sum_{k=-K_L}^{K_L} 
\int \ind \big\{ |\varphi(x^L_k) - \varphi_L(x^L_k)|\geq \rho_L \big\} \dgamma_L(\varphi,\varphi_L) \\ 
&\leq \rho_L^{-1}  \sum_{F=\Re,\Im} \sum_{k=-K_L}^{K_L} 
 e^{\theta | x^L_k |} 
 \int \big\| \varphi - \varphi_L \big\|_{\CE^\theta(\R)} \dgamma_L(\varphi,\varphi_L) \\
 &\leq 4C (2 K_L +1) \rho_{L}^{-1} e^{\theta R_L}  e^{-cL^\beta}.
\end{aligned}
\end{equation}
Using \eqref{sk:eq-hoelder}, the second term \eqref{sk:eq-prob-difference-6} can be estimated by\footnote{This estimate is a more quantitative version of the condition $P(\partial B^m_i)=0$ in \cite[(6.8)]{B99}. Roughly speaking, we not only show that $\partial \AL$ has zero measure, but show that a $\rho_L$-neighborhood of $\partial \AL$ has measure $\lesssim K_L \rho_L^\kappa$.}
\begin{equation}\label{sk:eq-prob-difference-8}
\begin{aligned}
\eqref{sk:eq-prob-difference-6}
&= \sum_{F=\Re,\Im} \sum_{k=-K_L}^{K_L}
 \hspace{-0.5ex}\mu \Big( \big\{ F(\varphi(x^L_k)-z_k^L) \in [0,\rho_L) \cup [\tau_L-\rho_L,\tau) \big\} \Big) \\ 
&\leq 4C (2K_L+1) \rho_L^\kappa. 
\end{aligned}
\end{equation}
In total, we therefore obtain that
\begin{align*}
\big| \mu(\AL) - \mu_L(\AL) \big|
\leq 4C (2K_L+1) \big( \rho_L^{-1} e^{\theta R_L} e^{-cL^\beta} + \rho_L^\kappa \big). 
\end{align*}
From \eqref{sk:eq-parameters-collection} and \eqref{sk:eq-prob-rho}, it follows that
\begin{align*}
4C (2K_L+1) \big( \rho_L^{-1} e^{\theta R_L} e^{-cL^\beta} + \rho_L^\kappa \big) 
\leq 8C L^{\frac{4\eta}{\alpha}} \Big( e^{\theta L^\eta} e^{\frac{1}{2}cL^\beta} e^{-cL^\beta} 
+ e^{-\frac{1}{2} c\kappa L^{\beta}} \Big) \leq e^{-\frac{1}{4}c\kappa L^\beta}, 
\end{align*}
and we therefore obtain the desired estimate \eqref{sk:eq-prob-AL-difference}.
It remains to use \eqref{sk:eq-prob-AL-difference} to prove \eqref{sk:eq-prob-AL}, \eqref{sk:eq-prob-BL}, and \eqref{sk:eq-probability-BL-upper}. In order to see \eqref{sk:eq-prob-AL}, we let $z^L \in \Zc^L_g$ be arbitrary. Using \eqref{sk:eq-parameters-collection}, \eqref{sk:eq-ZLgood}, \eqref{sk:eq-prob-AL-difference}, we obtain that 
\begin{align*}
\frac{\mu_L(\AL)}{\mu(\AL)}\geq 1 - \frac{\big|\mu_L(\AL)-\mu(\AL)\big|}{\mu(\AL)} \geq 1 - \frac{1}{\widetilde{\varepsilon}_L} e^{-\frac{1}{4}c\kappa L^\beta}  \geq 1- \varepsilon_L, 
\end{align*}
which proves \eqref{sk:eq-prob-AL}. To obtain \eqref{sk:eq-prob-BL}, we first use \eqref{sk:eq-zL-cardinality}, \eqref{sk:eq-good-and-bad}, and \eqref{sk:eq-prob-AL-difference}, which yield that 
\begin{align*}
&\, \big| \mu\big(B^L\big) - \mu_L\big(B^L\big) \big| 
= \big| \mu\big(G^L\big) - \mu_L\big(G^L\big) \big| 
\leq  \sum_{z^L \in \Zc^L_g} \big| \mu\big( \AL \big) - \mu_L \big( \AL \big) \big| \\[-0.5ex]
\leq&\, \big( \# \Zc^L \big) e^{-\frac{1}{4} c\kappa L^\beta} \leq e^{-\frac{1}{8} c\kappa L^\beta}. 
\end{align*}
From \eqref{sk:eq-parameter-ML}, we also have that $\mu(B^L)\geq \widetilde{\varepsilon}_L$.  Together with \eqref{sk:eq-parameters-collection}, it therefore follows that 
\begin{equation*}
\frac{\mu_L(B^L)}{\mu(B^L)}\geq 1 - \frac{\big|\mu_L(B^L)-\mu(B^L)\big|}{\mu(B^L)} \geq 1 - \frac{1}{\widetilde{\varepsilon}_L} e^{-\frac{1}{8} c\kappa L^\beta} \geq 1-\varepsilon_L, 
\end{equation*}
which proves \eqref{sk:eq-prob-BL}. Finally, from \eqref{sk:eq-parameter-ML}, \eqref{sk:eq-zL-cardinality}, and \eqref{sk:eq-good-and-bad}, it follows that 
\begin{equation*}
\mu\big( B^L \big) \leq \widetilde{\varepsilon}_L +\sum_{z^L \in \Zc^L\backslash \Zc^L_g} \mu\big( \AL \big)
\leq \big(1+\# \Zc^L \big) \widetilde{\varepsilon}_L \leq e^{-\frac{1}{32} c\kappa L^\beta},
\end{equation*}
which yields \eqref{sk:eq-probability-BL-upper}.\\

\noindent\emph{Step 3: A collection of random variables.} Given any probability measure $\nu$ on $(S,\mathcal{B}(S))$, where $S$ is a metric space and $\mathcal{B}(S)$ is the Borel $\sigma$-algebra, one can always find\footnote{For example, one can simply choose the probability space as $(S,\mathcal{B}(S),\nu)$ and define the random variable as the identity map on $S$.} a probability space supporting an $S$-valued random variable whose law is given by $\nu$. By passing to infinite product spaces, one can therefore find a probability space $(\Omega,\mathcal{F},\bP)$ supporting random variables $\phi$, $(\phi_{L,z^L})_{L\in \dyadic,z^l\in \Zc^L}$, $(\phi_{L,b})_{L\in \dyadic}$, $(\psi_L)_{L\in \dyadic}$, and $U$, all independent of one another, such that the following properties are satisfied: 
\begin{enumerate}[label=(\roman*)]
\item It holds that $\Law_\bP(\phi)=\mu$.
\item For all $L \in\dyadic$ and $z^L \in \Zc^L_g$, it holds that 
$\Law_\bP(\phi_{L,z^L})=\mu_L(\, \cdot \,| \AL )$, where the measure is obtained by  conditioning  $\mu_L$ on $\AL$.  Due to \eqref{sk:eq-prob-AL}, this is well-defined. 
\item For all $L\in \dyadic$, it holds that $\Law_\bP(\phi_{L,b})=\mu_L(\cdot|B^L)$. Due to \eqref{sk:eq-prob-BL}, this is well-defined.
\item For all $L\in \dyadic$, it holds that $\Law_\bP(\psi_L)=\nu_L$, where $\nu_L$ is defined as 
\begin{equation}\label{sk:eq-psi-L}
\begin{aligned}
\nu_L(A) &= \varepsilon_L^{-1} \sum_{z^L\in \Zc^L_g}
\mu_L ( A | \AL ) \big( \mu_L(\AL) - (1-\varepsilon_L) \mu(\AL)\big) \\
&+ \varepsilon_L^{-1} \mu_L( A| B^L) 
\big( \mu_L(B^L) - (1-\varepsilon_L) \mu(B^L)\big).
\end{aligned}
\end{equation}
Using \eqref{sk:eq-prob-AL}, \eqref{sk:eq-prob-BL}, and that $(\AL)_{z^L\in \Zc^L_g}$ and $B^L$ form a partition of $\CE^\theta(\R)$, one can easily check that $\nu_L$ is a probability measure.
\item The random variable $U$ is uniformly distributed on $[0,1]$.
\end{enumerate}
Due to the above, the common probability space $(\Omega,\mathcal{F},\bP)$ and the random variable $\phi=\phi_\infty$ from the statement of this proposition have now been defined, and it remains to define the random variables $\phi_L$, where $L\in\dyadic$.\\

\noindent\emph{Step 4: The coupling.} For all $L\in \dyadic$, we define
\begin{equation}\label{sk:eq-phi-L}
\phi_L := \sum_{z^L \in \Zc^L_g} \ind\big\{ \phi \in \AL, U \leq 1 - \varepsilon_L \big\} \phi_{L,z^L} 
+ \ind\big\{ \phi \in B^L, U\leq 1 -\varepsilon_L \big\} \phi_{L,b}
+ \ind\big\{ U > 1 - \varepsilon_L \big\} \psi_L .
\end{equation}
We now show that $\Law_\bP(\phi_L)=\mu_L$. To this end, we first 
note that the supports of the indicator functions 
\begin{equation*}
\ind\big\{ \phi \in \AL, U \leq 1 - \varepsilon_L \big\}, \qquad \ind\big\{ \phi \in B^L, U\leq 1 -\varepsilon_L \big\}, \qquad \text{and} \qquad 
\ind\big\{ U > 1 - \varepsilon_L \big\}
\end{equation*}
are disjoint. By using the independence and laws of the random variables from the previous step, we then obtain for all Borel sets $A\subseteq \CE^\theta$ that 
\begin{align*}
\bP(\phi_L \in A) 
&= \sum_{z^L\in \Zc^L_g} \bP \big( U \leq 1-\varepsilon_L, \, \phi\in \AL, \, \phi_{L,z^L} \in A \big) 
+ \bP \big( U \leq 1-\varepsilon_L, \, \phi\in B^L, \, \phi_{L,b} \in A \big) \\
&\hspace{4.5ex}+ \bP\big( U>1-\varepsilon_L, \, \psi_L \in A \big) \\[1.5ex]
&= (1-\varepsilon_L) \sum_{z^L\in \Zc^L_g} \mu\big(\AL\big) \mu_L\big( A \big| \AL \big) + (1-\varepsilon_L) \mu\big(B^L \big) \mu_L\big( A \big| B^L \big) + \varepsilon_L \nu_L(A) 
=\mu_L(A),
\end{align*}
where the last identity follows directly from the definition of $\nu_L$.\\

\noindent \emph{Step 5: Estimating the difference of $\phi$ and $\phi_L$.}
Using the definition of the grid points from \eqref{sk:eq-parameters-delta}, we obtain that
\begin{equation*}
\big\| \phi - \phi_L\big\|_{C^0([-R_L,R_L])}
\leq \max_{-K_L\leq k \leq K_L} \big| \phi(x^L_k) - \phi_L(x^L_k) \big|
+ \delta_L^{\alpha} \big( \big\| \phi \big\|_{C^\alpha([-R_L,R_L]}
+ \big\| \phi_L \big\|_{C^\alpha([-R_L,R_L]} \big).
\end{equation*}
As a result, we obtain that 
\begin{align}
&\, \bP \Big( \big\| \phi - \phi_L\big\|_{C^0([-R_L,R_L])} > L^{-\eta} \Big) \notag \\ 
\leq&\, \bP\Big( \max_{-K_L \leq k \leq K_L} \big| \phi(x^L_k) - \phi_L(x^L_k) \big|>\tfrac{1}{3} L^{-\eta} \Big) \label{sk:eq-as-1} \\
+&\,\bP \Big(  \| \phi \|_{C^\alpha([-R_L,R_L])} >\tfrac{1}{3} \delta_L^{-\alpha} L^{-\eta}\Big)
+ \bP \Big(  \| \phi_L \|_{C^\alpha([-R_L,R_L])} >\tfrac{1}{3} \delta_L^{-\alpha} L^{-\eta}\Big).\label{sk:eq-as-2}
\end{align}
It follows from the definition of $\phi_L$ that, on the event $\{ \phi \in G^L \} \medcap \{ U \leq 1-\varepsilon_L\}$, there exists a $z^L\in \Zc^L_g$ such that $\phi,\phi_L\in \AL$. On the event $\{ \phi \in G^L \} \medcap \{ U \leq 1-\varepsilon_L\}$, it then follows from the definition of $\AL$ that
\begin{equation*}
 \max_{-K_L \leq k \leq K_L} \big| \phi(x^L_k) - \phi_L(x^L_k) \big| \leq 2 \tau_L. 
\end{equation*}
Since $6\tau_L \leq L^{-\eta}$, we then obtain that
\begin{equation}\label{sk:eq-as-3}
\eqref{sk:eq-as-1} \leq \bP \big(  \phi\not\in G^L  \big) + \bP\big(  U > 1-\varepsilon_L \big) \leq e^{-\frac{1}{32} c\kappa L^\beta} + \varepsilon_L \leq \frac{1}{2} C^\prime e^{-c^\prime L^\eta},
\end{equation}
where we also used \eqref{sk:eq-parameters-collection} and \eqref{sk:eq-probability-BL-upper}. Using  \eqref{sk:eq-hoelder}, \eqref{sk:eq-parameters-collection}, and \eqref{sk:eq-parameters-delta},  we also obtain that 
\begin{equation}\label{sk:eq-as-4}
\begin{aligned}
\eqref{sk:eq-as-1} &\leq \bP \Big(  \| \phi \|_{C^\alpha([-R_L,R_L])} > L^\eta \Big)
+ \bP \Big(  \| \phi_L \|_{C^\alpha([-R_L,R_L])} > L^{\eta} \Big) \\ 
&\leq 2 C \exp\Big( - c \big( C^{-2} L^{2\eta} - \log(L^\eta) \big) \Big) \leq \tfrac{1}{2} C^\prime e^{-c^\prime L^\eta}.
\end{aligned}
\end{equation}
By combining \eqref{sk:eq-as-1}-\eqref{sk:eq-as-4}, we then obtain the desired estimate \eqref{sk:eq-final-estimate}.
\end{proof}
\end{appendix}

\bibliography{BS_Library}
\bibliographystyle{myalpha}

\end{document}